\documentclass[a4paper]{amsart}

\usepackage{amsmath,amstext,amssymb,mathrsfs,amscd,amsthm,indentfirst}
\usepackage{amsfonts}

\newcommand\mnote[1]{\marginpar{\tiny #1}}

\usepackage{booktabs}
\usepackage{verbatim}
\usepackage{enumerate}
\usepackage{microtype}
\usepackage[british]{babel}

\usepackage{graphicx}

\numberwithin{equation}{section}

\newcommand{\bN}{\mathbb{N}}

\newcommand{\bQ}{\mathbb{Q}}
\newcommand{\bR}{\mathbb{R}}

\newcommand{\bZ}{\mathbb{Z}}

\newcommand{\sM}{\mathsf{M}}
\newcommand{\sN}{\mathsf{N}}
\newcommand{\sH}{\mathsf{H}}
\newcommand{\sI}{\mathsf{I}}

\newcommand\lra{\longrightarrow}

\newcommand\Diff{\mathrm{Diff}}
\newcommand\Emb{\mathrm{Emb}}

\newcommand\Bun{\mathrm{Bun}}

\newcommand\colim{\operatorname*{colim}}
\newcommand\hocolim{\operatorname*{hocolim}}

\newcommand\Ker{\operatorname*{Ker}}

\newcommand{\hcoker}{/\!\!/}

\newcommand{\Hmfld}{H}
\newcommand{\Mst}{\mathcal{M}^{\mathrm{st}}}

\newcommand{\Lk}{\mathrm{Lk}}
\newcommand{\Fr}{\mathrm{Fr}}
\newcommand{\R}{\bR}

\newcommand{\Int}{\mathrm{int}}

\newcommand{\Hom}{\mathrm{Hom}}

\newcommand{\Aut}{\mathrm{Aut}}

\renewcommand{\epsilon}{\varepsilon}

\newcommand{\MM}{\mathscr{M}}

\newcommand{\nin}{\not\in}
\newcommand{\hatfr}{^\mathrm{fr}}
\newcommand{\lCM}{\mathrm{lCM}}
\newcommand{\wCM}{\mathrm{wCM}}
\newcommand{\CircNum}[1]{\ooalign{\hfil\raise .00ex\hbox{\scriptsize #1}\hfil\crcr\mathhexbox20D}}

\mathchardef\ordinarycolon\mathcode`\:
\mathcode`\:=\string"8000
\begingroup \catcode`\:=\active
  \gdef:{\mathrel{\mathop\ordinarycolon}}
\endgroup

\usepackage[all]{xy}
\SelectTips{cm}{10}
\CompileMatrices

\usepackage{amsthm}
\theoremstyle{plain}

\newtheorem{theorem}{Theorem}[section]

\newtheorem{proposition}[theorem]{Proposition}
\newtheorem{lemma}[theorem]{Lemma}

\newtheorem{corollary}[theorem]{Corollary}

\theoremstyle{definition}
\newtheorem{definition}[theorem]{Definition}

\theoremstyle{remark}

\newtheorem{remark}[theorem]{Remark}
\newtheorem*{remark*}{Remark}

\renewcommand{\#}{\sharp}

\title[Homological stability I]{Homological stability for moduli spaces of high dimensional manifolds. I}

\author{S{\o}ren Galatius} 
\email{galatius@stanford.edu}
\address{Department of Mathematics\\
  Stanford University\\
  Stanford CA, 94305}

\author{Oscar Randal-Williams}
\email{o.randal-williams@dpmms.cam.ac.uk}
\address{Centre for Mathematical Sciences\\
Wilberforce Road\\
Cambridge CB3 0WB\\
UK}

\dedicatory{Dedicated to Ulrike Tillmann}


\begin{document}
\begin{abstract}
  We prove a homological stability theorem for moduli spaces of simply-connected manifolds of dimension $2n > 4$, with respect to forming
  connected sum with $S^n \times S^n$. This is analogous to Harer's
  stability theorem for the homology of mapping class groups. Combined
  with previous work of the authors, it gives a calculation of the
  homology of the moduli spaces of manifolds diffeomorphic to
  connected sums of $S^n \times S^n$ in a range of degrees.
\end{abstract}
\maketitle

\section{Introduction and statement of results}\label{sec:intr-stat-results}

A famous result of Harer (\cite{H}) established \emph{homological
  stability} for mapping class groups of oriented surfaces.  For
example, if $\Gamma_{g,1}$ denotes the group of isotopy classes of
diffeomorphisms of an oriented connected surface of genus $g$ with one
boundary component, then the natural homomorphism $\Gamma_{g,1} \to
\Gamma_{g+1,1}$, given by gluing on a genus one surface with two
boundary components, induces an isomorphism in group homology $H_k(\Gamma_{g,1}) \to H_k(\Gamma_{g+1,1})$ as long as $k \leq (2g-2)/3$.  (Harer proved this for $k \leq (g+1)/3$, but the range was later improved by Ivanov (\cite{Ivanov}) and Boldsen (\cite{Boldsen}), see also \cite{R-WResolution}.)  This result can be interpreted in terms of moduli spaces of Riemann surfaces, and has lead to a wealth of research in topology and algebraic geometry.  In this paper we will prove an analogous homological stability result for moduli spaces of manifolds of higher (even) dimension.

\begin{definition}
  For a compact smooth
  manifold $W$, let $\Diff_\partial(W)$ denote the topological group of
  diffeomorphisms of $W$ restricting to the identity near its boundary.  The moduli space of manifolds of type $W$ is defined as the classifying space
  $\MM(W) = B\Diff_\partial(W)$.
\end{definition}

If we are given another compact smooth manifold $W'$ and a codimension zero embedding
$W \hookrightarrow W'\setminus \partial W'$ then we obtain a continuous
homomorphism $\Diff_\partial (W) \to \Diff_\partial (W')$ by
extending diffeomorphisms of $W$ by the identity diffeomorphism on the
cobordism $K = W' \setminus \Int(W)$.  The induced map of classifying
spaces shall be denoted
\begin{equation}\label{eq:GluingMap}
- \cup K : \MM(W) \lra \MM(W \cup_{\partial W} K).
\end{equation}
We shall give point-set models for $\MM(W)$ and the map~(\ref{eq:GluingMap}) in
Section \ref{sec:Resolutions}.

When $W$ is an orientable surface of genus $g$ with one boundary
component, and $W' = W \cup_{\partial W} K$ is also orientable of genus
$g+1$ with one boundary component, it can be shown that the
map~\eqref{eq:GluingMap} is equivalent to the map studied by
Harer, and hence it induces an isomorphism on homology in a range of degrees which increases with the genus of the surface.  Our main result is analogous to this, but for simply-connected manifolds of higher even dimension (although we exclude the case $2n=4$ for the usual reason: we shall need to use the Whitney trick).  We must first describe the analogue of genus which we will use.

In each dimension $2n$ we define manifolds
\begin{equation*}
  W_{g,1} = W_{g,1}^{2n} = D^{2n} \# g(S^n \times S^n),
\end{equation*}
the connected sum of $g$ copies of $S^n \times S^n$ with an open disc
removed, and if $W$ is a compact path-connected $2n$-manifold we define the number
 \begin{equation*}
   g(W) = \max\{g \in \bN \,\,\vert\,\, \text{there exists an embedding $W_{g,1} \hookrightarrow W$}\},
 \end{equation*}
which we call the \emph{genus} of $W$. Let $S$ be a manifold obtained by forming the connected sum of $[0,1]
\times \partial W$ with $S^n \times S^n$.  The corresponding gluing
map shall be denoted
\begin{equation*}
  s = - \cup S: \MM(W) \lra \MM(W \cup_{\partial W} S).
\end{equation*}
(If $\partial W$ is not path-connected then the diffeomorphism type of $S$ relative to $\{0\} \times \partial W$,
and hence the homotopy class of $s$, will depend on which path
component the connected sum is formed in.  The following theorem holds
for any such choice.)
\begin{theorem}\label{thm:main}
  For a simply-connected manifold $W$ of dimension $2n \geq 6$, the
  stabilisation map
\begin{equation*}
  s_*: H_k(\MM(W)) \lra H_k(\MM(W \cup_{\partial W} S))
\end{equation*}
is an isomorphism if $2k \leq g(W)-3$ and an epimorphism if $2k \leq g(W)-1$.
\end{theorem}

Our methods are similar to those used to prove many homological
stability results for homology of discrete groups, namely to use a
suitable action of the group on a simplicial complex.  For example,
Harer used the action of the mapping class group on the \emph{arc complex} to prove his homological stability result.  In our case the
relevant groups are not discrete, so we use a simplicial space
instead.

\subsection{Tangential structures and abelian coefficients}

In Section \ref{sec:TS} we shall generalise Theorem~\ref{thm:main} in
two directions. Firstly, we shall establish a version of
Theorem~\ref{thm:main} where $\MM(W)$ is replaced by a space of
manifolds equipped with certain tangential structures; secondly, we shall allow certain non-trivial systems of local coefficients.

Let $\gamma_{2n} \to BO(2n)$ denote the universal vector bundle.  
A \emph{tangential structure} for $2n$-dimensional manifolds is a fibration $\theta: B \to
BO(2n)$ with $B$ path-connected, classifying a vector bundle $\theta^*\gamma_{2n}$ over $B$. Examples include $B = EO(2n)/G \to BO(2n)$ for $G = SO(2n)$, corresponding to an orientation, $G = U(n)$, corresponding to an almost complex structure, or $G = \{1\}$, corresponding to a framing. A $\theta$-structure on a $2n$-dimensional manifold is a map of vector bundles $\hat{\ell}:TW \to \theta^*\gamma_{2n}$, and we then write $\ell : W \to B$ for the underlying map of spaces. If $\hat{\ell}_{\partial W} : TW\vert_{\partial W} \to \theta^*\gamma_{2n}$ is a given $\theta$-structure, then we write $\Bun_\partial^\theta(W;\hat{\ell}_{\partial W})$ for the space of bundle maps extending $\hat{\ell}_{\partial W}$. The group $\Diff_\partial(W)$ acts on this space, and we write
$$\MM^\theta(W;\hat{\ell}_{\partial W}) = \Bun_\partial^\theta(W;\hat{\ell}_{\partial W}) \hcoker \Diff_\partial(W)$$
for the Borel construction.  This space need not be path-connected, but if $\hat{\ell}_W \in \Bun_\partial^\theta(W;\hat{\ell}_{\partial W})$, then we write $\MM^\theta(W,\hat{\ell}_{W})$ for the path component of $\hat{\ell}_W$.

Our generalisation of Theorem~\ref{thm:main} to manifolds with tangential structures will replace the spaces $\MM(W)$ by $\MM^\theta(W,\hat{\ell}_{W})$. Before stating the theorem, we must explain the corresponding notion of genus and the analogue of the cobordism $S$.

\begin{definition}\label{defn:admissible}
Let us say that a $\theta$-structure $\hat{\ell} : TW_{1,1} \to \theta^*\gamma_{2n}$ is \emph{admissible} if there is a pair of orientation-preserving embeddings $e, f : S^n \times D^n \hookrightarrow W_{1,1} \subset S^n \times S^n$ whose cores $e(S^n \times \{0\})$ and $f(S^n \times \{0\})$ intersect transversely in a single point, such that each of the $\theta$-structures $e^*\hat{\ell}$ and $f^*\hat{\ell}$ on $S^n \times D^n$ extend to $\bR^{2n}$ for some orientation-preserving embeddings $S^n \times D^n \hookrightarrow \bR^{2n}$.  We then define the $\theta$-genus of a compact path-connected $\theta$-manifold $(W, \hat{\ell}_W)$ to be
\begin{equation*}
  g^\theta(W, \hat{\ell}_W) = \max\left\{g \in \bN \,\,\bigg|\,\, \parbox{20em}{there are $g$ disjoint copies of $W_{1,1}$ in $(W, \hat{\ell}_W)$,\\each with admissible $\theta$-structure}\right\}.
\end{equation*}
\end{definition}  

If $W$ contains a copy of $W_{g,1}$ such that $\ell_W\vert_{W_{g,1}} : W_{g,1} \to B$ is nullhomotopic, then in Proposition \ref{prop:reframing} we show that $g^\theta(W, \hat{\ell}_W) \geq g-1$, and if $n \neq 3,7$ this can be strengthened to $g^\theta(W, \hat{\ell}_W) \geq g$.  When $B$ is simply-connected, the number $g^\theta(W,\hat\ell_W)$ may be estimated in terms of characteristic numbers, with a constant error term depending only on $n$ and $H_n(B;\bZ)$, cf.\ Remark~\ref{remark:estimating-genus}.

In order to define the stabilisation map, we say that a $\theta$-structure $\hat{\ell}_S$ on $S$ is admissible if it is admissible in the sense above when restricted to $W_{1,1} \subset S$. Suppose furthermore that it restricts to $\hat{\ell}_{\partial W}$ on $\{0\} \times \partial W$, and write $\hat{\ell}_{\partial W}^\prime$ for its restriction to $\{1\} \times \partial W$. Then there is an induced stabilisation map
$$ s = - \cup (S, \hat{\ell}_S) : \MM^\theta(W,\hat{\ell}_{W}) \lra \MM^\theta(W \cup_{\partial W} S,\hat{\ell}_{W} \cup \hat{\ell}_S)$$
given by gluing on $S$ to $W$ and extending $\theta$-structures using $\hat{\ell}_S$.

We require two additional terms to describe our result. We say that $\theta$ is \emph{spherical} if $S^{2n}$ admits a $\theta$-structure, and we say that a local coefficient
system is \emph{abelian} if it has trivial monodromy along all nullhomologous loops.
  
\begin{theorem}\label{thm:mainTheta}
  For a simply-connected manifold $W$ of dimension $2n \geq 6$, a $\theta$-structure $\hat{\ell}_W$ on $W$, an admissible $\theta$-structure $\hat{\ell}_S$ on $S$,
and an abelian local coefficient system $\mathcal{L}$ on $\MM^\theta(W \cup_{\partial W} S,\hat{\ell}_{W \cup S})$, the stabilisation map
\begin{equation*}
  s_*: H_k(\MM^\theta(W,\hat{\ell}_{W});s^*\mathcal{L}) \lra H_k(\MM^\theta(W \cup_{\partial W} S,\hat{\ell}_{ W \cup S});\mathcal{L})
\end{equation*}
is 
\begin{enumerate}[(i)]
\item an epimorphism for $3k \leq g^\theta(W, \hat{\ell}_W)-1$ and an isomorphism for $3k \leq g^\theta(W, \hat{\ell}_W)-4$,

\item an epimorphism for $2k \leq g^\theta(W, \hat{\ell}_W)-1$ and an isomorphism for $2k \leq g^\theta(W, \hat{\ell}_W)-3$, if $\theta$ is spherical and $\mathcal{L}$ is constant.
\end{enumerate}
\end{theorem}

For example, consider the tangential structure $\theta : BU(3) \to BO(6)$. If $(W, \hat{\ell}_W)$ is an almost complex 6-manifold (with non-empty boundary), and $e : W_{g,1} \to W$ is an embedding, then $\ell_W \circ e : W_{g,1} \to BU(3)$ is nullhomotopic because $W_{g,1} \simeq \vee^{2g} S^3$ and $\pi_3(BU(3))=0$. Thus $g^\theta(W, \hat{\ell}_W) \geq g-1$. Furthermore, $S^6$ admits an almost complex structure so $\theta$ is spherical. So for any admissible
$\theta$-structure $\hat{\ell}_S$ on $S$, the stabilisation map
$$s : \MM^\theta(W,\hat{\ell}_{W}) \lra \MM^\theta(W \cup_{\partial W} S,\hat{\ell}_{ W \cup S})$$
induces an isomorphism on integral homology in degrees up to $\tfrac{g-4}{2}$.
  
In the sequel \cite{HomStabII} to this paper we prove an
analogue of Theorem~\ref{thm:mainTheta} where the manifold $S$ is replaced
by a more general cobordism $K$, satisfying that $(K,\partial W)$ is
$(n-1)$-connected.  The theorem proved there includes the case
where $W \cup_{\partial W} K$ is a closed manifold.


\subsection{Stable homology}

If we have a sequence $\hat{\ell}_{S_1}, \hat{\ell}_{S_2}, \ldots$ of admissible $\theta$-structures on $S$ such that $\hat{\ell}_{S_1}\vert_{\{0\} \times \partial W} = \hat{\ell}_W\vert_{\partial W}$ and $\hat{\ell}_{S_i}\vert_{\{1\} \times \partial W} = \hat{\ell}_{S_{i+1}}\vert_{\{0\} \times \partial W}$ for all $i$, then the manifold $W \cup gS$ given by the composition of $W$ and $g$ copies of the cobordism $S$ has a $\theta$-structure $\hat{\ell}_{W \cup gS} = \hat{\ell}_W \cup \hat{\ell}_{S_1} \cup \cdots \cup \hat{\ell}_{S_g}$, and there are maps
$$- \cup (S, \hat{\ell}_{S_{g+1}}) : \MM^\theta(W \cup gS, \hat{\ell}_{W \cup gS}) \lra \MM^\theta(W \cup (g+1)S, \hat{\ell}_{W \cup (g+1)S}).$$
In this situation the homology of the limiting space
$$\hocolim_{g \to \infty} \MM^\theta(W \cup gS, \hat{\ell}_{W \cup gS})$$
can be described in homotopy-theoretic terms for any $W$ and any $\theta$, from which explicit calculations are quite feasible. In many cases this description is given in \cite{GR-W2}, and we shall describe the general case in \cite[Sections 1.2 and 7]{HomStabII}. Here we shall focus on the interesting special case $W = D^{2n}$ and $\theta = \mathrm{Id}_{BO(2n)}$, in which case $\MM^\theta(W \cup gS , \hat{\ell}_{W \cup gS}) = \MM(W_{g,1}) = B\Diff_\partial(W_{g,1})$.

The boundary of $W_{g,1}$ is a sphere, so $S = ([0,1] \times S^{2n-1}) \# (S^n \times S^n)$ and
hence there is a diffeomorphism $W_{g,1} \cup_{\partial W_{g,1}} S \approx W_{g+1,1}$ relative to their already identified boundaries. Theorem \ref{thm:main} (or Theorem \ref{thm:mainTheta} for abelian coefficients) therefore implies the following.

\begin{corollary}\label{cor:main}
  For $2n \geq 6$ and an abelian coefficient system $\mathcal{L}$ on $\MM(W_{g+1,1})$, the stabilisation map 
  \begin{equation*}
    s_* : H_k(\MM(W_{g,1});s^*\mathcal{L}) \lra H_k(\MM(W_{g+1,1});\mathcal{L})
  \end{equation*}
  is an epimorphism for $3k \leq g-1$ and an isomorphism for $3k \leq g-4$. If $\mathcal{L}$ is constant, then it is an epimorphism for $2k \leq g-1$ and an isomorphism for $2k \leq g-3$.
\end{corollary}

It is an immediate consequence of this corollary and \cite[\S 3]{R-WPerfect} that the mapping class group $\pi_1(\MM(W_{g,1}))$ has perfect commutator subgroup as long as $g \geq 7$.

\begin{remark}
  Independently, Berglund and Madsen (\cite{BerglundMadsen}) have
  obtained a result similar to this corollary, for rational cohomology in
  the range $2k \leq \min(2n-6, g-6)$.
\end{remark}
\begin{remark}
  In the earlier preprint \cite{OldVersion} we considered only the
  manifolds $W_{g,1}$, rather than the more general manifolds of
  Theorem~\ref{thm:main}.  Although the present paper entirely
  subsumes \cite{OldVersion}, the reader mainly interested in
  Corollary~\ref{cor:main} may want to consult the preprint for a streamlined text adapted to that special case (at least if they are only concerned with constant coefficients).
\end{remark}

By the universal coefficient theorem, stability for homology implies
stability for cohomology; in the surface case, Mumford
(\cite{Mumford}) conjectured an explicit formula for the stable
rational cohomology, which in our notation asserts that a certain ring
homomorphism
\begin{equation*}
  \bQ[\kappa_1, \kappa_2, \dots] \lra H^*(\MM(W^2_{g,1});\bQ)
\end{equation*}
is an isomorphism for $g \gg *$.  Mumford's conjecture was proved in a
strengthened form by Madsen and Weiss (\cite{MW}).

Corollary~\ref{cor:main} and our previous paper \cite{GR-W2} allow us
to prove results analogous to Mumford's conjecture and the
Madsen--Weiss theorem for the moduli spaces $\MM(W_{g,1})$ with
$2n \geq 6$. The analogue of the Madsen--Weiss theorem for these
spaces concerns the homology of the limiting space 
$\MM(W_\infty) = \hocolim_{g \to \infty} \MM(W_{g,1})$.
 There is a certain
infinite loop space $\Omega^\infty MT\theta^n$ associated to the tangential structure $\theta^n : BO(2n)\langle n \rangle \to BO(2n)$ given by the $n$-connected cover,
and a map
$$\alpha : \MM(W_\infty) \lra \Omega^\infty MT\theta^n$$
given by a parametrised form of the Pontryagin--Thom construction, and in \cite[Theorem 1.1]{GR-W2} we proved that $\alpha$ induces an isomorphism between the homology of $\MM(W_\infty)$ and the homology of the basepoint component of $\Omega^\infty MT\theta^n$. In \cite{GR-WAb} we used this to compute $H_1(\MM(W_{g,1});\bZ)$ for $g \geq 5$. 

It is easy to calculate the rational cohomology ring of a component of $\Omega^\infty MT\theta^n$, and hence of $\MM(W_{g,1})$ in the range of degrees given by Corollary~\ref{cor:main}.  The result is Corollary~\ref{cor:higher-Mumford} below, which is a higher-dimensional analogue of Mumford's conjecture. Recall that for each $c \in
H^{k+2n}(BSO(2n))$ there is an associated cohomology class $\kappa_c \in
H^{k}(\Omega^\infty MT\theta^n)$. Pulling it back via $\alpha$ and
all the stabilisation maps $\MM(W_{g,1}) \to \MM(W_{\infty})$ defines
classes $\kappa_c \in H^k(\MM(W_{g,1}))$ for all $g$, sometimes called
``generalised Mumford--Morita--Miller classes''. These can also be described in terms of fibrewise integration, see e.g.\ \cite[\S 1.1]{GR-WDetect}. The following result is our higher-dimensional analogue of Mumford's conjecture.

\begin{corollary}\label{cor:higher-Mumford}
Let $2n \geq 6$ and let $\mathcal{B}\subset H^*(BSO(2n);\bQ)$ be the set of monomials in the classes $e, p_{n-1}, \dots, p_{\lceil \frac{n+1}4 \rceil}$, of degree greater than $2n$.  Then the induced map
  \begin{equation*}
    \bQ[\kappa_c \,\,|\,\, c \in \mathcal{B}] \lra H^*(\MM(W_{g,1});\bQ)
  \end{equation*}
is an isomorphism in degrees satisfying $2* \leq g-3$.
\end{corollary}

For example, if $2n = 6$, the set $\mathcal{B}$ consists of monomials in $e$, $p_1$ and $p_2$, and therefore $H^*(\MM(W_{g,1}^{6});\bQ)$ agrees for $\ast \leq \tfrac{g-3}{2}$ with a polynomial ring in variables of degrees 2, 2, 4, 6, 6, 6, 8, 8, 10, 10, 10, 10, 12, 12, \dots.

\subsection{Acknowledgements}

The authors would like to thank Nathan Perlmutter for bringing to our attention an oversight in an earlier proof of Lemma \ref{lemthm:conn-K-delta}, and Michael Weiss for useful comments on a draft of this paper.
S.~Galatius was partially supported by NSF grants DMS-1105058 and
DMS-1405001, and the European Research Council (ERC) under the European
Union's Horizon 2020 research and innovation programme (grant
agreement No  682922). O.~Randal-Williams was supported by EPSRC grant EP/M027783/1 and the Herchel Smith Fund, and both authors were supported by ERC Advanced Grant
No.~228082, and the Danish National Research Foundation through the
Centre for Symmetry and Deformation. 


\section{Techniques}\label{sec:techniques}

In this section we collect the technical results needed to establish
high connectivity of certain simplicial spaces which will be relevant for the proof of Theorem \ref{thm:main}.  The main results
are Theorem~\ref{thm:simplex-wise-injective} and
Corollary~\ref{cor:serre-microf-connectivity}.

\subsection{Cohen--Macaulay complexes}\label{sec:CM}

Recall from \cite[Definition 3.4]{HW} that a simplicial complex $X$ is called
\emph{weakly Cohen--Macaulay} of dimension $n$ if it is
$(n-1)$-connected and the link of any $p$-simplex is
$(n-p-2)$-connected.  In this case, we write $\wCM(X) \geq n$. We shall
also say that $X$ is \emph{locally weakly Cohen--Macaulay} of
dimension $n$ if the link of any $p$-simplex is $(n-p-2)$-connected
(but no global connectivity is required on $X$ itself).  In this case
we shall write $\lCM(X) \geq n$.

\begin{lemma}\label{lem:wcm-of-link}
  If $\lCM(X) \geq n$ and $\sigma < X$ is a $p$-simplex,
  then $\wCM(\Lk(\sigma)) \geq n-p-1$.
\end{lemma}
\begin{proof}
  By assumption, $\Lk(\sigma)$ is $((n-p-1)-1)$-connected.  If $\tau <
  \Lk(\sigma)$ is a $q$-simplex, then
  \begin{equation*}
    \Lk_{\Lk(\sigma)}(\tau) = \Lk_X(\sigma \ast \tau)
  \end{equation*}
  is $((n-p-1) - q - 2)$-connected, since $\sigma \ast \tau$ is a
  $(p+q+1)$-simplex, and hence its link in $X$ is
  $(n-(p+q+1)-2)$-connected.
\end{proof}

\begin{definition}
  Let us say that a simplicial map $f: X \to Y$ of simplicial
  complexes is \emph{simplexwise injective} if its restriction to each
  simplex of $X$ is injective, i.e.\ the image of any $p$-simplex of
  $X$ is a (non-degenerate) $p$-simplex of $Y$.
\end{definition}
\begin{lemma}
  Let $f: X \to Y$ be a simplicial map between simplicial complexes.  Then
  the following conditions are equivalent.
  \begin{enumerate}[(i)]
  \item\label{item:1} $f$ is simplexwise injective,
  \item\label{item:2} $f(\Lk(\sigma)) \subset \Lk(f(\sigma))$ for
    all simplices $\sigma < X$,
  \item\label{item:3} $f(\Lk(v)) \subset \Lk(f(v))$ for all vertices
    $v \in X$,
  \item \label{item:4} The image of any 1-simplex in $X$ is a
    (non-degenerate) 1-simplex in $Y$.
  \end{enumerate}
\end{lemma}
\begin{proof}\mbox{} 

  \noindent (\ref{item:1}) $\Rightarrow$~(\ref{item:2}). If $\sigma =
  \{v_0, \dots, v_p\}$ and $v \in \Lk(\sigma)$, then $\{v,v_0, \dots,
  v_p\} < X$ is a simplex, and therefore $\{f(v), f(v_0), \dots,
  f(v_p)\} < Y$ is a simplex.  Since $f$ is simplexwise injective, we
  must have $f(v) \not\in f(\sigma)$, so $f(v) \in \Lk(f(\sigma))$.

  \noindent (\ref{item:2}) $\Rightarrow$~(\ref{item:3}). Trivial.

  \noindent (\ref{item:3}) $\Rightarrow$~(\ref{item:4}).  If $\sigma =
  \{v_0, v_1\}$ is a 1-simplex, then $v_1 \in \Lk(v_0)$ so $f(v_1) \in
  \Lk(f(v_0))$, but then $\{f(v_0), f(v_1)\}$ is a 1-simplex.

  \noindent (\ref{item:4}) $\Rightarrow$~(\ref{item:1}). Let $\sigma =
  \{v_0, \dots, v_p\} < X$ be a $p$-simplex and assume for
  contradiction that $f\vert_\sigma$ is not injective.  This means that
  $f(v_i) = f(v_j)$ for some $i \neq j$, but then the restriction of
  $f$ to the 1-simplex $\{v_i, v_j\}$ is not injective.
\end{proof}

The following theorem generalises the ``colouring lemma'' of Hatcher
and Wahl (\cite[Lemma 3.1]{HW}), which is the special case where $X$ is
a simplex.  The proof given below is an adaptation of theirs.

\begin{theorem}\label{thm:simplex-wise-injective}
Let $X$ be a simplicial complex with $\lCM(X) \geq n$, $f: \partial I^n \to |X|$ be a map and $h : I^n \to \vert X \vert$ be a nullhomotopy. If $f$ is simplicial with
  respect to a PL triangulation $\partial I^n \approx |L|$, then this triangulation extends to a PL triangulation $I^n \approx |K|$ and $h$ is homotopic relative to $\partial I^n$ to a simplicial map $g : |K| \to |X|$ such that
\begin{enumerate}[(i)]
\item\label{it:simplex-wise-injective:1}  for each vertex $v \in K \setminus L$, the star $v *\Lk_K(v)$ intersects $L$ in a single (possibly empty) simplex, and

\item\label{it:simplex-wise-injective:2} for each vertex $v \in K \setminus L$, $g(\Lk_K(v)) \subset \Lk_X(g(v))$.
\end{enumerate}  
 In particular, $g$ is simplexwise
  injective if $f$ is.
\end{theorem}

\begin{proof}
  We proceed by
  induction on $n$, the case $n=0$ being clear.  By the simplicial
  approximation theorem we may change the map $h$ by a homotopy
  relative to the boundary, after which $h: I^n \to |X|$ is simplicial
  with respect to some PL triangulation $I^n \approx |K|$ extending
  the given triangulation on $\partial I^n$. After barycentrically
  subdividing $K$ relative to $L$ twice,
  (\ref{it:simplex-wise-injective:1}) is satisfied.
  
  Let us say that a simplex $\sigma < K$ is \emph{bad} if every vertex
  $v \in \sigma$ is contained in a 1-simplex $(v,v') \subset \sigma$
  with $h(v) = h(v')$.  If all bad simplices are contained in
  $\partial I^n$, we are done.  If not, let $\sigma < K$ be a bad
  simplex not contained in $\partial I^n$, of maximal dimension $p$.
  We may then write $I^n = A \cup B$ for subsets
  \begin{align*}
    A &= \sigma * \Lk_K(\sigma)\\
    B &= I^n \setminus \Int(A)\\
    A \cap B &= (\partial \sigma) * \Lk_K(\sigma),
  \end{align*}
  which are also subcomplexes with respect to the triangulation
  $I^n \approx |K|$.  We shall describe a procedure which changes the
  triangulation of $A$ as well as the map $h\vert_{A}$ in a way that
  strictly decreases the number of bad simplices of dimension $p$ and
  not contained in $\partial I^n$, creates no bad simplices of higher dimension, and preserves the
  property~(\ref{it:simplex-wise-injective:1}).  This will complete
  the proof, since we may eliminate all bad simplices not contained in
  $\partial I^n$ in finitely many steps.

  Badness of $\sigma$ implies $p > 0$, and we must also have
  $h(\Lk_K(\sigma)) \subset \Lk_X(h(\sigma))$, since otherwise we could
  join a vertex of $\Lk_K(\sigma)$ to $\sigma$ and get a bad simplex of
  strictly larger dimension.  Now $|\sigma| \subset |K| \approx I^n$,
  so $h$ restricts to a map
  \begin{equation*}
    \partial I^{n-p} \approx |\Lk_K(\sigma)| \lra |\Lk_X(h(\sigma))|.
  \end{equation*}
  The image $h(\sigma)$ is a simplex of dimension at most $p-1$, since
  otherwise $h|_\sigma$ would be injective (in fact $h(\sigma)$ must have dimension
  at most $(p-1)/2$ by badness).  Then $|\Lk_X(h(\sigma))|$ is
  $(n-(p-1)-2)$-connected since we assumed $\lCM(X) \geq n$, and in
  fact Lemma~\ref{lem:wcm-of-link} gives
  \begin{equation*}
    \wCM(\Lk_X(h(\sigma))) \geq n - (p-1) - 1 = n - p.
  \end{equation*}
  Therefore, $h|_{|\Lk_K(\sigma)|}$ extends to a map
  \begin{equation*}
    \tilde h : C(|\Lk_K(\sigma)|) \lra |\Lk_X(h(\sigma))|
  \end{equation*}
  and we may apply the induction hypothesis to this map. It follows that there is a PL triangulation $C(|\Lk_K(\sigma)|) \approx \vert \tilde{K} \vert$ satisfying (i) and extending the triangulation of $|\Lk_K(\sigma)|$, and a simplicial map $\tilde{g} : \tilde{K} \to \Lk_X(h(\sigma))$ extending $h\vert_{\Lk_K(\sigma)}$ and satisfying (ii). In
  particular, the star of each vertex in
  $\tilde{K} \setminus \Lk_K(\sigma)$ intersects $\Lk_K(\sigma)$ in a
  single simplex, and all bad simplices of $\tilde g$ are in
  $\Lk_K(\sigma) \subset \tilde{K}$ (but in fact there cannot be any, by maximality of $\sigma$).  The composition
  \begin{equation*}
    A = |\sigma \ast \Lk_K(\sigma)| \approx |(\partial \sigma) \ast (C
    \Lk_K(\sigma))| \xrightarrow{(h|_{\partial \sigma}) * \tilde g}
    |h(\partial \sigma) * \Lk_X(h(\sigma))| \subset |X|
  \end{equation*}
  agrees with $h$ on $\partial A = (\partial \sigma) * \Lk_K(\sigma)$
  and may therefore be glued to $h\vert_{B}$ to obtain a new map
  $h': I^n \to |X|$ which is simplicial with respect to
  a new triangulation $I^n \approx |K'|$ which on $B \subset I^n$ agrees
  with the old one and on $A$ comes from the triangulation
  $|\partial \sigma| \ast C(|\Lk_K(\sigma)|) \approx |(\partial \sigma) \ast \tilde{K}|$.  In particular all vertices of $K$ are also vertices of $K'$ and we call these the ``old'' vertices.

  We claim that this new triangulation of $I^n$ still
  satisfies~(\ref{it:simplex-wise-injective:1}).  There are two cases
  to consider depending on whether $v \in A \setminus \partial A$ or
  $v \in B$ and we first consider the case where
  $v \in A \setminus \partial A$.  Such a vertex is necessarily
  interior to the triangulation $C(|\Lk_K(\sigma)|) \approx |\tilde K|$
  and hence its star intersects $\Lk_K(\sigma)$ in a single simplex
  $\tau < \Lk_K(\sigma)$ and hence intersects
  $\partial A = A \cap B \approx |(\partial \sigma) * \Lk_K(\sigma)|$ in
  $|(\partial \sigma) * \tau|$.  The intersection of $|L|$ and
  $|\sigma * \Lk_K(\sigma)|$ is also a single simplex (since the
  intersection of $L$ and the star of a vertex of $\sigma$ is), of the form
  $\iota_0 * \iota_1 < (\partial \sigma) * \Lk_K(\sigma)$.  Hence
  the intersection of $L$ and the star of $v \in K'$ is
  $\iota_0 *(\iota_1 \cap \tau)$ which is again a
  simplex.

  In the second case where $v \in K'$ is in $B$, it was also a vertex
  in the old triangulation of $I^n$ by $K$, and its link in $K'$
  consists of old vertices which were in its link in $K$, along with
  some new vertices which are interior to $A$ and hence not in
  $\partial I^n$, but no old vertices which were not in the link of
  $v$ in $K$ (since that would violate
  condition~(\ref{it:simplex-wise-injective:1}) for the triangulation
  $C(|\Lk_K(\sigma)|) \approx |\tilde{K}|$).  Hence the star of $v \in K$ will
  intersect $L$ in a face of its old intersection, which is still a
  single simplex.
\end{proof}

\begin{proposition}\label{prop:epi}
  Let $X$ be a simplicial complex, and $Y \subset X$ be a full
  subcomplex.  Let $n$ be an integer with the property that for each
  $p$-simplex $\sigma < X$ having no vertex in $Y$, the complex $Y \cap \Lk_X(\sigma)$ is
  $(n-p-1)$-connected.  Then the inclusion $|Y| \hookrightarrow |X|$
  is $n$-connected.
\end{proposition}
\begin{proof}
  This is very similar to the proof of
  Theorem~\ref{thm:simplex-wise-injective}.  Let $k \leq n$ and
  consider a map $h: (I^k, \partial I^k) \to (|X|,|Y|)$ which is
  simplicial with respect to some PL triangulation of $I^k$.  Let
  $\sigma < I^k$ be a $p$-simplex such that $h(\sigma) \subset
  |X|\setminus |Y|$.  If $p$ is maximal with this property, we will
  have $h(\Lk(\sigma)) \subset |Y| \cap \Lk_X(h(\sigma))$, since
  otherwise we could make $p$ larger by joining $\sigma$ with a vertex
  $v \in \Lk(\sigma)$ such that $h(v) \nin Y$ or $h(v) \in
  h(\sigma)$.

  Now, $\Lk(\sigma) \approx S^{k-p-1}$ since $\sigma$ is a
  $p$-simplex, and $|Y| \cap \Lk_X(h(\sigma))$ is assumed
  $(n-p-1)$-connected, so $h\vert_{\Lk(\sigma)}$ extends over the cone
  of $\Lk(\sigma)$.  Then modify $h$ on the ball
  \begin{equation*}
    \sigma \ast \Lk(\sigma) \approx (\partial \sigma) \ast C \Lk(\sigma)
    \subset I^k
  \end{equation*}
  by replacing it with the join of $h \vert_{\partial \sigma}$ and
  some map $C \Lk(\sigma) \to |Y| \cap \Lk_X(h(\sigma))$ extending
  $h\vert_{\Lk(\sigma)}$.  As in the proof of
  Theorem~\ref{thm:simplex-wise-injective}, the modified map
  $(I^k, \partial I^k) \to (|X|,|Y|)$ is homotopic to the old one (on
  the ball where the modification takes place, both maps have image in
  the contractible set $h(\sigma) \ast \Lk_X(h(\sigma))$), and has
  strictly fewer $p$-simplices mapping to $|X|\setminus |Y|$.
\end{proof}

\subsection{Serre microfibrations}
\label{sec:serre-micr}

Let us recall from \cite{WeissCatClass} that a map $p: E \to B$ is called a
\emph{Serre microfibration} if for any $k$ and any commutative diagram
\begin{equation*}
  \xymatrix{
    \{0\} \times D^k \ar[r]^-f \ar[d] & E \ar[d]^p\\
    [0,1] \times D^k \ar[r]^-h & B}
\end{equation*}
there exists an $\epsilon > 0$ and a map $H: [0,\epsilon]\times D^k
\to E$ with $H(0,x) = f(x)$ and $p \circ H(t,x) = h(t,x)$ for all $x
\in D^k$ and $t \in [0,\epsilon]$.  

Examples of Serre microfibrations include Serre fibrations and submersions of manifolds; if $p: E \to B$ is a Serre microfibration then $p \vert_U: U \to B$ is too, for any open $U \subset E$.

The microfibration condition implies that if
$(X,A)$ is a finite CW pair then any map $X \to B$ may be lifted in a
neighbourhood of $A$, extending any prescribed lift over $A$. It also
implies the following useful observation: suppose $(Y,X)$ is a finite CW
pair and we are given a lifting problem
\begin{equation}\label{eq:1}
    \begin{aligned}
      \xymatrix{
        X \ar[r]^-f \ar[d] & E \ar[d]^p\\
        Y \ar[r]^-F & B.}
    \end{aligned}
\end{equation}
If there exists a map $G : Y \to E$ lifting $F$ and so that
$G\vert_X$ is fibrewise homotopic to $f$, then there is also a lift $H$ of
$F$ so that $H\vert_X=f$. To see this, choose a fibrewise homotopy
$\varphi : [0,1] \times X \to E$ from $G\vert_X$ to $f$, let $J =
([0,1] \times X) \cup (\{0\} \times Y) \subset [0,1] \times Y$ and
write $\varphi \cup G : J \to E$ for the map induced by $\varphi$ and
$G$. The following diagram is then commutative
  \begin{equation*}
    \xymatrix{
      J \ar[rr]^-{\varphi \cup G} \ar[d] & &  E \ar[d]^p\\
      [0,1] \times Y \ar[r]^-{\pi_Y}  & Y \ar[r]^-{F} & B, }
  \end{equation*}
and by the microfibration property there is a lift $g : U \to E$ defined on an open neighbourhood $U$ of $J$. Let $\phi : Y \to [0,1]$ be a continuous function with graph inside $U$ and so that $X \subset \phi^{-1}(1)$. Then we set $H(y) = g(\phi(y),y)$; this is a lift of $F$ as $g$ is a lift of $F \circ \pi_Y$, and if $y \in X$ then $\phi(y)=1$ and so $H(y) = g(1,y) = f(y)$, as required.

Weiss proved in \cite[Lemma 2.2]{WeissCatClass} that
if $f: E \to B$ is a Serre microfibration with weakly contractible
fibres (i.e.\ $f^{-1}(b)$ is weakly contractible for all $b \in B$),
then $f$ is in fact a Serre fibration and hence a weak homotopy equivalence.
We shall need the following generalisation, whose proof is essentially
the same as Weiss'.
\begin{proposition}\label{prop:Weiss-lemma}
  Let $p: E \to B$ be a Serre microfibration such that $p^{-1}(b)$ is
  $n$-connected for all $b \in B$.  Then the homotopy fibres of $p$
  are also $n$-connected, i.e.\ the map $p$ is $(n+1)$-connected.
\end{proposition}
\begin{proof}
  Let us first prove that $p^I: E^I \to B^I$ is a Serre microfibration
  with $(n-1)$-connected fibres, where $X^I = \mathrm{Map}([0,1],X)$
  is the space of (unbased) paths in $X$, equipped with the
  compact-open topology.  Using the mapping space adjunction, it is
  obvious that $p^I$ is a Serre microfibration, and showing the
  connectivity of its fibres amounts to proving that any diagram of
  the form
  \begin{equation*}
    \xymatrix{
       [0,1] \times \partial D^k \ar[rr] \ar[d] & & E\ar[d]^p\\
      [0,1] \times D^k \ar[r]^-{\text{proj}}& [0,1] \ar[r] & B }
  \end{equation*}
  with $k \leq n$ admits a diagonal $h: [0,1] \times D^k \to E$.
  Since fibres of $p$ are $(k-1)$-connected (in fact $k$-connected),
  such a diagonal can be found on each $\{a\} \times D^k$, and by the
  microfibration property these lifts extend to a neighbourhood.  By
  the Lebesgue number lemma we may therefore find an integer $N \gg 0$
  and lifts $h_i: [(i-1)/N,i/N]\times D^k \to E$ for $i = 1, \dots,
  N$.  The two restrictions $h_i, h_{i+1}: \{i/N\} \times D^k \to E$
  agree on $\{i/N\} \times \partial D^k$ and map into the same fibre
  of $p$.  Since these fibres are $k$-connected, the restrictions of
  $h_i$ and $h_{i+1}$ are homotopic relative to $\{i/N\}
  \times \partial D^k$ as maps into the fibre, and we may use
  diagram~\eqref{eq:1} with $Y = [i/N,(i+1)/N] \times D^k$ and $X =
  (\{i/N\}\times D^k) \cup ([i/N,(i+1)/N] \times \partial D^k)$ to
  inductively replace $h_{i+1}$ with a homotopy which can be
  concatenated with $h_i$.  The concatenation of the $h_i$'s then
  gives the required diagonal.

  Let us now prove that for all $k \leq n$, any lifting diagram
  \begin{align*}
    \begin{aligned}
      \xymatrix{
        \{0\} \times I^k \ar[r]^-f \ar[d] & E \ar[d]^p\\
        [0,1] \times I^k \ar@{..>}[ur]^H\ar[r]^-h & B }
    \end{aligned}
  \end{align*}
  admits a diagonal map $H$ making the diagram commutative.  To see
  this, we first use that fibres of the map $p^{I^{k+1}}: E^{I^{k+1}}
  \to B^{I^{k+1}}$ are non-empty (in fact $(n-k-1)$-connected) to find
  a diagonal $G$ making the lower triangle commute.  The restriction
  of $G$ to $\{0\} \times I^k$ need not agree with $f$, but they lie
  in the same fiber of $p^{I^k}: E^{I^k} \to B^{I^k}$.  Since this map
  has path-connected fibres, these are fibrewise homotopic, and hence
  we may apply~\eqref{eq:1} to replace $G$ with a lift $H$ making both
  triangles commute.

By a standard argument (see \cite[p.\ 375]{Spanier}) the above homotopy lifting property implies that $p: E \to B$ has the homotopy lifting property with respect to all CW complexes of dimension $\leq n$. By another standard argument (\cite[p.\ 376]{Spanier}) this implies that for any $e \in E$ with $b=p(e)$ the map $p_*: \pi_i(E, p^{-1}(b), e) \to \pi_i(B, \{b\}, b)$ is surjective for $i \leq n+1$ and injective for $i \leq n$. As $p^{-1}(b)$ is $n$-connected, the long exact sequence for the pair $(E, p^{-1}(b))$ shows that $\pi_i(E, e) \to \pi_i(E, p^{-1}(b), e)$ is surjective for $i \leq n+1$ and injective for $i \leq n$, so it follows that $p_*: \pi_i(E, e) \to \pi_i(B, b)$ is surjective for $i \leq n+1$ and injective for $i \leq n$ for all basepoints, i.e.\ $p$ is $(n+1)$-connected.
\end{proof}

\subsection{Semisimplicial sets and spaces}\label{sec:connectivity}

Let $\Delta_\mathrm{inj}^*$ be the category whose objects are the
ordered sets $[p] = (0 < \dots < p)$ with $p \geq -1$, and whose
morphisms are the injective order preserving functions.  An augmented
semisimplicial set is a contravariant functor $X$ from
$\Delta_\mathrm{inj}^*$ to the category of sets.  As usual, such a
functor is specified by the sets $X_p = X([p])$ and face maps $d_i:
X_p \to X_{p-1}$ for $i = 0, \dots, p$.  A (non-augmented)
semisimplicial set is a functor defined on the full subcategory
$\Delta_\mathrm{inj}$ on the objects with $p \geq 0$, and is the same thing as is sometimes called a ``$\Delta$-complex''.  Semisimplicial
spaces are defined similarly.  

Let us briefly discuss the relationship between simplicial complexes
and semisimplicial sets.  To any simplicial complex $K$ there is an
associated semisimplicial set $K_\bullet$, whose $p$-simplices are the
injective simplicial maps $\Delta^p \to K$, i.e.\ ordered
$(p+1)$-tuples of distinct vertices in $K$ spanning a $p$-simplex.  There is a
natural surjection $|K_\bullet| \to |K|$, and any choice of total
order on the set of vertices of $K$ induces a splitting $|K| \to
|K_\bullet|$.  In particular, $|K|$ is at least as connected as
$|K_\bullet|$.

We shall use the following well known
result.

\begin{proposition}
  \label{prop:connectivity-of-realisation}
  Let $f_\bullet : X_\bullet \to Y_\bullet$ be a map of semisimplicial
  spaces such that $f_p: X_p \to Y_p$ is $(n-p)$-connected for all
  $p$.  Then $|f_\bullet|: |X_\bullet| \to |Y_\bullet|$ is
  $n$-connected.
\end{proposition}
\begin{proof}
We shall use the skeleton filtration of geometric realizations, and recall that $|X_\bullet|^{(q)}$ is homeomorphic to a pushout $|X_\bullet|^{(q-1)} \leftarrow X_q \times  \partial \Delta^q \to X_q \times \Delta^q$ and similarly for $|Y_\bullet|$.
By induction on $q$ we will prove that $|X_\bullet|^{(q)} \to |Y_\bullet|^{(q)}$ is $n$-connected for all $q$.  The case $q = -1$ is vacuous, so we proceed to the induction step.
As $f_q : X_q \to Y_q$ is $(n-q)$-connected, we can find a factorisation
$$X_q \overset{g_q}\lra Z_q \overset{h_q}\lra Y_q$$
where $h_q$ is a weak homotopy equivalence, and the map $g_q$ is a relative CW-complex which only has cells of dimensions strictly greater than $n-q$. Define a new semi-simplicial space $Z_\bullet$ with $Z_q$ in degree $q$ and
$$Z_i = \begin{cases}
Y_i & \text{ if $i < q$}\\
X_i & \text { if $i > q$}
\end{cases}$$
and the evident face maps (in particular $d_i^Z : Z_q \to Z_{q-1}=Y_{q-1}$ is $d_i^Y \circ h_q$). There is then a factorisation $f_\bullet = h_\bullet \circ g_\bullet: X_\bullet \to Z_\bullet \to Y_\bullet$, and the map $|h_\bullet|^{(q)}: |Z_\bullet|^{(q)} \to |Y_\bullet|^{(q)}$ is a weak homotopy equivalence because $h_i$ is a weak homotopy equivalence for all $i \leq q$.

By induction, the map $|X_\bullet|^{(q-1)} \to |Z_\bullet|^{(q-1)} = |Y_\bullet|^{(q-1)}$ is $n$-connected and hence the map from $|X_\bullet|^{(q)}$ to the pushout of $|Z_\bullet|^{(q-1)} \leftarrow X_q \times  \partial \Delta^q \to X_q \times \Delta^q$ is also $n$-connected.  Since $|Z_\bullet|^{(q)}$ is obtained from this pushout by attaching cells of dimension strictly greater than $n$, namely $\Delta^q$ times the relative cells of $(Z_q,X_q)$,
the map $|g_\bullet|^{(q)}$ is the composition of an $n$-connected map and a relative CW complex with cells of dimension strictly greater than $n$, and hence $n$-connected.
\end{proof}

The following is the main result of this section.

\begin{proposition}\label{prop:semisimplicial-Serre-mic-fib}
  Let $Y_\bullet$ be a semisimplicial set, and $Z$ be a Hausdorff
  space. Let $X_\bullet \subset Y_\bullet \times Z$ be a
  sub-semisimplicial space which in each degree is an open
  subset. Then $\pi: |X_\bullet| \to Z$ is a Serre microfibration.
\end{proposition}
\begin{proof}
Let us write $p : |X_\bullet| \to |Y_\bullet \times Z| \to |Y_\bullet|$ for the map induced on geometric realisations by the inclusion and projection.
For $\sigma \in Y_n$, let us write $Z_\sigma \subset Z$ for the open
  subset defined by $(\{\sigma\} \times Z) \cap X_n = \{\sigma\}
  \times Z_\sigma$.  Points in $|X_\bullet|$ are described by data
$$(\sigma \in Y_n ;\, z \in Z_\sigma;\, (t_0, \ldots, t_n) \in \Delta^n)$$
up to the evident relation when some $t_i$ is zero, but we emphasise
that the continuous, injective map $\iota = p \times \pi : |X_\bullet|
\hookrightarrow |Y_\bullet| \times Z$ will not generally be a
homeomorphism onto its image, as the quotient of a subspace is not always a subspace of the quotient.
(For example, let $Y_0 = \{0,1\}$, $Y_1 = \{\iota\}$, $d_0 \iota = 0$, $d_1 \iota = 1$,
and $Y_i = \emptyset$ for $i > 1$,  $Z=[0,1]$, $X_0 = (\{0\} \times (0,1]) \cup (\{1\} \times [0,1])$ and $X_1 = \{\iota\} \times (0,1]$.  Then the image of $|X_\bullet| \to |Y_\bullet| \times Z \cong [0,1] \times [0,1]$ is $([0,1] \times (0,1]) \cup (\{1\} \times [0,1])$ but the inverse map is not continuous at $(1,0)$.  In fact
$|X_\bullet|$ is not first countable, so not homeomorphic to any subspace of $[0,1] \times [0,1]$.) This is not a problem: in fact we shall \emph{make use} of the topology on $|X_\bullet|$ being finer than the subspace topology.

Suppose now given a lifting problem
  \begin{equation*}
    \xymatrix{
       \{0\} \times D^k \ar[r]^-{f} \ar[d] &  \vert X_\bullet \vert\ar[d]^{\pi}\\
      [0,1] \times D^k \ar[r]^-{F} & Z. }
  \end{equation*}
  The composition $D^k \overset{f}\to |X_\bullet| \overset{p}\to
  |Y_\bullet|$ is continuous, so the image of $D^k$ is compact and
  hence contained in a finite subcomplex, and it intersects finitely
  many open simplices $\{\sigma_i\} \times \Int(\Delta^{n_i})\subset
  |Y_\bullet|$.  The sets $C_{\sigma_i} = (p \circ f)^{-1}(\{\sigma_i\}
  \times \Int(\Delta^{n_i}))$ then cover $D^k$, and their closures
  $\overline{C}_{\sigma_i}$ give a finite cover of $D^k$ by closed
  sets. Let us write $f\vert_{{C}_{\sigma_i}}(x) =
  (\sigma_i;z(x);t(x))$, with $z(x) \in Z_{\sigma_i} \subset Z$ and
  $t(x) = (t_0(x), \ldots, t_{n_i}(x)) \in \Int (\Delta^{n_i})$.

Certainly $\pi \circ f$ sends the set $C_{\sigma_i}$ into the open set
$Z_{\sigma_i}$, but it fact the following stronger property is true.

\vspace{1ex}

\noindent \textbf{Claim.} The map $\pi \circ f$ sends the set $\overline{C}_{\sigma_i}$ into $Z_{\sigma_i}$.

\begin{proof}[Proof of claim]
We consider a sequence $(x^j) \in C_{\sigma_i}$, $j \in \bN$, converging to a point $x \in
\overline{C}_{\sigma_i} \subset D^k$ and we shall verify that $z= \pi\circ f(x)
\in Z_{\sigma_i}$.

As $f$ is continuous, the sequence $f(x^j) =
(\sigma_i;z(x^j);t(x^j)) \in |X_\bullet|$ converges to $f(x)$ and,
passing to a subsequence, we may assume that the $t(x^j)$ converge to
a point $t\in\Delta^{n_i}$.  The
subset
\begin{equation*}
  A = \{f(x^j)\,\, \vert\,\, j \in \bN\} \subset |X_\bullet|
\end{equation*}
is contained in $\pi^{-1}(Z_{\sigma_i})$ and has $f(x)$ as a limit
point in $|X_\bullet|$, so if $z = \pi(f(x)) \not\in Z_{\sigma_i}$ then
the set $A$ is not closed in $|X_\bullet|$.  

For a contradiction we
will show that $A$ is closed, by proving that its inverse image in
$\coprod_\tau \{\tau\} \times Z_\tau \times \Delta^{|\tau|}$ is
closed, where the coproduct is over all simplices $\tau \in \coprod_n
Y_n$.  The inverse image in $\{\sigma_i\} \times Z_{\sigma_i} \times
\Delta^{n_i}$ is
\begin{equation*}
  B = \{(\sigma_i;\, z(x^j);\, t(x^j))\,\, \vert\,\, j \in \bN\},
\end{equation*}
which is closed (since $Z$ is Hausdorff, taking the closure in
$\{\sigma_i\} \times Z \times \Delta^{n_i}$ adjoins only the point
$(\sigma_i; z; t)$, which by assumption is outside $\{\sigma_i\} \times
Z_{\sigma_i} \times \Delta^{n_i}$).  If $\sigma_i = \theta^*(\tau)$
for a morphism $\theta \in \Delta_{\mathrm{inj}}$, we have $Z_\tau
\subset Z_{\sigma_i}$ and hence $B \cap (\{\sigma_i\} \times Z_\tau
\times \Delta^{|\sigma_i|})$ is closed in $\{\sigma_i\} \times Z_\tau
\times \Delta^{|\sigma_i|}$ so applying $\theta_*: \Delta^{|\sigma_i|}
\to \Delta^{|\tau|}$ gives a closed subset $B_\theta \subset
\{\tau\} \times Z_\tau \times \Delta^{|\tau|}$.  The inverse image of
$A$ in $\{\tau\} \times Z_\tau \times \Delta^{|\tau|}$ is the union of
the $B_\theta$ over the finitely many $\theta$ with $\theta^*(\tau) =
\sigma_i$, and is hence closed.
\end{proof}

We have a continuous map $F_i = F\vert_{[0,1] \times \overline{C}_{\sigma_i}} : [0,1] \times \overline{C}_{\sigma_i} \to
Z$ and by the claim $F_i^{-1}(Z_{\sigma_i})$ is an open neighbourhood of the
compact set $\{0\} \times \overline{C}_{\sigma_i}$, so there is an
$\epsilon_i > 0$ such that $F_i([0,\epsilon_i] \times
\overline{C}_{\sigma_i}) \subset Z_{\sigma_i}$. We set $\epsilon =
\min_i (\epsilon_i)$ and define the lift
$$\widetilde{F}_i(s, x) = (\sigma_i;F_i(s,x);t(x)) : [0,\epsilon]
\times \overline{C}_{\sigma_i} \lra \{\sigma_i\} \times Z_{\sigma_i} \times \Delta^{n_i} \lra |X_\bullet|,$$
which is clearly continuous. The functions $\widetilde{F}_i$ and
$\widetilde{F}_j$ agree where they are both defined, and so these glue
to give a continuous lift $\widetilde{F}$ as required.
\end{proof}

\begin{corollary}\label{cor:serre-microf-connectivity}
  Let $Z$, $Y_\bullet$, and $X_\bullet$ be as in Proposition
  \ref{prop:semisimplicial-Serre-mic-fib}.  For $z \in Z$, let
  $X_\bullet(z) \subset Y_\bullet$ be the sub-semisimplicial set
  defined by $X_\bullet \cap (Y_\bullet \times \{z\}) = X_\bullet(z)
  \times \{z\}$ and suppose that $|X_\bullet(z)|$ is $n$-connected for all
  $z \in Z$.  Then the map $\pi: |X_\bullet| \to Z$ is
  $(n+1)$-connected.
\end{corollary}
\begin{proof}
  This follows by combining Propositions~\ref{prop:Weiss-lemma}
  and~\ref{prop:semisimplicial-Serre-mic-fib}, once we prove that
  $|X_\bullet(z)|$ is homeomorphic to $\pi^{-1}(z)$ (in the subspace
  topology from $|X_\bullet|$).  Since $X_\bullet(z)\subset
  Y_\bullet$, the composition $|X_\bullet(z)| \to |X_\bullet| \to
  |Y_\bullet|$ is a homeomorphism onto its image.  It follows that
  $|X_\bullet(z)| \to |X_\bullet|$ is a homeomorphism onto its image,
  which is easily seen to be $\pi^{-1}(z)$.
\end{proof}


\section{Algebra}\label{sec:algebra}

We fix $\epsilon = \pm 1$ and let $\Lambda \subset \bZ$ be a subgroup satisfying
$$\{a - \epsilon {a} \,\,\vert\,\, a \in \bZ \} \subset \Lambda \subset \{a \in \bZ \,\,\vert\,\, a + \epsilon {a} =0\}.$$
Following Bak (\cite{Bak, Bak2}), we call such a pair $(\epsilon,
\Lambda)$ a \emph{form parameter}.  Since we work over the ground ring
$\bZ$, there are only three options for $(\epsilon,\Lambda)$, namely
$(+1, \{0\})$, $(-1,2\bZ)$ and $(-1,\bZ)$.  An $(\epsilon,\Lambda)$-\emph{quadratic module} is a triple $\sM = (M, \lambda,
\mu)$ where $M$ is a $\bZ$-module, $\lambda: M \otimes M \to \bZ$ is
bilinear and satisfies $\lambda(x,y) = \epsilon \lambda(y,x)$, and $\mu : M \to
\bZ/\Lambda$ is a quadratic form whose associated bilinear form is
$\lambda$ reduced modulo $\Lambda$.  By this we mean that $\mu$ is a function such that
\begin{enumerate}[(i)]
\item $\mu(a \cdot x) = a^2 \cdot \mu(x)$ for $a \in \bZ$,
\item $\mu(x+y) - \mu(x) - \mu(y) \equiv \lambda (x, y) \mod \Lambda$.
\end{enumerate}
We say the quadratic module $\sM = (M,\lambda,\mu)$ is \emph{non-degenerate} if the map
\begin{align*}
  \lambda^\vee : M &\lra M^\vee\\
  x &\longmapsto \lambda(-, x)
\end{align*}
is an isomorphism.
If $\sM = (M,\lambda_M,\mu_M)$ and $\sN = (N,\lambda_N,\mu_N)$ are quadratic modules with the same form parameter, then their (orthogonal) \emph{direct sum} is $\sM \oplus \sN = (M \oplus N,\lambda_\oplus,\mu_\oplus)$ where $\lambda_\oplus(m + n,m'+n') = \lambda_M(m,m') + \lambda_N(n+n')$ and $\mu_\oplus(m+n) = \mu_M(m) + \mu_N(n)$.
A \emph{morphism} $f: \sM \to \sN$ is a homomorphism $f: M \to N$ with the properties that
$\lambda_N \circ (f \otimes f) = \lambda_M$ and 
$\mu_M = \mu_N \circ f$.  If $\sM$ is non-degenerate, any such
morphism is 
split injective because
\begin{equation*}
  M \overset{f}\lra N \overset{\lambda_N^\vee}\lra N^\vee \overset{f^*}\lra M^\vee
\end{equation*}
is an isomorphism, and in fact induces an isomorphism $\sN \cong \sM \oplus
(f(\sM))^\perp$ of $(\epsilon,\Lambda)$-quadratic modules.

The \emph{hyperbolic module} $\sH$ is the non-degenerate
$(\epsilon,\Lambda)$-quadratic module with underlying abelian group free of rank 2 on two basis elements $e$ and $f$, and the unique quadratic module structure with $\lambda_\sH(e,f) = 1$, $\lambda_\sH(e,e) = \lambda_\sH(f,f) = 0$, and $\mu_\sH(e) = \mu_\sH(f) = 0$.
We write $\sH^g$ for the orthogonal direct sum of $g$ copies of $\sH$ and define the
\emph{Witt index} of an $(\epsilon, \Lambda)$-quadratic module $\sM$ as
\begin{equation*}
  g(\sM) = \sup\{g \in \bN\,\,\vert\,\, \text{there exists a morphism
    $\sH^g \to \sM$} \}.
\end{equation*}
The Witt index obviously satisfies $g(\sM \oplus \sH) \geq g(\sM) + 1$, and
we shall also consider the \emph{stable Witt index} defined by
\begin{equation}\label{eq:2}
  \overline{g}(\sM) = \sup \{g(\sM \oplus \sH^{k}) - k \mid k \geq 0\}.
\end{equation}
This satisfies $\overline{g}(\sM \oplus \sH) = \overline{g}(\sM) +1$ for all $\sM$.

\begin{definition}
  Let $(\epsilon,\Lambda)$ be a form parameter.  For a quadratic
  module $\sM$, let $K^a(\sM)$ be the simplicial complex
  whose vertices are morphisms $h:\sH \to \sM$ of quadratic modules.  The
  set $\{h_0, \dots, h_p\}$ is a $p$-simplex if the submodules $h_i(\sH)
  \subset \sM$ are orthogonal with respect to $\lambda$ (we impose no additional
  condition on the quadratic forms).
\end{definition}

The complex $K^a(\sM)$ is almost the same as one considered by Charney
\cite{Charney}, which she proves to be highly connected when $\sM =
\sH^{g}$.  We shall need a connectivity theorem for more general $\sM$,
assuming only that $\overline{g}(\sM) \geq g$.  In particular, we do not
wish to assume $\sM$ is non-degenerate (or even that the underlying
$\bZ$-module is free).  In Section~\ref{sec:proof-theor-refthm:c} we
shall give a self-contained proof of the following generalisation of
Charney's result.

\begin{theorem}\label{thm:Charney}
  Let $g \in \bN$ and let $\sM$ be a quadratic module with
  $\overline{g}(\sM) \geq g$. Then the geometric realisation $|K^a(\sM)|$ is $\lfloor
  \tfrac{g-4}{2} \rfloor$-connected, and $\lCM(K^a(\sM)) \geq \lfloor
  \tfrac{g-1} 2 \rfloor$.
\end{theorem}

Before embarking on the proof, let us deduce two consequences of the path-connectedness of $|K^a(\sM)|$.

\begin{proposition}[Transitivity]\label{propcor:AlgebraicTransitivity}
  If $|K^a(\sM)|$ is path-connected and $h_0, h_1 : \sH \to \sM$ are
  morphisms of quadratic modules, then there is an isomorphism of
  quadratic modules $f : \sM \to \sM$ such that $h_1 = f \circ h_0$.
\end{proposition}
\begin{proof}
  Suppose first that $h_0$ and $h_1$ are orthogonal. Then there is an
  orthogonal decomposition
  \begin{equation*}
    \sM \cong h_0(\sH) \oplus h_1(\sH) \oplus \sM'
  \end{equation*}
  and so an evident automorphism of quadratic modules which swaps the
  $h_i(\sH)$.  Now, the relation between morphisms $e : \sH \to \sM$ of
  differing by an automorphism is an equivalence relation, and we have
  just shown that adjacent vertices in $K^a(\sM)$ are equivalent.  If
  the complex is path-connected, then all vertices are equivalent.
\end{proof}

\begin{proposition}[Cancellation]\label{propcor:AlgebraicCancellation}
  Suppose that $\sM$ and $\sN$ are quadratic modules and there is an
  isomorphism $\sM \oplus \sH \cong \sN \oplus \sH$.  If $|K^a(\sM \oplus \sH)|$
  is path-connected, then there is also an isomorphism $\sM \cong \sN$.
\end{proposition}
\begin{proof}
  An isomorphism $\varphi : \sM \oplus \sH \to \sN \oplus \sH$ gives a
  morphism $\varphi\vert_\sH : \sH \to \sN \oplus \sH$ of quadratic modules, and
  we also have the standard inclusion $\iota : \sH \to \sN \oplus \sH$.  By
  Proposition~\ref{propcor:AlgebraicTransitivity}, these differ by an
  automorphism of $\sN \oplus \sH$, so in particular their orthogonal
  complements are isomorphic.
\end{proof}

By Theorem~\ref{thm:Charney}, the complex $K^a(\sM)$ is path-connected
provided $\overline{g}(\sM) \geq 4$.  As long as $\overline{g}(\sM) \geq
3$, Proposition \ref{propcor:AlgebraicCancellation} therefore gives
the implication $(\sM \oplus \sH \cong \sN \oplus \sH) \Rightarrow (\sM \cong
\sN)$.  It follows that as long as $\overline{g}(\sM) \geq 3$ we have $g(\sM
\oplus \sH) = g(\sM) + 1$ and hence $g(\sM) = \overline{g}(\sM)$, but for the
inductive proof of Theorem~\ref{thm:Charney} it is more convenient to
work with $\overline{g}$.

\section{Proof of Theorem \ref{thm:Charney}} \label{sec:proof-theor-refthm:c}

\begin{proposition}\label{prop:trans}
  Let $\Aut(\sH^{n+1})$ act on $\sH^{n+1}$, and consider the orbits of
  elements of $\sH\oplus 0 \subset \sH^{n+1}$.  Then we have
  \begin{align*}
    \Aut(\sH^{n+1}) \cdot (\sH \oplus 0) = \sH^{n+1}.
  \end{align*}
\end{proposition}

\begin{proof}
  We consider the form parameters $(+1,\{0\})$ and $(-1,2\bZ)$
  separately.  The case $(-1,\bZ)$ follows from the case $(-1,2\bZ)$, as the automorphism group is larger.  Recall that a vector $v =
  (a_0,b_0, \dots, a_n,b_n) \in \sH^{n+1}$ is called \emph{unimodular}
  if its coordinates have no common divisor.  Any $v \in \sH^{n+1}$ can
  be written as $v = d v'$ with $d \in \bN$ and $v' \in \sH^{n+1}$
  unimodular, so it suffices to prove that for any unimodular $v \in
  \sH^{n+1}$ there exists $\phi \in \Aut(\sH^{n+1})$ with $\phi(v) \in \sH
  \oplus 0$.

  In the case of form parameter $(+1,\{0\})$, this follows from
  \cite[Theorem 1]{WallTrans}, which asserts that $\Aut(\sH^{n+1})$ acts
  transitively on unimodular vectors of a given length.  Therefore,
  any unimodular vector $v \in \sH^{n+1}$ is in the same orbit as
  $(1,a,0,\dots, 0) \in \sH\oplus 0$ for $a = (\lambda(v,v))/2 =
  \mu(v)\in \bZ$.

  The case $(-1,2\bZ)$ can be proved in a manner similar to
  \cite[Theorem 1]{WallTrans}. First,
  in $\Aut(\sH)$ we have the transformations $(a,b) \mapsto \pm (b,-a)$,
  so any orbit has a representative with $0 \leq b \leq |a|$.  For $0
  < b < a$, the transformation $(a,b) \mapsto (a - 2b,b)$ will
  decrease the number $\max(|a|,|b|)$ and for $0 < b < -a$ the inverse
  transformation will do the same, so inductively we see that any
  orbit has a representative of the form $(a,0)$ or $(a,a)$ for some
  integer $a \geq 0$.  It follows that under $\Aut(\sH) \times \Aut(\sH)
  \leq \Aut(\sH^2)$ acting on $\sH^2$, the orbit of a unimodular vector
  has a representative of the form $v = (a,b,c,d)$ with $b\in
  \{0,a\}$, $d \in \{0,c\}$ and $\gcd(a,c)=1$.  On such a
  representative we then use the transformation
  \begin{equation*}
    (a,b,c,d) \longmapsto (a,b+c,c,d+a)
  \end{equation*}
  and since $\gcd(a,b+c)=1 = \gcd(c,d+a)$, we can use $\Aut(\sH) \times
  \Aut(\sH)$ to get to a representative with $a=c=1$ and $b,d \in
  \{0,1\}$. We can act on this representative 
  $(1,b,1,d)$ 
by the element of $\Aut(\sH^2)$ given  by right multiplication by the matrix
	$$
	\begin{pmatrix}
    1       & d & -d & 1  \\
    0       & 1 & 0 & 0  \\
    0       & -d & 0 & -1  \\
    0       & 1 & 1 & 0
\end{pmatrix}
	$$
to get the representative $(1,b+d,0,0) \in \sH \oplus 0 \leq \sH^2$.
This proves the case $n=1$, and the general case follows from this by induction.
\end{proof}

\begin{corollary}\label{cor:kernel}
  Let $\sM$ be a quadratic module with $g(\sM) \geq g$ and let $\ell: M
  \to \bZ$ be linear.  Then the quadratic form $\mathsf{K} = (\Ker(\ell), \lambda\vert_{\Ker(\ell)}, \mu\vert_{\Ker(\ell)})$ satisfies $g(\mathsf{K}) \geq g-1$.  Similarly if
  $\overline{g}(\sM) \geq g$ then $\overline{g}(\mathsf{K}) \geq g-1$.
\end{corollary}
\begin{proof}
  We can find a morphism $\phi: \sH^g \to \sM$.  By non-degeneracy of the
  form on $\sH^g$, the composite $\ell \circ \phi \in \Hom(\sH^g,\bZ)$ is of
  the form $\lambda(x,-)$.  By Proposition~\ref{prop:trans} we can,
  after precomposing $\phi$ with an automorphism, assume $x \in \sH
  \oplus 0 \subset \sH \oplus \sH^{g-1}$, and hence $0 \oplus \sH^{g-1}
  \subset \Ker(\ell)$, so $\phi$ restricts to a morphism $\sH^{g-1} \to  \mathsf{K}$.  
The claim about the stable Witt index follows from the
  unstable by considering $\sM \oplus \sH^k$.
\end{proof}

We may deduce the first non-trivial cases of Theorem~\ref{thm:Charney}
from this corollary.
\begin{proposition}\label{prop:non-empty-and-connected}
  If $\overline{g}(\sM) \geq 2$, then $K^a(\sM) \neq \emptyset$, and if
  $\overline{g}(\sM) \geq 4$ then $K^a(\sM)$ is path-connected.
\end{proposition}
\begin{proof}
  We consider the second case first.  Let us first
  make the stronger assumption that the (unstable) Witt index is $g(\sM)
  \geq 4$.  Then there exists an $h_0: \sH \to \sM$ with $g(h_0(\sH)^\perp)
  \geq 3$.  Any $h: \sH \to \sM$ then gives rise to a map of $\bZ$-modules
  \begin{equation*}
    h_0(\sH)^\perp \lra \sM \lra h(\sH)
  \end{equation*}
  where the first map is the inclusion and the second is orthogonal
  projection.  The kernel of this map is $h_0(\sH)^\perp \cap
  h(\sH)^\perp$. This is the intersection of the kernels of two linear maps on $h_0(\sH)^\perp$, so by Corollary~\ref{cor:kernel} we have
  $g(h_0(\sH)^\perp \cap h(\sH)^\perp) \geq 1$ and we can find an $h_1 \in
  K^a(h_0(\sH)^\perp \cap h(\sH)^\perp)$.
Then $\{h_0,h_1\}$ and
  $\{h_1,h\}$ are 1-simplices in $K^a(\sM)$, so there is a path (of length at most
  2) from any vertex to $h_0$.

  The general case $\overline{g}(\sM) \geq 4$ can be reduced to this by
  an argument as in the last paragraph of Section~\ref{sec:algebra}.
  Indeed, we can write $\sM \oplus \sH^k \cong \sN \oplus \sH^k$ for some
  integer $k \geq 0$ and quadratic module $\sN$ with $g(\sN) \geq 4$ and
  use the connectivity of $K^a(\sN \oplus \sH^j)$ for all $j \geq 0$ to
  inductively deduce from
  Proposition~\ref{propcor:AlgebraicCancellation} that $\sM \cong \sN$ and
  in particular $g(\sM) = g(\sN) \geq 4$.

  Similarly, if $\overline{g}(\sM) \geq 2$ we can write $\sM \oplus \sH^k
  \cong \sN \oplus \sH^{k}$ with $g(\sN) \geq 2$, and inductively use
  Proposition~\ref{propcor:AlgebraicCancellation} to see $\sM \oplus \sH
  \cong \sN \oplus \sH$.  As in the first part of the proof, $\sM$ is
  isomorphic to the intersection of the kernels of two linear maps $\sN
  \oplus \sH \to \bZ$.  By Corollary~\ref{cor:kernel} the Witt index
  drops by at most one for each linear map, so the Witt index of $\sM$
  is at least 1.
\end{proof}

\begin{proof}[Proof of Theorem~\ref{thm:Charney}]
  We proceed by induction on $g$.  At each stage of the induction, the
  statement $\lCM(K^a(\sM)) \geq \lfloor \tfrac{g-1}{2}\rfloor$ follows easily
  from the induction hypothesis.  Indeed, a $p$-simplex $\sigma =
  \{h_0, \ldots, h_p\} < K^a(\sM)$ induces (after choosing an ordering
  of its vertices) a canonical splitting $\sM \cong \sM' \oplus \sH^{p+1}$,
  where $\sM' = (\bigoplus h_i (\sH))^\perp \subset \sM$.  We then have an
  isomorphism of simplicial complexes $\Lk(\sigma) \cong K^a(\sM')$ and
  we have $\overline{g}(\sM') = \overline{g}(\sM) - p-1$.  By induction,
  $\Lk(\sigma)$ is then $\lfloor \tfrac{(g-p-1)-4}{2} \rfloor$-connected, and
  by the inequality
  \begin{equation*}
    \big\lfloor \tfrac{g-p-1-4}2\big\rfloor \geq \big\lfloor
    \tfrac{g-5}2\big\rfloor - p =
    \big\lfloor \tfrac{g-1}2\big\rfloor -p -2
  \end{equation*}
  we see that the link of any $p$-simplex is $(\lfloor
  \tfrac{g-1}{2}\rfloor -p -2)$-connected as required.

  It remains to prove the statement about connectivity of $K^a(\sM)$.
  The connectivity statement is void for $g\leq 1$.  For $g=2,3$ we
  assert that $|K^a(\sM)| \neq \emptyset$ and for $g=4,5$ we assert that
  $|K^a(\sM)|$ is path-connected; both are covered by
  Proposition~\ref{prop:non-empty-and-connected}.  For the induction
  step, let us assume that Theorem~\ref{thm:Charney} holds up to
  $g-1$, and let $\sM$ be a quadratic module with $\overline{g}(\sM) \geq
  g$.  By Proposition~\ref{prop:non-empty-and-connected} we have
  $K^a(\sM) \neq \emptyset$ so we may pick some $h: \sH \to \sM$.  If we
  write $\sM' = h(\sH)^\perp \subset \sM$, we have $\sM \cong \sM' \oplus \sH$
  and $h(e)^\perp = \sM' \oplus \bZ e$.  The inclusion $\sM'
  \hookrightarrow \sM$ may then be factored as $\sM' \hookrightarrow \sM'
  \oplus \bZ e \hookrightarrow \sM$ and we have an induced factorisation
  \begin{equation}\label{eq:inc}
    \xymatrix{
      \vert K^a(\sM') \vert \ar@{^(->}[r]^-{\CircNum{1}} &
      \vert K^a(\sM' \oplus \bZ e) \vert \ar@{^(->}[r]^-{\CircNum{2}}
      & \vert K^a(\sM) \vert.
    }
  \end{equation}

  We now wish to show that Proposition~\ref{prop:epi} applies to the
  maps $\CircNum{1}$ and $\CircNum{2}$, both of which are inclusions
  of full subcomplexes, with $n = \lfloor \frac{g-4}2 \rfloor$.  For
  $\CircNum{1}$, we use the projection $\pi: \sM' \oplus \bZ e \to
  \sM'$.  The summand $\bZ e$ has trivial quadratic structure and pairs
  trivially with anything in $\sM' \oplus \bZ e$, so $\pi$ is
  a morphism of quadratic modules and hence induces a retraction $\pi:
  K^a(\sM' \oplus \bZ e) \to K^a(\sM')$.  For any $p$-simplex $\sigma <
  K^a(\sM' \oplus \bZ e)$ we have
  \begin{equation*}
    K^a(\sM') \cap \Lk_{K^a(\sM' \oplus \bZ e)}(\sigma) =
    \Lk_{K^a(\sM')}(\pi(\sigma)),
  \end{equation*}
  and to apply Proposition~\ref{prop:epi} we must show that this
  simplicial complex is $(n-p-1)$-connected.  The splitting $\sM \cong
  \sM' \oplus \sH$ shows that $\overline{g}(\sM') \geq g-1$, so by induction
  we have $\lCM(K^a(\sM')) \geq \lfloor \frac{g-2}2\rfloor$, and therefore
  the link of the $p$-simplex $\pi(\sigma)$ is $(\lfloor
  \frac{g-2}2 \rfloor -p-2)$-connected.  But $\lfloor \frac{g-2}2
  \rfloor -p-2 = n-p-1$.

  For $\CircNum{2}$, we first note that $\sM' \oplus \bZ e \subset \sM$ is
  exactly the orthogonal complement $h(e)^\perp \subset \sM$.  For a
  $p$-simplex $\sigma = \{h_0, \dots, h_p\} \subset K^a(\sM)$ we write
  $\sM'' = (\sum h_i(\sH))^\perp \subset \sM$ and have
  \begin{equation*}
    K^a(\sM' \oplus \bZ e) \cap \Lk_{K^a(\sM)}(\sigma) = K^a(\sM'' \cap
    h(e)^\perp).
  \end{equation*}
  The isomorphism $\sM \cong \sM'' \oplus \sH^{p+1}$ shows that
  $\overline{g}(\sM'') \geq g-p-1$, and passing to the kernel of the
  linear functional $\lambda(h(e),-)\vert_{M''}$ reduces stable Witt
  index by at most one by Corollary~\ref{cor:kernel}, so we have
  $\overline{g}(\sM'' \cap h(e)^\perp) \geq g-p-2$.  By induction, the
  connectivity of $K^a(\sM'' \cap h(e)^\perp)$ is therefore at least
  $\big\lfloor \tfrac{(g-p-2)-4} 2\big\rfloor \geq n - p -1$.
  
  We have shown that both inclusions $\CircNum{1}$ and $\CircNum{2}$
  satisfy the hypothesis of Proposition~\ref{prop:epi} and therefore
  these maps are $n$-connected.  The composition factors through the
  star of the vertex given by $h$, and is therefore nullhomotopic.  This
  implies that $K^a(\sM)$ is $n$-connected, finishing the induction
  step.
\end{proof}


\section{Topology}\label{sec:Top}

Recall that in Section~\ref{sec:intr-stat-results} we defined the
manifold $W_{g,1} = D^{2n} \# g(S^n \times S^n)$ and that for a path-connected compact $2n$-manifold $W$ we defined $g(W)$ to be the maximal $g \in \bN$ for which there
exists an embedding $W_{g,1} \hookrightarrow W$. In analogy with~\eqref{eq:2} we define the \emph{stable genus} of $W$ to be
\begin{equation*}
  \overline{g}(W) = \max\{g(W \# k(S^n \times S^n)) - k\mid k \in \bN\}.
\end{equation*}
Notice that $k \mapsto g(W \# k(S^n \times S^n)) - k$ is non-decreasing and
bounded above by $b_n(W)/2$. In particular, the maximum is
well-defined.

It will be convenient to have available the following small
modification of the manifold $W_{1,1}$. Firstly, we may choose an embedding $W_{1,1} \hookrightarrow S^n \times S^n$ with complement an open disc lying in $D^n_+ \times D^n_+$, the product of the two upper hemispheres, and from now on we shall implicitly use this embedding identify $W_{1,1}$ with a subset of $S^n \times S^n$. 
Let $\Hmfld$ denote the manifold
obtained from $W_{1,1}$ by gluing $[-1,0] \times D^{2n-1}$ to $\partial W_{1,1}$ along an orientation preserving
embedding
\begin{equation*}
  \{-1\} \times D^{2n-1} \lra \partial W_{1,1},
\end{equation*}
which we also choose once and for all.  This gluing of course does not
change the diffeomorphism type (after smoothing corners), so $\Hmfld$ is
diffeomorphic to $W_{1,1}$, but contains a standard embedded $[-1,0]
\times D^{2n-1} \subset \Hmfld$. When we discuss embeddings of $\Hmfld$ into a
manifold with boundary $W$, we shall always insist that $\{0\} \times
D^{2n-1}$ is sent into $\partial W$, and that the rest of $\Hmfld$ is sent
into the interior of $W$.

We shall also need a \emph{core} $C \subset \Hmfld$, defined as
follows.  Let $x_0 \in S^n$ be a basepoint.  Let $S^n \vee S^n = (S^n
\times \{x_0\}) \cup (\{x_0\} \times S^n) \subset S^n \times S^n$,
which we may suppose is contained in $\Int(W_{1,1})$.  Choose an
embedded path $\gamma$ in $\Hmfld$
from $(x_0, -x_0)$ to $(0,0) \in
[-1,0] \times D^{2n-1}$ whose interior does not intersect $S^n \vee
S^n$, and whose image agrees with $[-1,0] \times \{0\}$ inside $[-1,0]
\times D^{2n-1}$, and let
\begin{equation*}
  C = (S^n \vee S^n) \cup \mathrm{Im}(\gamma) \cup (\{0\} \times D^{2n-1})
  \subset \Hmfld.
\end{equation*}
The manifold $\Hmfld$ is depicted together with $C \subset H$ in Figure
\ref{fig:1}.
\begin{figure}[h]
  \begin{center}
    \includegraphics[bb=00 0 181 78]{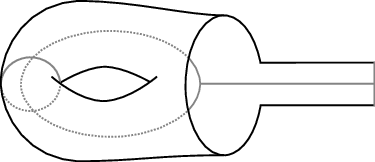}
  \end{center}
  \caption{The manifold $\Hmfld$ in the case $2n=2$, with the core $C$
    indicated in gray.}
  \label{fig:1}
\end{figure}

\begin{definition}\label{defn:K-p}
  Let $W$ be a compact manifold, equipped with (the germ of) an
  embedding $c: (-\delta, 0] \times \R^{2n-1} \to W$ for some $\delta
  > 0$, such that $c^{-1}(\partial W) = \{0\} \times \R^{2n-1}$.  Two
  embeddings $c$ and $c'$ define the same germ if they agree after
  making $\delta$ smaller.
  \begin{enumerate}[(i)]
  \item\label{item:5} Let $K_0(W) = K_0(W,c)$ be the space of pairs $(t,\phi)$, where $t \in \R$ and $\phi: \Hmfld \to W$ is an embedding whose
    restriction to $(-1,0] \times D^{2n-1} \subset \Hmfld$ satisfies
    that there exists an $\epsilon \in (0,\delta)$ such that
    \begin{equation*}
      \phi(s,p) = c(s, p + te_1)
    \end{equation*}
    for all $s \in (-\epsilon,0]$ and all $p \in D^{2n-1}$.  Here, $e_1 \in
    \R^{2n-1}$ denotes the first basis vector.
  \item Let $K_p(W) \subset (K_0(W))^{p+1}$ consist of those tuples
    $((t_0, \phi_0), \dots, (t_p, \phi_p))$ satisfying that $t_0 <
    \dots < t_p$ and that the embeddings $\phi_i$ have disjoint cores, i.e.\ the sets $\phi_i(C)$ are disjoint.
  \item Topologise $K_p(W)$ using the $C^\infty$-topology on the space
    of embeddings and let $K_p^\delta(W)$ be the same set considered
    as a discrete topological space.
  \item The assignments $[p] \mapsto K_p(W)$ and $[p] \mapsto
    K_p^\delta(W)$ define semisimplicial spaces, where the face map
    $d_i$ forgets $(t_i,\phi_i)$.
  \item Let $K^\delta(W)$ be the simplicial complex with vertices
    $K^\delta_0(W)$, and where the (unordered) set $\{(t_0,\phi_0),
    \dots, (t_p,\phi_p)\}$ is a $p$-simplex if, when written with $t_0
    < \dots < t_p$, it satisfies $((t_0, \phi_0), \dots, (t_p,
    \phi_p)) \in K_p^\delta(W)$.
  \end{enumerate}
  We shall often denote a vertex $(t,\phi)$ simply by $\phi$, since
  $t$ is determined by $\phi$.  Since a $p$-simplex of
  $K_\bullet^\delta(W)$ is determined by its (unordered) set of
  vertices, there is a natural homeomorphism $|K_\bullet^\delta(W)| \cong
  |K^\delta(W)|$.
\end{definition}

We wish to associate to each simply-connected $2n$-manifold a
quadratic module with form parameter $(1, \{0\})$ if $n$ is even and $(-1,
2\bZ)$ if $n$ is odd. Essentially, we take the group of immersed
framed $n$-spheres in $W$, with pairing given by the intersection
form, and quadratic form given by counting self-intersections. We shall often have to work with framings, and for a $d$-manifold $M$ we let $\Fr(M)$ denote the frame bundle of $M$, i.e.\ the (total space of the) principal $GL_d(\bR)$-bundle associated to the tangent bundle $TM$.

In the following definition we shall use the \emph{standard framing} $\xi_{S^n \times D^n}$ of $S^n \times D^n$, induced by the embedding
\begin{equation}\label{eq:StdEmbeding}
\begin{aligned}
S^n \times D^n &\lra \bR^{n+1} \times \bR^{n-1} = \bR^{2n}\\
(x; y_1, y_2, \ldots, y_n) &\longmapsto (2e^{-\frac{y_1}{2}} x ; y_2, y_3, \ldots, y_n).
\end{aligned}
\end{equation}
This standard framing at $(0,0,\ldots,0,-1;0) \in S^n \times D^n$ gives a point $b_{S^n \times D^n} \in \Fr(S^n \times D^n)$.

\begin{definition}\label{defn:ImmSphereQuadMod}
  Let $W$ be a compact manifold of dimension $2n$,
  equipped with a \emph{framed basepoint} i.e.\ a point $b_W \in
  \Fr(W)$, and an orientation compatible with $b_W$.
\begin{enumerate}[(i)]
\item 

Let $I\hatfr_n(W) = I\hatfr_n(W,b_W)$ denote the set of regular homotopy
  classes of immersions $i: S^n \times D^n \looparrowright W$ equipped
  with a path in $\Fr(W)$ from $Di(b_{S^n \times D^n})$ to $b_W$. Smale--Hirsch immersion theory \cite[Section 5]{Hirsch} identifies this set with the homotopy group $\pi_n(\Fr(W), b_W)$ of
  the frame bundle of $W$. The (abelian) group structure this induces
  on $I\hatfr_n(W)$ corresponds to a connect-sum operation. 
  
\item Let $\lambda : I\hatfr_n(W) \otimes
  I\hatfr_n(W) \to H_n(W;\bZ) \otimes H_n(W;\bZ) \to \bZ$ be
  the map which applies the homological intersection pairing to the cores of a pair of immersed framed spheres.

\item Let
$$\mu : I\hatfr_n(W) \lra \begin{cases}
\bZ & \text{$n$ even}\\
\bZ/2 & \text{$n$ odd}
\end{cases}$$
be the function which counts (signed, if $n$ is even) self-intersections of the core of a framed immersion (once it is perturbed to be self-transverse).
\end{enumerate}
\end{definition}

\begin{lemma}\label{lem:QuadMod}
For any smooth simply-connected $2n$-manifold $W$ equipped with a framed basepoint $b_W$ the data $\mathsf{I}\hatfr_n(W) = (I\hatfr_n(W), \lambda, \mu)$ is a quadratic
  module with form parameter $(1, \{0\})$ if $n$ is even and $(-1,
  2\bZ)$ if $n$ is odd.
\end{lemma}
\begin{proof}
Let $E \to W$ be the $n$th Stiefel bundle associated to the tangent bundle of $W$; there is a map $\phi: \Fr(W) \to E$ given by sending a frame to its first $n$ vectors, and this determines a basepoint $\phi(b_W) \in E$. Choosing a framing $\xi$ of $T_{(0,0,\ldots, 0,-1)}S^n$, Wall \cite[Section 5]{Wall} uses Smale--Hirsch immersion theory to identify the homotopy group $\pi_n(E, \phi(b_W))$ with the set $I_n(W)$ of regular homotopy classes of immersions $i: S^n \looparrowright W$ equipped with a path in $E$ from $Di(\xi)$ to $\phi(b_W)$.

The homomorphism $\pi_n(\phi) : \pi_n(\Fr(W), b_W) \to \pi_n(E, \phi(b_W))$ is thus identified with the map $u: I\hatfr_n(W) \to I_n(W)$ which restricts an immersion of $S^n \times D^n$ to $S^n \times \{0\}$. (In Wall's identification $I_n(W) \cong \pi_n(E)$ the element $0 \in \pi_n(E)$ corresponds to an embedding $S^n \hookrightarrow \bR^{2n} \subset W$ with trivial normal bundle. Our choice \eqref{eq:StdEmbeding} is compatible with this.)

Our functions $\lambda$ and $\mu$ factor through $u$, and Wall shows \cite[Theorem 5.2]{Wall} that $(I_n(W), \lambda, \mu)$ is a quadratic module with the appropriate form parameter.
%
\end{proof}

We remark that the bilinear form $\lambda$ can in general be quite far
from non-degenerate.  Let us also remark that 
 the
basepoint $b_W \in \mathrm{Fr}(W)$ and the paths from $Di(b_{S^n \times D^n})$ are used only for defining the addition
on the abelian group $I\hatfr_n(W)$, neither $\lambda$ nor
$\mu$ depends on this data.

For the manifold
$\Hmfld = ([-1,0] \times D^{2n-1}) \cup_{\{-1\} \times D^{2n-1}}W_{1,1}$ we
choose the framed basepoint $b_H$ given by the Euclidean framing of
$T_{(0,0)} ([-1,0] \times D^{2n-1})$, i.e.\ the framing induced by the
inclusion $[-1,0] \times D^{2n-1} \subset \R^{2n}$.  We define
canonical elements $e,f \in I\hatfr_n(\Hmfld)$ in the following
way. There are embeddings
\begin{equation}\label{eq:barE}
\begin{aligned}
\bar{e} : S^n \times D^n &\lra W_{1,1} \subset S^n \times S^n \subset \bR^{n+1} \times \bR^{n+1}\\
(x,y) &\longmapsto \left(x; \tfrac{y}{2}, - \sqrt{1- \vert \tfrac{y}{2}\vert^2}\right)
\end{aligned}
\end{equation}
and
\begin{equation}\label{eq:barF}
\begin{aligned}
\bar{f} : S^n \times D^n &\lra W_{1,1} \subset S^n \times S^n \subset \bR^{n+1} \times \bR^{n+1}\\
(x,y) &\longmapsto \left(-\tfrac{y}{2}, - \sqrt{1- \vert \tfrac{y}{2}\vert^2}; x\right)
\end{aligned}
\end{equation}
which are orientation preserving. These may be considered as embeddings into $\Hmfld$. 

The embedding $\bar{e}$, together with a choice of path in $\Fr(\Hmfld)$ from $D\bar{e}(b_{S^n \times D^n})$ to $b_H$, defines an element $e \in I\hatfr_n(\Hmfld)$ and $\bar{f}$ similarly defines an element $f \in I\hatfr_n(\Hmfld)$. These elements satisfy
\begin{equation*}
  \lambda(e,e) = \lambda(f,f) = 0 \quad\quad \lambda(e,f)=1 \quad\quad
  \mu(e)=\mu(f)=0,
\end{equation*}
and so determine a morphism of quadratic modules $\sH \to \mathsf{I}\hatfr_n(W)$. (This morphism depends on the choice of paths in $\Fr(\Hmfld)$ made above
and 
is used to define the map \eqref{eq:7} below. 
Once
that map has been used to prove
 that its domain is highly-connected, 
this choice plays no further role. The ambiguity in this choice 
arises at the end of the proof of Lemma \ref{lemthm:conn-K-delta}.)

\begin{remark}\label{rem:HFraming}
For use in Section \ref{sec:TS}, let us point out that there exists a framing  of $\Hmfld$ which is homotopic to the standard framing on the images of $\bar{e}$ and $\bar{f}$, and extends the Euclidean framing on $\{0\} \times D^{2n-1} \subset [-1,0] \times D^{2n-1} \subset \Hmfld$. To see this, note that the standard framings on $\bar{e}(S^n \times D^n)$ and $\bar{f}(S^n \times D^n)$ are homotopic when restricted to the (contractible) intersection of these subsets, as both embeddings are orientation preserving. This allows us to construct the framing on $W_{1,1} \subset \Hmfld$, and because $[-1,0] \times D^{2n-1}$ is glued to $W_{1,1}$ along an orientation preserving map this framing may be extended to one of $\Hmfld$ agreeing with the Euclidean framing on $\{0\} \times D^{2n-1}$.
\end{remark}

If $(W_0, b_{W_0})$ and $(W, b_W)$ are manifolds with framed basepoints, a
\emph{morphism} $(e_0, \sigma_0) : (W_0, b_{W_0}) \to (W, b_W)$ consists of a
(codimension 0) embedding $e_0 : W_0 \hookrightarrow W$ and a path
$\sigma_0 : De_0(b_{W_0}) \leadsto b_W \in \Fr(W)$. Such a morphism
induces a homomorphism 
of quadratic modules $(e_0, \sigma_0)_* : \sI\hatfr_n(W_0)
\to \sI\hatfr_n(W)$, since
$\lambda$ and $\mu$ are computed by counting
intersections which may be done in either manifold. Furthermore, if
$(e_1, \sigma_1) : (W_1, b_{W_1}) \to (W, b_W)$ is another morphism such that $e_1$ is disjoint from $e_0$
(up to isotopy), then $(e_0, \sigma_0)_*$ and $(e_1, \sigma_1)_*$ have
orthogonal images in $\mathsf{I}\hatfr_n(W)$.

For a manifold $W$ with distinguished chart $c : (-1,0] \times \bR^{2n-1} \hookrightarrow W$ we choose $b_W$ to be induced by $Dc$ and the Euclidean framing of $T_{(0,0)}((-1,0] \times \bR^{2n-1})$. Then an embedding $\phi : \Hmfld \hookrightarrow W$ representing a vertex of $K^\delta(W)$ has a canonical homotopy class of path $\sigma_\phi$ from $\phi \circ b_H$ to $b_W$ (as the manifolds $\phi([-1,0] \times D^{2n-1})$ and $c((-1,0] \times \bR^{2n-1})$ are both contractible and framed). Thus $\phi$ gives a
hyperbolic submodule
\begin{equation*}
  \sH \lra\mathsf{I}\hatfr_n(\Hmfld) \overset{(\phi, \sigma_\phi)_*}\lra
  \mathsf{I}\hatfr_n(W)
\end{equation*}
and disjoint embeddings give orthogonal hyperbolic submodules, which
defines a map of simplicial complexes
\begin{equation}\label{eq:7}
  K^\delta(W) \lra K^a(\mathsf{I}\hatfr_n(W)).
\end{equation}
We will use this map to compute the connectivity of $\vert
K_\bullet^\delta(W)\vert = |K^\delta(W)|$. The proof of the following lemma, and its generalisation in Section~\ref{sec:TS} to the presence of tangential structures, make essential use of the Whitney trick.

\begin{lemma}\label{lemthm:conn-K-delta}
If $(W, b_W)$ as in Lemma \ref{lem:QuadMod} has dimension $2n \geq 6$ and is  simply-connected, then the space $|K^\delta_\bullet(W)|$ is $\lfloor \tfrac{\overline{g}(W)-4}{2} \rfloor$-connected.  
\end{lemma}
\begin{proof}
  For brevity we
  shall just write $K^\delta \to K^a$ for the map \eqref{eq:7} and
  write
  $\overline{g} = \overline{g}(\mathsf{I}\hatfr_n(W))$.
  We have $\overline{g}(W) \leq \overline{g}$, so it suffices to show
  that $|K^\delta_\bullet(W)|$ is
  $\lfloor \tfrac{\overline{g}-4}{2} \rfloor$-connected.

  Let $k \leq (\overline{g}-4)/{2}$ and consider a map $h: \partial
  I^{k+1} \to |K^\delta|$, which we may assume is simplicial with
  respect to some PL triangulation $\partial I^{k+1} \approx |L|$.  By Theorem~\ref{thm:Charney}, the
  composition $\partial I^{k+1} \to |K^\delta| \to |K^a|$ is
  nullhomotopic and so extends to a map $\overline{h}: I^{k+1} \to
  |K^a|$. By Theorem~\ref{thm:Charney}, we also have $\lCM(K^a) \geq
  \lfloor \tfrac{\overline{g}-1}{2}\rfloor \geq k+1$, so by
  Theorem~\ref{thm:simplex-wise-injective} we may find a triangulation $I^{k+1} \approx |K|$ extending $L$ such that the star of each vertex $v \in K \setminus L$ intersects $L$ in a single simplex, and change $\overline{h}$ by a homotopy relative to $\partial I^{k+1}$ so that $\overline{h}(\Lk(v)) \subset \Lk(\overline{h}(v))$ for each vertex $v \in K \setminus L$. We will prove that $\overline{h}: I^{k+1} \to |K^a|$ lifts to a nullhomotopy $F : I^{k+1} \to |K^\delta|$ of $h$.

  Choose an enumeration of the
  vertices in $K$ as $v_1, \dots, v_N$ such that the vertices in $L$ come before
  the vertices in $K \setminus L$. For each $v_i \in L$ the vertex $h(v_i) = F(v_i) \in K^\delta$ is given by an embedding $j_i: \Hmfld \to W$. For $v_i \in K \setminus L$ we shall inductively pick lifts of
   $\overline{h}(v_i) \in K^a$ to a
  vertex $F(v_i) \in K^\delta$ given by an embedding $j_i: \Hmfld \to W$
  satisfying
  \begin{enumerate}[(i)]
  \item\label{item:9} if $v_i,v_k$ are adjacent vertices in $K$ with $k < i$ then
    $j_i(C) \cap j_k(C) = \emptyset$,
  \item\label{item:10} for $v_i \in K \setminus L$ the core $j_i(C)$
    is in general position with $j_k(C)$ for $k < i$.
  \end{enumerate}
  Suppose $\overline{h}(v_1), \dots, \overline{h}(v_{i-1})$ have been
  lifted to $j_1, \dots, j_{i-1}$ satisfying~(\ref{item:9})
  and~(\ref{item:10}).  Then $v_i \in K\setminus L$ gives a morphism
  of quadratic modules
  $\overline{h}(v_i) = \phi: \sH \to \mathsf{I}\hatfr_n(W)$ which we wish
  to lift to an embedding $j_i$ satisfying the two properties. The
  element $\phi(e)$ is represented by an immersion
  $x : S^n \times D^n \looparrowright W$, which has $\mu(x)=0$, along with a path in $\Fr(W)$ from $Dx(b_{S^n \times D^n})$ to $b_W$. As $W$ is simply-connected and of dimension at least 6 we may use the
  Whitney trick as in \cite[Theorem 5.2]{Wall} to replace $x$ by an
  embedding $j(e) : S^n \times D^n \hookrightarrow W$. Similarly,
  $\phi(f)$ can be represented by an embedding
  $j(f) : S^n \times D^n \hookrightarrow W$, along with another path in $\Fr(W)$.

  As $\lambda(\phi(e), \phi(f))=1$, these two embeddings have algebraic
  intersection number 1. We may again use the Whitney trick to isotope
  the embeddings $j(e)$ and $j(f)$ so that their cores
  $S^n \times \{0\}$ intersect transversely in precisely one point,
  and so obtain an embedding of the plumbing of $S^n \times D^n$ and
  $D^n \times S^n$, which is diffeomorphic to $W_{1,1} \subset \Hmfld$.  We then use the framed path from $Dx(b_{S^n \times D^n})$ to $b_W$ to
  extend to the remaining $[-1,0]\times D^{2n-1} \subset \Hmfld$, 
	giving an embedding
  $j : \Hmfld \hookrightarrow W$.  Setting $j_i = j$ may not
  satisfy~(\ref{item:9}) and (\ref{item:10}), but after a small
  perturbation it will satisfy~(\ref{item:10}).

  It remains to explain how to achieve
  that $j_1, \dots, j_i$ satisfy~(\ref{item:9}).  If
  $v_k \in \Lk(v_i)$ is an already lifted vertex (i.e.\ $k < i$), we
  must ensure that $j_i(C) \cap j_k(C) = \emptyset$.  Since they are in
  general position and have algebraic intersection numbers zero (as
  $\overline{h}$ is simplexwise injective so $\{v_i,v_k\}$ maps to a
  1-simplex in $K^a$) we may use the Whitney trick to replace $j_i$ by
  an embedding satisfying $j_i(C) \cap j_k(C) = \emptyset$.  The
  necessary Whitney discs may be chosen disjoint from $j_m(C)$, for all $m < i$ such that $m \neq k$ and either $v_m \in \Lk(v_i)$ or $v_m \in K \setminus L$, again since
  all such $v_m$ are in general position with each other.  Then the
  Whitney trick will not create new intersections and after finitely
  many such Whitney tricks we will have
  $j_i(C) \cap j_k(C) = \emptyset$ whenever $v_k \in \Lk(v_i)$ and
  $k < i$, ensuring that the lifts $j_1, \dots, j_i$
  satisfy~(\ref{item:9}) and~(\ref{item:10}).

  In finitely many steps, we arrive at a lift of
  $\overline{h}: I^{n+1} \to K^a$ to a nullhomotopy of
  $h: \partial I^{n+1} \to K^\delta$, as desired.  (Strictly speaking, we may not have lifted the chosen nullhomotopy
  $\overline{h}: I^{k+1} \to |K^a|$.  The data of an element in
  $I\hatfr_n(W)$ includes a path in $\Fr(W)$.  If $W$ is spin,
  there will be two choices of such paths, related by a ``spin flip'',
  and we have only lifted the $\phi(e)$ and $\phi(f)$ up to spin flip.  Thus
  instead of lifting $\overline{h}: I^{k+1} \to |K^a|$ we may have
  lifted another nullhomotopy, related to $\overline{h}$ by spin
  flips on some vertices.)
\end{proof}

Finally, we compare $|K_\bullet^\delta(W)|$ and $|K_\bullet(W)|$.  The
bisemisimplicial space in Definition~\ref{defn:D-bullet-bullet} below
will be used to leverage the known connectivity of
$|K_\bullet^\delta(W)|$ to prove the following theorem, which is
the main result of this section.

\begin{theorem}\label{thm:high-conn}
If $W$ is a compact simply-connected manifold of dimension $2n \geq 6$ equipped with a framed basepoint, then the space $|K_\bullet(W)|$ is $\lfloor
  \tfrac{\overline{g}(W)-4}{2}\rfloor$-connected.
\end{theorem}

The semisimplicial space $K_\bullet(W)$ has an analogue in the case $2n=2$, which has been considered by Hatcher--Vogtmann \cite{HV} and shown to be highly connected.

\begin{definition}\label{defn:D-bullet-bullet}
  With $W$ and $c$ as in Definition~\ref{defn:K-p}, let $D_{p,q} =
  K_{p+q+1}(W)$, topologised as a subspace of $K_p(W) \times
  K^\delta_q(W)$.  This is a bisemisimplicial space, equipped with
  augmentations
  \begin{align*}
    D_{p,q} &\overset{\epsilon}\lra K_p(W)\\
    D_{p,q} &\overset{\delta}\lra K_q^\delta(W).
  \end{align*}
\end{definition}
\begin{lemma}\label{lem:homotopy-commute}
  Let $\iota: K_\bullet^\delta(W) \to K_\bullet(W)$ denote the
  identity map.  Then
  \begin{equation*}
    |\iota| \circ |\delta| \simeq |\epsilon|: |D_{\bullet,
      \bullet}| \lra |K_\bullet(W)|.
  \end{equation*}
\end{lemma}
\begin{proof}
  For each $p$ and $q$ there is a homotopy
  \begin{align*}
    [0,1] \times \Delta^p \times \Delta^q \times D_{p,q} \lra
    \Delta^{p+q+1} \times K_{p+q+1}(W)\\
    (r,s,t,x,y)) \longmapsto ((rs,(1-r)t),(x,\iota y)),
  \end{align*}
  where we write $(x,y) \in D_{p,q} \subset K_p(W)\times
  K_q^\delta(W)$ and $(x,\iota y) \in K_{p+q+1}(W) \subset K_p(W)
  \times K_q(W)$ and $(rs,(1-r)t) = (rs_0,\dots, rs_p,(1-r)t_0, \dots,
  (1-r)t_q) \in \Delta^{p+q+1}$.  These homotopies glue to a homotopy
  $[0,1] \times |D_{\bullet,\bullet}| \to |K_\bullet(W)|$ which
  starts at $|\iota| \circ |\delta|$ and ends at $|\epsilon|$.
\end{proof}

\begin{proof}[Proof of Theorem \ref{thm:high-conn}]
  Let us write $\overline{g} = \overline{g}(W) \leq
  \overline{g}(\mathsf{I}\hatfr_n(W))$.  We will apply
  Corollary \ref{cor:serre-microf-connectivity} with $Z = K_p(W)$,
  $Y_\bullet= K_\bullet^\delta(W)$ and $X_\bullet = D_{p,
    \bullet}$. For $z = ((t_0,\phi_0), \dots, (t_p,\phi_p)) \in
  K_p(W)$, we shall write $W_z \subset W$ for the complement of the
  $\phi_i(C)$. The realisation of the semisimplicial subset
  $X_\bullet(z) \subset Y_\bullet = K_\bullet^\delta(W)$ is
  homeomorphic to the full subcomplex $F(z) \subset K^\delta(W)$ on
  those $(t, \phi)$ such that $\phi(C) \subset W_z$ and $t > t_p$. The
  map of simplicial complexes~\eqref{eq:7} restricts to a map
  \begin{equation*}
    F(z) \lra K^a(\mathsf{I}\hatfr_n(W_z)).
  \end{equation*}
  We have $\mathsf{I}\hatfr_n(W_z) \oplus \sH^{p+1} \cong
  \mathsf{I}\hatfr_n(W)$, so
  $\overline{g}(\mathsf{I}\hatfr_n(W_z)) =
  \overline{g}(\mathsf{I}\hatfr_n(W))-p-1$,
  and hence by Theorem~\ref{thm:Charney} the target is $\lfloor
  \tfrac{\overline{g}-4-p-1}{2}\rfloor$-connected. The argument of Lemma \ref{lemthm:conn-K-delta} shows that $F(z)$ is
  also $\lfloor \tfrac{\overline{g}-4-p-1}{2}\rfloor$-connected.

  By Proposition~\ref{prop:semisimplicial-Serre-mic-fib}, the map $\vert\epsilon_p\vert : \vert D_{p, \bullet}\vert \to K_p(W)$ is a Serre microfibration, and we have just
  shown that it has $\lfloor \tfrac{\overline{g}-p-5}{2}\rfloor$-connected
  fibres, so by Proposition~\ref{prop:Weiss-lemma} it is $\lfloor
  \tfrac{\overline{g}-p-3}{2}\rfloor$-connected.  Since $\lfloor
  \tfrac{\overline{g}-p-3}{2}\rfloor \geq \lfloor \tfrac{\overline{g}-3}{2}\rfloor - p$, we deduce by Proposition~\ref{prop:connectivity-of-realisation}
  that the map $|D_{\bullet,\bullet}| \to |K_\bullet(W)|$ is $\lfloor
  \tfrac{\overline{g} -3}{2} \rfloor$-connected.  But up to homotopy it
  factors through the $\lfloor \tfrac{\overline{g}-4}{2}\rfloor$-connected
  space $|K_\bullet^\delta(W)|$, and therefore $|K_\bullet(W)|$ is
  $\lfloor \tfrac{\overline{g}-4}{2}\rfloor$-connected, too.
\end{proof}

Finally, we define the semisimplicial space which will shall use in the following section.

\begin{definition}\label{defn:Kbar}
Let $\overline{K}_\bullet(W)
\subset K_\bullet(W)$ denote the sub-semisimplicial space with $p$-simplices those tuples of
embeddings which are disjoint.  (Recall that in $K_\bullet(W)$ we only ask for the
embeddings to have disjoint cores.)
\end{definition}

\begin{corollary}\label{cor:KOverlineConn}
  The space $|\overline{K}_\bullet(W)|$ is
  $\lfloor\frac{\overline{g}(W)-4}2\rfloor$-connected.
\end{corollary}
\begin{proof}
  We may choose an isotopy of embeddings $\rho_t : \Hmfld \to \Hmfld$, defined
  for $t \in [0,\infty)$, which starts at the identity, eventually has
  image inside any given neighbourhood of $C$, and which for each $t$
  is the identity on some neighbourhood of $C$.  Precomposing with the
  isotopy $\rho_t$, any tuple of embeddings with disjoint cores
  eventually become disjoint.  It follows that the inclusion is a
  levelwise weak homotopy equivalence.
\end{proof}


\section{Resolutions of moduli spaces}\label{sec:Resolutions}

In the introduction we defined $\MM(X)$ as the classifying space $B\Diff_\partial(X)$ of the group of diffeomorphisms of $X$ fixing its boundary. In this section we will describe a specific point-set model for this classifying space, together with a simplicial resolution. 
We then use the spectral sequence arising from this resolution to prove Theorem~\ref{thm:main}.

\begin{definition}
For a $2n$-manifold $X$ with boundary $P$ and collar $c :
  (-\infty,0] \times P \hookrightarrow X$, and an $\epsilon >0$, let
  $\Emb_\epsilon(X, (-\infty,0] \times \bR^\infty)$ denote the space,
  in the $C^\infty$-topology, of those embeddings $e$ such that $e
  \circ c(t, x) = (t, x)$ as long as $t \in (-\epsilon,0]$, and let
  \begin{equation}
    \label{eq:6}
    \mathcal{E}(X)= \colim_{\epsilon \to 0} \Emb_\epsilon(X, (-\infty,0] \times \bR^\infty).
  \end{equation}
	The space $\mathcal{E}(X)$ has a (free) action of
$\Diff_\partial(X)$ by precomposition, and we write
\begin{equation*}
  \mathcal{M}(X) = \mathcal{E}(X)/\Diff_\partial(X).
\end{equation*}

Two elements of $\mathcal{E}(X)$ are in the same orbit if and
only if they have the same image, so as a set, $\mathcal{M}(X)$ is the set of
submanifolds $M \subset (-\infty,0] \times \R^\infty$ such that
  \begin{enumerate}[(i)]
  \item\label{item:6} $M \cap (\{0\} \times \bR^\infty) = \{0\} \times
    P$ and $M$ contains $(-\epsilon,0] \times P$ for some $\epsilon >
    0$,

  \item\label{item:7} the boundary of $M$ is precisely $\{0\} \times
    P$,

  \item\label{item:8} $M$ is diffeomorphic to $X$ relative to
    $P$.
  \end{enumerate}
	(The underlying set of $\mathcal{M}(X)$ depends on the specified
identification $\partial X \cong P \subset \R^\infty$.) 
\end{definition}

By \cite{MR613004} the quotient map $\mathcal{E}(X) \to \mathcal{E}(X)/\Diff_\partial(X)$ has slices and hence is a principal $\Diff_\partial(X)$-bundle.  Since
$\mathcal{E}(X)$ is weakly contractible by Whitney's embedding theorem, the
quotient space $\mathcal{M}(X)$ is a model for $B\Diff_\partial(X) = \MM(X)$. In this model for $\MM(X)$, the map \eqref{eq:GluingMap} which glues
on a cobordism $K$ is modelled using a choice of
collared embedding $K \hookrightarrow [-1,0] \times \bR^\infty$ such that $K
\cap (\{-1\} \times \bR^\infty) = P$.  Then the gluing map is
\begin{equation*}
  \begin{aligned}
    - \cup K : \mathcal{M}(X) &\lra \mathcal{M}(X \cup_{P} K)\\
    M &\longmapsto (M -  e_1) \cup K,
  \end{aligned}
\end{equation*}
that is, translation by one unit in the first coordinate
direction followed by union of submanifolds of $(-\infty, 0] \times \bR^\infty$.

Let $P$ be a closed non-empty $(2n-1)$-manifold, and let $W$ and $M$
be path-connected compact $2n$-manifolds with identified boundaries
$\partial M = P = \partial W$.  We say that $M$ and $W$ are
\emph{stably diffeomorphic} relative to $P$ if there is a
diffeomorphism
$$W \# W_g \cong M \# W_h$$
relative to $P$, for some $g, h \geq 0$.

\begin{definition}\label{defn:ModuliSpaceStab}
  Let $P \subset \bR^\infty$ be a closed non-empty $(2n-1)$-manifold,
  and let $W$ be a compact manifold, with a specified identification
  $\partial W = P$. Let
	$$\Mst = \Mst(W) = \coprod_{[T]} \mathcal{M}(T),$$
	where the union is taken over the set of compact manifolds
  with $\partial T = P$ and $T$ stably diffeomorphic to $W$ relative to
  $P$, one in each diffeomorphism class relative to $P$. The space $\Mst$ depends on $P \subset \R^\infty$ and the stable
diffeomorphism class of $W$ relative to $P = \partial W$, but we shall
suppress that from the notation.
\end{definition}

In order to formulate the analogue of Theorem \ref{thm:main} in this
model, we choose a submanifold $S \subset [-1,0] \times \bR^\infty$
with collared boundary $\partial S = \{-1,0\} \times P = S \cap
(\{-1,0\} \times \R^\infty)$, such that $S$ is diffeomorphic to $([-1,0] \times P)\#W_1$ relative to its boundary.  If $P$ is not path-connected, we also choose in which path component to perform the
connected sum.  Gluing $K = S$ then induces the self-map
\begin{equation}\label{eq:StabMap}
  \begin{aligned}
    s = - \cup S : \Mst &\lra \Mst\\
    M &\longmapsto (M - e_1) \cup S.
  \end{aligned}
\end{equation}
(As $M \cup_P S \cong M \# W_1$ relative to $P$, so $M \cup_P S$ is
stably diffeomorphic to $W$ if and only if $M$ is.) 

Let us write $\Mst_g \subset \Mst$ for the subspace of manifolds of stable genus precisely $g$, that is, those manifolds $M \in \Mst$ such that $\bar{g}(M)=g$. Then (by definition of the stable genus) $s$ restricts to a map $s : \Mst_g \to \Mst_{g+1}$.

\begin{theorem}\label{thm:Main:sec6}
  If $2n \geq 6$ and $W$ is simply-connected then the map
$$s_* : H_k(\Mst_g) \lra H_k(\Mst_{g+1})$$  
  induces an isomorphism for $2k \leq g-3$ and an epimorphism for $2k \leq g-1$.
\end{theorem}

In particular, this theorem implies that for any simply-connected
manifold $W$ with boundary $P$, the restriction
$$s: \mathcal{M}(W) \lra \mathcal{M}(W \cup_P S)$$
induces an isomorphism on homology for degrees satisfying $2* \leq
\bar{g}(W)-3$ and an epimorphism on homology
in degrees satisfying $2* \leq \bar{g}(W)-1$, which, along
with the observation that $g(W) \leq \bar{g}(W)$, establishes Theorem
\ref{thm:main}. 

Let us point out that Theorem~\ref{thm:Main:sec6} is slightly stronger
than Theorem~\ref{thm:main}, and in particular includes a non-trivial
assertion about $H_0$: Theorem~\ref{thm:Main:sec6} implies that $s$
induces an isomorphism in $H_0$ when $\overline{g} \geq 3$, which in
turn implies the following cancellation result.

\begin{corollary}\label{cor:cancellation}
  Let $P$ be a $(2n-1)$-manifold, $2n \geq 6$, and $W$ and $W'$ be two
  simply-connected manifolds with boundary $P$ such that $W \# W_g
  \approx W' \# W_g$ relative to $P$, for some $g \geq 0$. If
  $\bar{g}(W) \geq 3$, then $W \approx W'$ relative to $P$.
\end{corollary}
\begin{proof}[Proof, using Theorem~\ref{thm:Main:sec6}]
  Choose an embedding $P \subset \bR^\infty$, and let $\Mst$ be the
  space of manifolds stably diffeomorphic to $W$ as in Definition
  \ref{defn:ModuliSpaceStab}. The manifolds $W$ and $W'$ determine path
  components $[W]$ and $[W']$ of $\Mst$. As $W \# W_g \approx W' \#
  W_g$, $W$ and $W'$ have the same stable genus, say $\overline{g}(W)
  = \overline{g}(W') = h \geq 3$.  Thus they represent classes $[W], [W'] \in
  H_0(\Mst_h)$.  Theorem \ref{thm:Main:sec6} implies
  that the $g$-fold stabilisation map $s^g: \Mst_h \to \Mst_{h+g}$ induces an
  isomorphism $H_0(\Mst_h) \to H_0(\Mst_{h+g})$, so $s^g([W]) =
  [W\#W_g] = [W'\# W_g] = s^g([W'])$ implies $[W] = [W']$.  Therefore
  $W$ and $W'$ are in the same path component of $\Mst_h$, i.e.\ they are
  diffeomorphic relative to $P$.
\end{proof}

\begin{remark}\label{rem:kreck}
  Kreck (\cite[Theorem D]{Kreck}) has proved a cancellation result
  similar to Corollary~\ref{cor:cancellation} with the assumption
  $\overline{g}(W) \geq 3$ weakened to $g(W) \geq 1$ if $n$ is even
  and $g(W) \geq 0$ if $n$ is odd. (The hypotheses of Kreck's theorem are equivalent to ours, by \cite[Theorem C]{Kreck}.)
  Theorem~\ref{thm:Main:sec6} could
  perhaps be viewed as a ``higher-homology analogue'' of Kreck's
  cancellation result.
\end{remark}

As an immediate consequence of the above cancellation result, we may
deduce that  if $W$ is simply-connected and $2n \geq 6$, then
$\bar{g}(W)=g(W)$ as long as $\bar{g}(W) \geq 3$.  (The result of Kreck mentioned in Remark~\ref{rem:kreck} implies that, for $W$ simply-connected and of dimension $2n \geq 6$, $g(W)$ and $\bar{g}(W)$ are equal without assuming that $\bar{g}(W) \geq 3$.)

\subsection{Graded spaces}\label{sec:GradedSpaces}

In order to prove Theorem \ref{thm:Main:sec6}, it is convenient to treat all genera at once. In order to do so we will work directly with $\Mst$ instead of the individual spaces $\Mst_g$, and the following language will be convenient to keep track of things.

\begin{definition}
  A \emph{graded space} is a pair
  $(X, h_X)$ of a space $X$ and a continuous map $h_X : X \to \bN$. A
  (degree zero) map $f : (X, h_X) \to (Y, h_Y)$ of graded spaces is a
  continuous map $f : X \to Y$ such that $h_Y \circ f =
  h_X$.
  Similarly, a degree $k$ map is a continuous map $f : X \to Y$ such
  that $h_Y \circ f = h_X + k$. 

The homology of a graded space $(X,h_X)$ acquires an extra grading
$H_i(X) = \bigoplus_{n \geq 0} H_i(X)_n$, where $H_i(X)_n =
H_i(h_X^{-1}(n))$.  Maps of graded spaces respect this additional
grading (a degree $k$ map of graded spaces induces a map with a shift
of $k$). For a function $c : \bN \to \bN$, we say a degree $k$ map $f:
(X, h_X) \to (Y, h_Y)$ is \emph{$c$-connected} if for every $g$ the map $h_X^{-1}(g) \to
  h_Y^{-1}(g+k)$ is $c(g)$-connected.
\end{definition}

To put our definitions above in this framework, the stable genus defines a grading $\bar{g} : \Mst \to \bN$ on the space $\Mst$, and the stabilisation map \eqref{eq:StabMap} has degree 1 with respect to this grading.

\subsection{A semisimplicial resolution}\label{sec:resolution}

Recall that $S \subset [-1,0] \times \bR^\infty$ is a submanfold with collared boundary $\partial S = \{-1,0\} \times P$ which is diffeomorphic to $([-1,0] \times P) \# W_1$ relative to its boundary, so we may as well suppose that it is obtained from $[-1,0] \times P \subset [-1,0] \times
\bR^\infty$ by forming the ambient connect-sum with a disjoint copy of
$W_1$ along the disc $([-1,0] \times \partial c)(B_{1/4}(-1/2, 0, \ldots, 0))$, for some coordinate patch $\partial c : \bR^{2n-1}
\hookrightarrow P$.

  The coordinate patch $\partial c$ induces for each $M \in \Mst$ an embedding $c: (-\delta,0] \times \bR^{2n-1} \hookrightarrow M$ for some $\delta >
  0$ using the collar structure of $M$.


\begin{definition}
We define a semisimplicial space $X_\bullet = X_\bullet(W, P)$. Let
  $X_p$ be the set of pairs $(M,\phi)$ with $M \in \Mst$ and $\phi \in
  \overline{K}_p(M)$ (described in Definition \ref{defn:Kbar}),
  topologised as $\coprod_{[T]} (\mathcal{E}(T) \times \overline{K}_p(T))/\Diff_\partial(T)$, where the union is
  taken over 
  compact manifolds $T$ with $\partial T = P$ and $T$ stably diffeomorphic to $W$ relative to $P$, one in each diffeomorphism class relative to $P$. The
  face maps of the semisimplicial spaces $\overline{K}_\bullet(M)$
  induce face maps $d_i : X_p \to X_{p-1}$, so that $X_\bullet$ has
  the structure of a semisimplicial space augmented over $X_{-1} =
  \Mst$. By composing the augmentation map with $\bar{g} : \Mst \to
  \bN$, the augmented semisimplicial space $X_\bullet$ is a
  semisimplicial graded space.
\end{definition}

Theorem~\ref{thm:Main:sec6} follows by a rather standard spectral sequence argument (see \cite[Theorem IV.3.1]{Maazen} for an early reference) from the following properties of this resolution.

\begin{proposition}\label{prop:MainProperties}\mbox{}
\begin{enumerate}[(i)]
\item\label{it:prop:1} If $W$ is simply-connected, then the map $\vert X_\bullet \vert \to \Mst$,
  considered as a map of graded spaces, is
  $\lfloor\frac{g-2}{2}\rfloor$-connected.

\item\label{it:prop:2} For each $p \geq 0$ there is a commutative diagram
    \begin{equation}\label{eq:5}
      \begin{aligned}
      \xymatrix{
        {\Mst} \ar[d]_{s} \ar[r]^-{g_p}
        &
        X_p \ar[d]^{d_p}
        \\
        {\Mst}\ar[r]^-{g_{p-1}}
        & 
        X_{p-1}
      }
      \end{aligned}
    \end{equation}
    where $g_p$ is a weak homotopy equivalence of degree $p+1$, and
    $g_{p-1}$ is a weak homotopy equivalence of degree $p$.

\item\label{it:prop:3} The face maps $d_0, \dots, d_p: X_p \to X_{p-1}$ are all homotopic when precomposed with a CW approximation to $X_p$ (i.e.\ they are equal in the homotopy category).
\end{enumerate}
\end{proposition}

\begin{proof}[Proof of Theorem \ref{thm:Main:sec6}]
  The graded augmented semisimplicial space $X_\bullet \to \Mst$ gives
  rise to a tri-graded spectral sequence with
  \begin{equation*}
    E^1_{p,q, g} = H_q(X_p)_g \quad\text{ for }\quad p \geq -1,\, q \geq 0,\, g \geq 0,
  \end{equation*}
  where the $(p,q)$ grading is as usual, and the $g$ grading comes
  from the grading of the spaces.  The differential on $E^1$ is given
  by $d^1 = \sum_i (-1)^i (d_i)_*$, and in general $d^r$ has degree
  $(-r,r-1,0)$. Because the differentials do not change the $g$ grading, this is in fact just one spectral sequence for each $g$.
  The group $E^\infty_{p,q,g}$ is a subquotient of
  $H_{p+q+1}(\Mst,|X_\bullet|)_g$, and hence (\ref{it:prop:1}) implies
  that $E^\infty_{p,q,g} = 0$ for $p + q \leq \frac{g-4}2$.

The map
$$s_* : H_q(\Mst)_g \lra H_q(\Mst)_{g+1}$$
which we wish to show is an isomorphism for $q \leq \tfrac{g-3}{2}$ and an epimorphism for $q \leq \tfrac{g-1}{2}$ is, by (\ref{it:prop:2}), identified with $(d_0)_*: H_q(X_0)_{g+1} \to  H_q(X_{-1})_{g+1}$, which is the differential $E^1_{0,q,g+1} \to E^1_{-1,q,g+1}$ in the spectral sequence. The group $E^\infty_{-1,0, g+1}$ is the cokernel of this differential, and vanishes for $-1 \leq \tfrac{g+1-4}{2}$, so we deduce that this differential is surjective for $g \geq 1$. This proves the theorem for $g \leq 2$, providing the beginning of an induction argument.
  
Since all the face maps $d_i$ induce the same map in homology
by~(\ref{it:prop:3}), all but one of the terms in the alternating sum
in the differential $d^1 = \sum_i (-1)^i (d_i)_*: E^1_{2j,q,g+1} \to
E^1_{2j-1,q,g+1}$ cancel out, so this differential is $(d_{2j})_*:
H_q(X_{2j})_{g+1} \to H_q(X_{2j-1})_{g+1}$ which by~(\ref{it:prop:2})
is identified with the stabilisation map
$$s_* : H_q(\Mst)_{g-2j} \to H_q(\Mst)_{g-2j+1}.$$
By induction we can assume (for $j > 0$) that we have already proved the theorem for this map so for $j > 0$, the differential $d^1: E^1_{2j,q,g+1} \to E^1_{2j-1,q,g+1}$ is an epimorphism for $q \leq \frac{(g-2j)-1}2 = \frac{g-1}2 - j$ and an isomorphism for $q \leq \frac{g-3}2-j$.  In particular, all the groups $E^1_{p,q, g+1}$ in degrees where $p > 0$ and $p + q \leq \tfrac{g-1}{2}$ are killed by these differentials.

The induction hypothesis does not imply anything about $E^2_{0,q,g+1}$
and $E^2_{-1,q,g+1}$, but for degree reasons there is no room for a
$d^r$ differential for $r \geq 2$ whose target $E^r_{p,q,g+1}$ has $p + q \leq
\tfrac{g-3}{2}$.  Since $E^\infty_{p,q,g+1} = 0$ for such $(p,q,g)$,
we must have $E^2_{p,q,g+1} = 0$ for $p + q \leq \tfrac{g-3}{2}$ and
all $p \geq -1$.  It follows that $d^1: E^1_{0,q,g+1} \to
E^1_{-1,q,g+1}$ is an isomorphism if $q \leq \tfrac{g-3}{2}$ and an
epimorphism if $q-1 \leq \tfrac{g-3}{2}$.  This provides the induction
step.
\end{proof}

\subsection{Proof of Proposition \ref{prop:MainProperties}}
Part (\ref{it:prop:1}) is proved in Lemma~\ref{lem:proof:1} below.
Lemmas~\ref{lem:proof:2}, \ref{lem:proof:3a} and \ref{lem:proof:3b}
below, 
the latter two of which rely on Lemma~\ref{lem:proof:StabMap}, will
establish (\ref{it:prop:2}) and (\ref{it:prop:3}).

\begin{lemma}\label{lem:proof:1}
  If $W$ is simply-connected, then the map $\vert \epsilon \vert : \vert X_\bullet \vert \to \Mst$, considered as a map of graded spaces, is $\lfloor\frac{g-2}{2}\rfloor$-connected.
\end{lemma}
\begin{proof}
  The quotient map $\coprod_{[T]} \mathcal{E}(T) \to \Mst$ is a locally trivial
  fibre bundle, so all of the associated maps $X_p = \coprod_{[T]} (\mathcal{E}(T) \times \overline{K}_p(T))/\Diff_\partial(T) \to \Mst$ are too,
  and moreover $\Mst$ has a cover by open sets on which $X_p \to \Mst$
  is locally trivial for every $p$. Working in compactly generated
  spaces, it follows that $\vert \epsilon \vert$ is again a locally
  trivial fibre bundle, with fibre $\vert \overline{K}_\bullet(M)
  \vert$ over $M \in \Mst$. As we
  have supposed that $W$ is simply-connected and of dimension $2n \geq
  6$, $M$ is too, so this space is $\lfloor
  \tfrac{\overline{g}(M)-4}{2}\rfloor$-connected. The claim follows
  from the long exact sequence on homotopy groups.
\end{proof}

The manifold $S \subset [-1,0] \times \R^\infty$ used to define the
stabilisation map~\eqref{eq:StabMap} is a cobordism from $P$ to $P$, and
we shall need to consider its iterates.  Let us write
\begin{equation*}
  S_p = \bigcup_{i=0}^{p} (S - i \cdot e_1) \subset [-(p+1),0] \times \R^\infty
\end{equation*}
for the $(p+1)$-fold iteration, a manifold diffeomorphic to the
(ambient) connected sum of $[-(p+1),0] \times P$ and $p+1$ copies of
$W_1$. In this notation $S_0 = S$, and we will use these interchangeably.

\begin{definition}\label{defn:Yspace}
  Let $Y_{-1}(p)$ denote the set of
  submanifolds $N \subset [-(p+1),0] \times \bR^\infty$ such that
  \begin{enumerate}[(i)]
  \item $\partial N = N \cap (\{-(p+1),0\} \times \bR^\infty) =
    \{-(p+1),0\} \times P$, and the sets $(-\epsilon, 0] \times P$ and $[-(p+1), -(p+1)+\epsilon) \times P$ are contained in $N$ for some
    $\epsilon > 0$,

  \item $N$ is diffeomorphic to $S_p$ relative to its boundary.
  \end{enumerate}
  We topologise this space as a quotient space of $\Emb^\partial(S_p,
  [-(p+1),0] \times \bR^\infty)$, which is given the colimit topology analogously to \eqref{eq:6}.
\end{definition}

Any element $N \in Y_{-1}(p)$ gives rise to a map $\Mst \to \Mst$
defined in analogy with \eqref{eq:StabMap}, and there is a
continuous function $h: \Mst \times Y_{-1}(p) \to \Mst$ given by
\begin{equation*}
  h (M , N) = (M-(p+1)\cdot e_1) \cup N.
\end{equation*}
The manifold $S_p$ itself gives a point in $Y_{-1}(p)$, and $h(-,S_p)$
agrees with the $(p+1)$-fold iterate of \eqref{eq:StabMap}, but it
will be important to consider stabilisation with other manifolds.

The coordinate patch $c: \R^{2n-1} \to P$ induces a coordinate patch
$\R^{2n-1} \to \{0\} \times P \subset \partial N$ for any $N \in
Y_{-1}(p)$.  We may then define an augmented semisimplicial space
$Y_\bullet(p) \to Y_{-1}(p)$ in complete analogy with $X_\bullet \to
\Mst$: an element of $Y_q(p)$ is a pair $(N,x)$ where $N \in Y_{-1}(p)$
and $x = (\phi_0, \dots, \phi_q) \in K_q(N)$.  The map $h: \Mst \times
Y_{-1}(p) \to \Mst$ defined above extends to a map of semisimplicial
spaces $h: \Mst \times Y_\bullet(p) \to X_\bullet$ defined by
\begin{align*}
  h : \Mst \times Y_{q}(p) & \lra X_q\\
  (M ; N, \phi_0, \ldots, \phi_q) & \longmapsto ((M-(p+1)\cdot e_1) \cup
  N, \phi_0, \ldots, \phi_q).
\end{align*}

\begin{lemma}\label{lem:proof:StabMap}
  For any $y \in Y_p(p)$, the map $h(-,y): \Mst \to X_p$ is a weak homotopy  equivalence.
\end{lemma}
\begin{proof}
  For a point $y = (N, \phi_0, \ldots, \phi_p) \in Y_p(p)$ we
  consider the space $E$ consisting of pairs of an $M \in \Mst$ and an
  embedding $e: N \hookrightarrow M$ sending $\partial_\mathrm{out} N = \{0\} \times P$ to $P \subset M$ by $(0,p) \mapsto p$. We topologise $E$ as a subspace of $\Mst \times
  \Emb^{\partial_\mathrm{out}}(N, (-\infty, 0] \times \bR^\infty)$. As the
  inclusion
$$([-\epsilon, 0] \times P) \cup \left(\bigcup_{i=0}^p \phi_i(\Hmfld)\right) \lra N,$$
which is only defined for small enough $\epsilon$, has an inverse up to isotopy through embeddings,
we deduce that the map
\begin{align*}
E & \lra X_p\\
(M, e) & \longmapsto (M, e \circ \phi_0, \ldots, e \circ \phi_p)
\end{align*}
is a homotopy equivalence. On the other hand, the map
\begin{align*}
E & \lra \Emb^{\partial_\mathrm{out}}(N, (-\infty,0] \times \bR^\infty)\\
(M, e) & \longmapsto e
\end{align*}
is locally trivial (the map is $\Diff_\partial((-\infty,0] \times \bR^{\infty-1})$-equivariant, and by the parametrised isotopy extension theorem \cite[2.2.1 Th{\'e}or{\`e}me 5]{Cerf} the base is ``localement retractile" \cite[p.\ 239]{Cerf} with respect to this action, so by \cite[0.4.4 Lemme 2]{Cerf} the map is locally trivial)
over a contractible base space. The fibre over the canonical embedding $N \subset [-(p+1),0] \times \bR^\infty$ is the space of those $M \in \Mst$ which contain $N$, which is clearly weakly homotopy equivalent to $\Mst$.
\end{proof}

\begin{lemma}\label{lem:proof:2}
  Let $\phi_0, \dots, \phi_p: \Hmfld \to S_p$ be embeddings satisfying that
  the tuples $x_p = (S_p; \phi_0, \dots, \phi_p)$ and $x_{p-1} =
  (S_{p-1}; \phi_0, \dots, \phi_{p-1})$ define elements $x_p \in
  Y_p(p)$ and $x_{p-1} \in Y_{p-1}(p-1)$.  Then the diagram
    \begin{equation}\label{eq:12}
      \begin{aligned}
      \xymatrix{
        {\Mst} \ar[d]_-{s} \ar[rr]^-{h(-,x_p)}
        & &
        X_p \ar[d]^{d_p}
        \\
        {\Mst }\ar[rr]^-{h(-,x_{p-1})}
        & &
        X_{p-1}
      }
      \end{aligned}
    \end{equation}
    commutes.  Embeddings $\phi_i$ with this property
    exist.
\end{lemma}

\begin{proof}
  For the commutativity of the diagram, we just calculate
  \begin{align*}
    d_p(h(M, x_p)) &= d_p((M - (p+1)\cdot e_1) \cup S_{p} , \phi_0, \ldots, \phi_p)\\
    &=((M - (p+1)\cdot e_1) \cup S_{p} , \phi_0, \ldots, \phi_{p-1})\\
    &=(((M - e_1) \cup S - p \cdot e_1) \cup S_{p-1} , \phi_0, \ldots, \phi_{p-1})\\
    &=h(s(M), x_{p-1}).
  \end{align*}

  For the existence, we first note that the canonical
  embedding
  $\phi': W_{1,1} \hookrightarrow S$ induces, for each $i = 0, \dots,
  p$, an embedding
  \begin{equation*}
    \phi_i' = \phi' - i\cdot e_1: W_{1,1} \to S_p
  \end{equation*}
  with image in $S_p \cap ([-(i+1),-i] \times \R^\infty)$.  We
  extend these to disjoint embeddings $\phi_i$ of $\Hmfld$ by choosing a
  path from each $\phi'_i(\partial W_{1,1})$ to the point $c(0; 3i, 0,
  \ldots)$ in the coordinate patch $c$ and thickening it up.
  (Strictly speaking, this may not be possible if the path component
  of $P$ is orientable and the orientations induced by $\phi'$ and $c$
  are not compatible.  If this is the case we first change $\phi'$ by
  precomposing with an orientation-reversing diffeomorphism of
  $W_{1,1}$.)  These paths may clearly be chosen disjointly, and for
  $i < p$ they may be chosen with image in $S_{p-1}$.
\end{proof}

To establish property~(\ref{it:prop:3}) of
Proposition~\ref{prop:MainProperties} we must produce a homotopy of
maps into $X_{p-1}$.  The homotopy will be constructed in
Lemma~\ref{lem:proof:3a} and~\ref{lem:proof:3b} using the explicit
diffeomorphism provided by the following lemma.
\begin{lemma}\label{lem:SigmaDiffeo}
  If $i_0, i_1 : W_{1,1} \hookrightarrow W_{2,1}$ are disjoint
  orientation preserving embeddings into the interior of $W_{2,1}$
  then there is a diffeomorphism $\sigma$ of $W_{2,1}$ which restricts
  to the identity on the boundary and satisfies $\sigma \circ i_j =
  i_{1-j}$ for $j=0,1$.
\end{lemma}

\begin{proof}
  Let $\Sigma$ be the manifold obtained from the ball $B_{12}(0)
  \subset \bR^{2n}$ by cutting out the interiors of each of $B_1(\pm
  3e_1)$ and gluing in copies of $W_{1,1}$ along the boundaries (which
  are both canonically identified with $S^{2n-1}$). We will first
  construct a diffeomorphism $\sigma$ of $\Sigma$ which swaps the two
  copies of $W_{1,1}$, and then show that there is an embedding $e :
  \Sigma \hookrightarrow W_{2,1}$ restricting to $i_0$ and $i_1$ on
  the two copies of $W_{1,1}$. The required diffeomorphism is obtained
  by extending $\sigma$ from $\Sigma$ to $W_{2,1}$ by the identity.

  For $t \in [0,1]$, let $A(t)\in SO(2)$ be the rotation by $t \cdot
  \pi$, and let $B(t) \in SO(2n)$ be the block-diagonal matrix
  $\mathrm{diag}(A(t),I)$.  Let $r: [0,\infty) \to [0,1]$ be a smooth
  function with $r^{-1}(1) = [0,1]$ and $r^{-1}(0) = [2,\infty)$.
  Then the formula
  \begin{equation*}
    f(x) = (B\circ r(|x|))(x)
  \end{equation*}
  defines a ``half Dehn twist" diffeomorphism of $\bR^{2n}$, supported
  in $2D^{2n}$. Inside $D^{2n}$ it just rotates by $\pi$ in the first
  two coordinate directions. The function $g(x) = f(x+3e_1) +
  f(x-3e_1) - x$ does two of those half Dehn twists, supported in the
  balls $B_2(\pm 3 e_1)$, and the function $h(x) = 6f(x/6)$ does a
  half Dehn twist supported in the ball $B_{12}(0)$.  Finally, the map
  $\hat{\sigma} = h^{-1} \circ g$ is a diffeomorphism of $\bR^{2n}$
  supported in $B_{12}(0)$ which swaps the two balls $B_1(\pm 3e_1)$:
  on $B_1(-3e_1)$ it agrees with the map ``$+ 6e_1$" and vice
  versa. Hence $\hat{\sigma}$ restricts to a diffeomorphism $\sigma$
  of $\Sigma$ which is the identity on the boundary and swaps the two
  copies of $W_{1,1}$.

  The given $i_0$ and $i_1$ give an embedding into $W_{2,1}$ from the
  subspace $W_{1,1} \amalg W_{1,1} \subset \Sigma$.  To extend to an
  embedding $\Sigma \hookrightarrow W_{2,1}$, it suffices to extend to
  a neighbourhood of $(W_{1,1} \amalg W_{1,1}) \cup [-2e_1,2e_1]
  \subset \Sigma$.  The manifold $K = W_{2,1} \setminus
  \cup_{j=0,1}\mathrm{int}(i_j(W_{1,1}))$ is path-connected, so an
  extension over $[-2e_1,2e_1]$ exists.  This extension may be
  thickened to an embedding of a neighbourhood inside $\Sigma$ as we
  have assumed that the $i_j$ are both orientation preserving.
\end{proof}

\begin{lemma}\label{lem:proof:3a}
Let $y \in Y_p(p)$ be of the form $y = (S_p; \phi_0, \dots, \phi_p)$, and be such that
  the two elements of $Y_{p-1}(p)$ defined by
  \begin{align*}
    d_i(y) &= (S_p; \phi_0, \ldots, \hat{\phi}_i, \ldots, \phi_p)\\
    d_{i+1}(y) &= (S_p; \phi_0, \ldots, \hat{\phi}_{i+1}, \ldots,
    \phi_p)
  \end{align*}
  are in the same path component of $Y_{p-1}(p)$.  Then the two
  compositions
  \begin{equation*}
    \xymatrix{
      {\Mst} \ar[r]^-{h(-, y)}& X_p 
      \ar@/^/[r]^{d_i} \ar@/_/[r]_{d_{i+1}} 
      & X_{p-1}
    }
  \end{equation*}
  are homotopic.  
\end{lemma}
\begin{proof}
  If we pick a path $\gamma: [0,1] \to Y_{p-1}(p)$ with $\gamma(0) =
  d_i(y)$ and $\gamma(1) = d_{i+1}(y)$, then a homotopy can be defined as
  $s \mapsto h(-,\gamma(s))$.
\end{proof}

\begin{lemma}\label{lem:proof:3b}
  For each $i \in \{0, \dots, p-1\}$, there exists a $y \in Y_p(p)$ of the form $y=(S_p; \phi_0, \dots, \phi_p)$ 
such that $d_i(y)$ and $d_{i+1}(y)$ are in the same path component of $Y_{p-1}(p)$.
\end{lemma}
\begin{proof}
  Let us first note that for any isotopy of embeddings $h_t: S_p \to
  [-(p+1),0] \times \R^\infty$, constant near $\partial S_p$, starting
  at the canonical inclusion $\iota$ and ending at $\iota \circ
  \tau$ for a diffeomorphism $\tau: S_p \to S_p$, we get a loop $t
  \mapsto h_t(S_p) \in Y_{-1}(p)$ covered by the path
  \begin{equation}\label{eq:8}
    t \longmapsto (h_t(S_p); h_t \circ \phi_0, \dots, h_t \circ \phi_q)
  \end{equation}
  in $Y_q(p)$, starting at any given $(S_p; \phi_0, \dots, \phi_q) \in
  Y_q(p)$ and ending at $(S_p; \tau \circ \phi_0, \dots, \tau \circ
  \phi_q)$.

  For any $\tau \in \Diff_\partial(S_p)$, we may use Whitney's
  embedding theorem to choose an isotopy $h_t$ from $\iota$ to $\iota
  \circ \tau$, and this isotopy is unique up to isotopy.  It follows
  that the path~\eqref{eq:8} depends only on $\tau$, up to homotopy
  relative to $\partial I$.

  We then apply Lemma~\ref{lem:SigmaDiffeo} in the following way.  The
  images of the embeddings $\phi_i', \phi_{i+1}': W_{1,1} \to S_p$
  from the proof of Lemma \ref{lem:proof:2} may be enlarged and then
  joined by a thickened path to obtain a submanifold $B_i \subset S_p
  \cap ((-i-2,-i)\times \R^\infty)$ diffeomorphic to $W_{2,1}$ and
  disjoint from the images of $\phi_j'$ for $j \not\in\{i,i+1\}$.
  Applying Lemma~\ref{lem:SigmaDiffeo} to the embeddings $\phi_i',
  \phi_{i+1}': W_{1,1} \hookrightarrow B_i \cong W_{2,1}$, we obtain a
  diffeomorphism $\sigma_i: S_p \to S_p$ supported in $B_i$,
  satisfying $\sigma_i \circ \phi_i' = \phi_{i+1}'$ and $\sigma_i \circ
  \phi_{i+1}' = \phi_i'$.

  We first choose the $\phi_j$ for $j \neq \{i,i+1\}$ as in the proof of
  Lemma~\ref{lem:proof:2}, starting with $\phi'_j: W_{1,1} \to S_p$
  and extending by a thickening of a path to the point $c(0; 3j, 0,
  \ldots)$ in the coordinate patch $c$.  In doing so, we ensure that
  these paths are disjoint from each other and from $B_i$.  We then
  extend $\phi_{i+1}'$ to $\phi_{i+1}$, disjoint from previously
  chosen $\phi_j$.  The construction~\eqref{eq:8} now gives a path
  from $(S_p; \phi_0, \dots, \hat{\phi}_i, \dots, \phi_p)$ to the
  element $(S_p; \phi_0, \dots, \phi_{i-1}, \sigma_i \circ \phi_{i+1},
  \phi_{i+2}, \dots, \phi_p)$.  We can then choose an extension
  $\phi_i$ of $\phi_i'$ which is isotopic to $\sigma_i \circ \phi_{i+1}$, by an isotopy $g_t$ disjoint from previously chosen $\phi_j$.

  For $y = (S_p; \phi_0, \dots, \phi_p)$, the path~\eqref{eq:8}
  concatenated with the path arising from the isotopy $g_t$ starts at
  $d_i(y)$ and ends at $d_{i+1}(y)$.
\end{proof}

\begin{proof}[Proof of Proposition~\ref{prop:MainProperties}]
  Lemma~\ref{lem:proof:1} proves~(\ref{it:prop:1}).
  For~(\ref{it:prop:2}), we apply Lemma~\ref{lem:proof:2} and define the
  maps $g_p$ and $g_{p-1}$ of~\eqref{eq:5} as the maps $h(-,x_p)$ and
  $h(-,x_{p-1})$ of~\eqref{eq:12}.  These are weak homotopy equivalences by
  Lemma~\ref{lem:proof:StabMap}.

  Finally, for each $i \in \{0, \dots, p-1\}$
  Lemmas~\ref{lem:proof:3a} and \ref{lem:proof:3b} imply the existence
  of a diagram
  \begin{equation*}
    \xymatrix{
      {\Mst} \ar[r]^-{h(-, y)}_-\simeq& X_p 
      \ar@/^/[r]^{d_i} \ar@/_/[r]_{d_{i+1}} 
      & X_{p-1}
    }
  \end{equation*}
  where the first map is a weak homotopy equivalence by
  Lemma~\ref{lem:proof:StabMap}, and the two compositions $d_i \circ
  h(-,y)$ and $d_{i+1} \circ h(-,y)$ are homotopic.  As any CW approximation of $X_p$ may be lifted up to homotopy through the weak homotopy equivalence $h(-,y) : \Mst \to X_p$, it follows that the
  face maps $d_i, d_{i+1}: X_p \to X_{p-1}$ become homotopic when precomposed with a CW approximation of $X_p$, establishing~(\ref{it:prop:3}).
\end{proof}


\section{Tangential structures and abelian coefficient systems}\label{sec:TS}

In this section we shall improve Theorem \ref{thm:Main:sec6} in two ways. Firstly we shall consider moduli spaces of manifolds equipped with extra structure, and secondly we shall consider homology with coefficients in certain local coefficient systems.

Recall that a \emph{tangential structure} is a map $\theta : B \to BO(2n)$ with $B$ path-connected,
and a \emph{$\theta$-structure} on a $2n$-manifold $W$ is a bundle map
(fiberwise linear isomorphism) $\hat{\ell}_W : TW \to \theta^*\gamma_{2n}$, with underlying map $\ell_W : W \to B$.  We
shall write $\Bun^\theta(W)$ for the space of all such maps.  If $W$
has boundary $P$ equipped with a collar $(-\epsilon,0] \times P \to
W$, then the collar induces an isomorphism $\epsilon^1 \oplus TP \cong TW
\vert_P$.  If we fix a $\theta$-structure $\hat{\ell}_P : \epsilon^1 \oplus
TP \to \theta^*\gamma_{2n}$, then we may consider the subspace
$\Bun_\partial^\theta(W ; \hat{\ell}_P)\subset \Bun^\theta(W)$ consisting of bundle maps $\hat{\ell}_W : TW \to \theta^*\gamma_{2n}$ which restrict to $\hat{\ell}_P$ over the
boundary. Precomposition with the derivative of a diffeomorphism
induces an action of the group $\Diff_\partial(W)$ on the space
$\Bun_\partial^\theta(W; \hat{\ell}_P)$.

In Section \ref{sec:Resolutions}, we considered $\partial W \cong P \subset \R^\infty$ and used the
model $E\Diff_\partial(W) = \mathcal{E}(W)$, defined as the embedding space~\eqref{eq:6}, to construct the point set model $\mathcal{M}(W)
\simeq \MM(W) = B\Diff_\partial(W)$.  In this section, we shall need a similar
model for the homotopy quotient (alias Borel construction)
$\MM^\theta(W;\hat{\ell}_P) = \Bun_\partial^\theta(W;\hat{\ell}_P) \hcoker \Diff_\partial(W)$. We let
$$\mathcal{M}^\theta(W ; \hat{\ell}_{P}) = (\mathcal{E}(W) \times \Bun_\partial^\theta(W;\hat{\ell}_P))/\Diff_\partial(W).$$
As a set this may be described as pairs $(M,\hat{\ell})$ where $M
\in \mathcal{M}(W)$, and hence in particular $\partial M = \partial W = \{0\}
\times P$, and $\hat{\ell} : TM \to \theta^*\gamma_{2n}$ is a bundle map with
$\hat{\ell}\vert_{\partial M} = \hat{\ell}_{P}$.


Just as in Section \ref{sec:Resolutions}, it is convenient in this
section to consider all manifolds stably diffeomorphic to $W$ in one
go, so in analogy with Definition \ref{defn:ModuliSpaceStab} we define the
following larger moduli space of $\theta$-manifolds.

\begin{definition}\label{Defn:7.1}
  Let $P \subset \bR^\infty$ be a $(2n-1)$-manifold, equipped with a $\theta$-structure $\hat{\ell}_P$, and let $W$ be a manifold with boundary $P$. Let 
	$$\mathcal{M}^{\mathrm{st}, \theta}(\hat{\ell}_P) = \mathcal{M}^{\mathrm{st}, \theta}(W;\hat{\ell}_P) = \coprod_{[T]} \mathcal{M}^\theta(T ; \hat{\ell}_P),$$
where the union is taken over the set of compact manifolds with
$\partial T = P$ and $T$ stably diffeomorphic to $W$, one in each
diffeomorphism class relative to $P$. As a set this may be described as pairs
  of a $2n$-dimensional submanifold $M \subset (-\infty, 0] \times
  \bR^\infty$ and a $\theta$-structure $\hat{\ell}_M : TM \to
  \theta^*\gamma_{2n}$ such that
\begin{enumerate}[(i)]
\item $M \in \Mst(W)$,

\item $\hat{\ell}_M\vert_{P} = \hat{\ell}_P$.
\end{enumerate}
\end{definition}

The first difference which arises in the presence on $\theta$-structures is that there are potentially many $\theta$-structures on the cobordism $S \cong ([-1,0] \times P) \# W_{1}$, and furthermore these may restrict to different $\theta$-structures on $\{-1\} \times P$ and $\{0\} \times P$. We therefore require a discussion of the types of $\theta$-structure that $S$ should be allowed.

\begin{definition}\label{defn:7.2}
  Choose once and for all a bundle map $\tau: \bR^{2n} \to
  \theta^*\gamma_{2n}$ from the trivial $2n$-dimensional vector bundle over
  a point, or what is the same thing a basepoint $\tau \in
  \Fr(\theta^*\gamma_{2n})$. This determines a canonical $\theta$-structure
  on any framed $2n$-manifold (or $(2n-1)$-manifold); if $X$ is a  framed manifold we denote this $\theta$-structure by $\hat{\ell}^\tau_X$.
  
In \eqref{eq:StdEmbeding} we have defined a specific embedding $S^n \times D^n \hookrightarrow \bR^{2n}$, and hence obtained a framing $\xi_{S^n \times D^n}$ of $S^n \times D^n$. We will say that a $\theta$-structure on $S^n \times D^n$ is \emph{standard} if it is homotopic to $\hat{\ell}^\tau_{S^n \times D^n}$.

In \eqref{eq:barE} and \eqref{eq:barF} we defined embeddings $\overline{e},\overline{f}:
S^n \times D^n \to W_{1,1}$, and hence we obtain embeddings
\begin{equation*}
  \overline{e}_1, \overline{f}_1, \ldots, \overline{e}_g, \overline{f}_g : S^n \times D^n \lra W_{g,1}.
\end{equation*}
Let us say that a $\theta$-structure $\hat{\ell} : TW_{g,1} \to
\theta^*\gamma_{2n}$ on $W_{g,1}$ is \emph{standard} if all the pulled-back
structures $\overline{e}_i^*\hat{\ell}$ and $\overline{f}_i^*\hat{\ell}$ on $S^n \times D^n$ are standard.
\end{definition}

\begin{remark}\label{rem:StdVsAdmissible}
In Definition \ref{defn:admissible} we said that a $\theta$-structure $\hat{\ell}$ on $W_{1,1}$ was \emph{admissible} if there are orientation-preserving embeddings ${e},{f}: S^n \times D^n \to W_{1,1}$ with cores intersecting transversely in one point, such that each of the $\theta$-structures $e^*\hat{\ell}$ and $f^*\hat{\ell}$ on $S^n \times D^n$ extend to $\bR^{2n}$ for some orientation-preserving embeddings $S^n \times D^n \hookrightarrow \bR^{2n}$. We will now explain that if $\hat{\ell}$ is admissible then in fact there is an embedding $\phi : W_{1,1} \hookrightarrow W_{1,1}$ such that $\phi^*\hat{\ell}$ is standard.

This is a consequence of the following claim: if a $\theta$-structure $\hat{\ell}'$ on $S^n \times D^n$ extends to $\bR^{2n}$ for some orientation-preserving embedding $i : S^n \times D^n \hookrightarrow \bR^{2n}$, then there is a diffeomorphism $\varphi$ of $S^n \times D^n$ such that $\varphi^*\hat{\ell}'$ extends to $\bR^{2n}$ for the embedding \eqref{eq:StdEmbeding}.

We may isotope $i$ to have image disjoint from $\{0\} \times \bR^{n-1} \subset \bR^{n+1} \times \bR^{n-1} = \bR^{2n}$, by general position, and we may further ensure that $i$ has linking number 1 with this submanifold, by connect-summing the core of $i$ with a sphere isotopic to $S^n \times \{0\} \subset \bR^{n+1} \times \bR^{n-1}$. Then we may isotope $i$ to an embedding into the image $A$ of \eqref{eq:StdEmbeding} and $i : S^n \times D^n \hookrightarrow A$ will be a homotopy equivalence. The region $A \setminus \mathrm{int}\ i(S^n \times D^n)$ will therefore be a simply-connected $h$-cobordism, and hence $i$ is further isotopic to a diffeomorphism onto $A$. This proves the claim.
\end{remark}

We now choose a $\theta$-structure $\hat{\ell}_S$ on the cobordism $S  \cong ([-1,0] \times P) \# W_{1}$
 which is standard when pulled back along the canonical embedding $\phi' : W_{1,1} \to S$. Let us write $\hat{\ell}_P$ for its restriction to $\{0\} \times P \subset S$, and $\hat{\ell}'_P$ for its restriction to $\{-1\} \times P \subset S$. We therefore obtain a map
\begin{equation}\label{eq:StabMapTheta}
\begin{aligned}
  s = - \cup (S, \hat{\ell}_S) : \mathcal{M}^{\mathrm{st},\theta}(W;\hat{\ell}_P') &\lra \mathcal{M}^{\mathrm{st},\theta}(W;\hat{\ell}_P)\\
  (M, \hat{\ell}_M) &\longmapsto ((M - e_1) \cup S, \hat{\ell}_M \cup \hat{\ell}_S).
\end{aligned}
\end{equation}

Before stating our main result, we must define the analogue of the
function $\overline{g}$ for $\theta$-manifolds.  The na{\"i}ve definition, to define $\overline{g}(M,\hat{\ell}_M)$ as $\overline{g}(M)$, is not as
well behaved as the following: first define the \emph{$\theta$-genus} to be
$$g^\theta(M, \hat{\ell}_M) = \max\left\{g \in \bN \,\,\bigg|\,\, \parbox{18em}{there are $g$ disjoint copies of $W_{1,1}$ in $M$,\\ each with standard $\theta$-structure}\right\},$$
(by Remark \ref{rem:StdVsAdmissible} this agrees with the definition in the introduction) and then define the \emph{stable $\theta$-genus} to be
$$\bar{g}^\theta(M, \hat{\ell}_M) = \max\{g^\theta((M, \hat{\ell}_M) \natural k(W_{1,1}, \hat{\ell}_{W_{1,1}}))-k\,\,|\,\, k \in \bN \},$$
where the boundary connect-sum is formed with $k$ copies of $W_{1,1}$ each equipped with a standard $\theta$-structure $\hat{\ell}_{W_{1,1}}$. We use the function $\overline{g}^\theta$ is used to grade the spaces $\mathcal{M}^{\mathrm{st},\theta}(W;\hat{\ell}_P)$ and $\mathcal{M}^{\mathrm{st},\theta}(W;\hat{\ell}_P')$, so that the stabilisation map $s : \mathcal{M}^{\mathrm{st},\theta}(W;\hat{\ell}_P) \to \mathcal{M}^{\mathrm{st},\theta}(W;\hat{\ell}'_P)$ has degree 1. Our main theorem about this map is completely analogous to Theorem \ref{thm:Main:sec6}, but has a strong and a weak form, depending on whether $\theta$ has the following property.

\begin{definition}
  A tangential structure $\theta : B \to BO(2n)$ is \emph{spherical}
  if any $\theta$-structure on $D^{2n}$ extends to $S^{2n}$.
\end{definition}

This is a condition that we introduced in \cite[\S 5.1]{GR-W2}, and we
will refer there for some of its basic properties. Many tangential
structures of interest satisfy this condition, including all of those
which are pulled back from a fibration over $BO(2n+1)$, such as orientations and spin structures. Another example is $BU(3) \to BO(6)$, as $S^6$ admits an almost-complex structure, but in higher dimensions almost-complex structures are \emph{not} spherical.
 A notable tangential structure which is not spherical is that of a framing (corresponding to $EO(2n) \to BO(2n)$).

Finally, let us introduce a class of local coefficient systems. The spaces we consider are usually
disconnected and have no preferred basepoint, so local coefficients
are best defined as functors from the fundamental groupoid to the
category of abelian groups (or as bundles of abelian groups). Then an \emph{abelian coefficient system} is a coefficient system which has trivial monodromy along all commutators: in other words it has trivial monodromy along all nullhomologous loops.  Given a local coefficient system $\mathcal{L}$ on $\mathcal{M}^{\mathrm{st},\theta}(W;\hat{\ell}_P)$, we can define twisted homology with coefficients in $\mathcal{L}$, and $s$ induces a map
\begin{equation}\label{eq:3}
s_* :  H_k(\mathcal{M}^{\mathrm{st},\theta}(W;\hat{\ell}'_P); s^* \mathcal{L})_g \lra H_k(\mathcal{M}^{\mathrm{st},\theta}(W;\hat{\ell}_P); \mathcal{L})_{g+1}
\end{equation}
of twisted homology groups, where the subscripts denote the extra grading (as in Section \ref{sec:GradedSpaces}) on the homology of the graded spaces $(\mathcal{M}^{\mathrm{st},\theta}(W;\hat{\ell}'_P),\bar{g}^\theta)$ and $(\mathcal{M}^{\mathrm{st},\theta}(W;\hat{\ell}_P), \bar{g}^\theta)$.

\begin{theorem}\label{thm:main-theta}
  Suppose that $2n \geq 6$, and that $W$ is simply-connected.
  \begin{enumerate}[(i)]
  \item\label{it:main-theta:2} If $\mathcal{L}$ is abelian then the
    stabilisation map~\eqref{eq:3} is an epimorphism for
    $3k \leq g-1$ and an isomorphism for
    $3k \leq g-4$.
  \item\label{it:main-theta:1} If $\theta$ is spherical and
    $\mathcal{L}$ is constant, then the stabilisation map~\eqref{eq:3}
    is an epimorphism for $2k \leq g-1$ and an isomorphism
    for $2k \leq g-3$.
\end{enumerate}
\end{theorem}

In particular, if $\hat{\ell}_W \in \Bun^\theta_\partial(W; \hat{\ell}_P')$ and we write $\mathcal{M}^\theta(W, \hat{\ell}_W) \subset \mathcal{M}^\theta(W; \hat{\ell}_P') \subset \mathcal{M}^{\mathrm{st},\theta}(W;\hat{\ell}_P')$ for the path component containing $(W, \hat{\ell}_W)$, then \eqref{eq:StabMapTheta} restricts to a map
$$s : \mathcal{M}^\theta(W, \hat{\ell}_W) \lra \mathcal{M}^\theta(W \cup_P S, \hat{\ell}_W \cup \hat{\ell}_S)$$
and it follows from Theorem \ref{thm:main-theta} that this is an epimorphism or isomorphism on homology with (abelian) coefficients in a range of degrees depending on $\bar{g}^\theta(W, \hat{\ell}_W)$. This, and the fact that $g^\theta(W, \hat{\ell}_W) \leq \bar{g}^\theta(W, \hat{\ell}_W)$, proves Theorem \ref{thm:mainTheta}.

We shall describe those aspects of the proof of
Theorem~\ref{thm:main-theta} which differ from the proof of Theorem
\ref{thm:Main:sec6}. There are two main differences. Firstly a
slightly more elaborate analogue of $\overline{K}_\bullet(M)$ is
required to incorporate information about $\theta$-structures, and we
must also develop some basic tools for dealing with
$\theta$-structures on the manifolds $W_{g,1}$. Secondly, we must take
care that loops swept out by certain homotopies which we construct are
nullhomologous.

\subsection{$\theta$-structures on $W_{g,1}$ and $\Hmfld$}

Recall that $\Hmfld$ is the manifold obtained from $W_{1,1}$ by gluing on $[-1,0] \times D^{2n-1}$ along an orientation preserving embedding $\{-1\} \times D^{2n-1} \hookrightarrow \partial W_{1,1}$. In Remark \ref{rem:HFraming} we have explained that there is a framing of $\Hmfld$ which is standard on the image of $\bar{e}$ and $\bar{f}$ and extends the Euclidean framing on $\{0\} \times D^{2n-1} \subset \Hmfld$. We choose such a framing, and call it $\xi_H$. The associated $\theta$-structure $\hat{\ell}^\tau_H$ is therefore standard (in the sense of Definition \ref{defn:7.2}) when restricted to $W_{1,1} \subset \Hmfld$. We let $\xi_{W_{1,1}}$ be the framing on $W_{1,1} \subset \Hmfld$ induced by $\xi_H$, and $\xi_{W_{g,1}}$ be the framing on $W_{g,1}$ induced by boundary connect sum of $g$ copies of $(W_{1,1}, \xi_{W_{1,1}})$.

\begin{lemma}\label{lem:StrOnW11}
  The space of standard $\theta$-structures on $W_{1,1}$ (not fixed on
  the boundary) is path-connected.
\end{lemma}

\begin{proof}
  Let $\hat{\ell}$ and $\hat{\ell}'$ be two standard $\theta$-structures on
  $W_{1,1}$. Writing $E = \bar{e}(S^n \times D^n)$ and $F = \bar{f}(S^n \times D^n)$, the inclusion $E \cup F \hookrightarrow W_{1,1}$ is an isotopy equivalence, so it is enough
  to verify that they are homotopic when restricted to this
  subspace. The restrictions $\hat{\ell}'\vert_{E}$ and $\hat{\ell}\vert_{E}$ are both standard, so in particular are homotopic: choosing such a homotopy and extending to $E \cup F$, we see that we may change $\hat{\ell}$ and $\hat{\ell}'$ by homotopies so as to suppose that they are equal on $E$.
  
  Now the restrictions $\hat{\ell}'\vert_{F}$ and $\hat{\ell}\vert_{F}$ are equal on the contractible subspace $E \cap F \subset F$, and as they are standard they both extend over the contractible space $\bR^{2n}$ under the embedding \eqref{eq:StdEmbeding} (precomposed with $\bar{f}^{-1}$). Hence they are homotopic relative to $E \cap F$.
\end{proof}

\begin{lemma}\label{lem:StrOnH}
The space of $\theta$-structures on $\Hmfld$ which are standard on $W_{1,1} \subset \Hmfld$ and induced by the framing on $\{0\} \times D^{2n-1} \subset \Hmfld$ is path-connected.
\end{lemma}

\begin{proof}
Let $\hat{\ell}$ be a $\theta$-structure on $\Hmfld$ which is standard on $W_{1,1}$ and restricts to
$$\epsilon^1 \oplus TD^{2n-1} \lra \bR^{2n} \overset{\tau}\lra \theta^*\gamma_{2n}$$
on $\{0\} \times D^{2n-1} \subset \Hmfld$.  It is enough to show that there is a path of bundle maps from $\hat{\ell}$ to $\hat{\ell}^\tau_H$ which is constant over $\{0\} \times D^{2n-1}$. 

As the inclusion $W_{1,1} \hookrightarrow \Hmfld$ is an isotopy
  equivalence, by the previous lemma the $\theta$-structures $\hat{\ell}$
  and $\hat{\ell}^\tau_H$ are homotopic, and all that remains is to show
  that this homotopy may be taken to be constant over $\{0\} \times
  D^{2n-1}$. To do this, we first choose any path $\rho$ of $\theta$-structures
  from $\hat{\ell}$ to $\hat{\ell}^\tau_H$, and denote by $\gamma$ the loop in $\Fr(\theta^*\gamma_{2n})$ (based at $\tau$) obtained by restriction to $(0,0) \in \{0\} \times D^{2n-1}$. If this loop were nullhomotopic then by homotopy
  lifting along the restriction map we could modify $\rho$ to a path
  which is constant over $\{0\} \times D^{2n-1}$. 
  
  As $\gamma$ is not
  in general nullhomotopic, we compose $\rho$ with the loop based at
  $\hat{\ell}^\tau_H$ given by the framing $\xi$ and the loop of $\theta$-structures $\gamma^{-1}$ (which
  determines a loop of $\theta$-structures for any framed manifold) to
  obtain a new path $\rho'$. The loop given by restricting this path
  to $(0,0) \in \{0\} \times D^{2n-1}$ is now $\gamma \cdot
  \gamma^{-1}$, so nullhomotopic, and the construction described above
  now applies.
\end{proof}

\begin{corollary}\label{cor:ThetaStrOnWg}
  The space of standard
  $\theta$-structures on $W_{g,1}$ (not fixed on
  the boundary) is path-connected.
\end{corollary}
\begin{proof}
  $W_{g,1}$ is diffeomorphic to a manifold obtained from $D^{2n}$
  by gluing $g$ copies of $H$ to its boundary along $\{0\} \times
  D^{2n-1}$, and rounding corners. Hence this claim follows from the previous lemma, by
  first making any two standard $\theta$-structures be equal on the
  $D^{2n}$.
\end{proof}

\begin{lemma}\label{lem:ThetaStrOnW1}
  If $\theta$ is spherical, any standard $\theta$-structure on
  $W_{1,1}$ extends to the closed manifold $W_{1,1} \cup_{\partial W_{1,1}} D^{2n} \approx W_1$.
\end{lemma}
\begin{proof}
By Lemma \ref{lem:StrOnW11} it is enough to construct any $\theta$-structure on $W_1$ which restricts to a standard $\theta$-structure on $W_{1,1}$. Let $\xi$ denote the framing of $W_{1,1} \subset H$ constructed above, and $\hat{\ell}^\tau_{W_{1,1}}$ the $\theta$-structure associated to it.

Any framing $\zeta$ of $\epsilon^1 \oplus TW_{1,1}$ induces over the boundary a framing $\zeta_\partial$ of the  bundle  $\epsilon^1 \oplus TW\vert_{\partial W_{1,1}} \cong \epsilon^2 \oplus TS^{2n-1}$. If the framing $\zeta$ is changed by a map $\psi : W_{1,1} \to SO(2n+1)$ corresponding to elements $\alpha, \beta \in \pi_n(SO(2n+1))$ then $\zeta_\partial$ is changed by $[\alpha, \beta] \in \pi_{2n-1}(SO(2n+1))$, which is trivial as Whitehead products vanish in the homotopy groups of any $H$-space. Thus up to homotopy the framing $\zeta_\partial$ of $\epsilon^2 \oplus TS^{2n-1}$ is independent of $\zeta$, and hence extends over $D^{2n}$ for all $\zeta$ if it does for one. As $\epsilon^1 \oplus TW_1$ admits a framing, $\zeta_\partial$ must extend over $D^{2n}$ and it follows that any framing $\zeta$ of $\epsilon^1 \oplus TW_{1,1}$ extends over $W_1$. In particular the framing $\epsilon^1 \oplus \xi$ of $\epsilon^1 \oplus TW_{1,1}$ extends to a framing $\xi'$ of $\epsilon^1 \oplus TW_1$.

In \cite[Lemma 5.6]{GR-W2} we have shown that if $\theta$ is spherical then there is a commutative diagram
\begin{equation*}
\xymatrix{
B \ar[d]^-\theta \ar[r] & \bar{B} \ar[d]^-{\bar{\theta}}\\
BO(2n) \ar[r] & BO(2n+1)
}
\end{equation*}
which is $2n$-cartesian i.e.\ the induced map from $B$ to the homotopy pullback is $2n$-connected. Hence every $\bar{\theta}$-structure on $\epsilon^1 \oplus TW_1$ arises up to homotopy from a $\theta$-structure on $TW_1$. In particular, the $\bar{\theta}$-structure on $\epsilon^1 \oplus TW_1$ associated to (the stabilisation of) $\tau$ and the framing $\xi'$ gives a $\theta$-structure $\hat{\ell}_{W_1}$ on $W_1$. The $\theta$-structures $\hat{\ell}_{W_1} \vert_{W_{1,1}}$ and $\hat{\ell}^\tau_{W_{1,1}}$ become homotopic as $\bar{\theta}$-structures (they are both associated to the framing $\epsilon^1 \oplus \xi$), but $W_{1,1}$ only has cells of dimension $\leq n$ and the diagram is $2n$-cartesian, so they are also homotopic as $\theta$-structures: thus $\hat{\ell}_{W_1} \vert_{W_{1,1}}$ is standard.
\end{proof}

Recall that in the proof of Lemma \ref{lem:SigmaDiffeo} we constructed a particular
manifold $\Sigma$ diffeomorphic to $W_{2,1}$ containing two canonical copies of
$W_{1,1}$, and we constructed a diffeomorphism $\sigma \in
\Diff(\Sigma, \partial)$ interchanging these two copies.

\begin{lemma}\label{lem:SymmetricThetaStr}
If $\theta$ is spherical then any standard $\theta$-structure $\hat{\ell}$ on $\Sigma$ satisfies $\sigma^*\hat{\ell} \simeq \hat{\ell}$ relative to $\partial \Sigma$.
\end{lemma}
\begin{proof}
  We shall use some basic properties of spherical $\theta$-structures
  which we have developed in \cite[\S 5.1]{GR-W2}. Let us first prove the lemma for a special, highly symmetric, standard $\theta$-structure $\hat{\ell}'$ on $\Sigma$, which we shall construct using the special properties of spherical
  tangential structures.
	

  Recall from the proof of Lemma \ref{lem:SigmaDiffeo} that we
  constructed $\Sigma$ by starting with $B_{12}(0) \subset \bR^{2n}$,
  forming 
  $A = B_{12}(0) \setminus \mathrm{int}({B}_1(\pm 3e_1))$, and
  gluing in copies of $W_{1,1}$ along each $\partial B_1(\pm
  3e_1)$. The manifold $A$ has a canonical Euclidean framing, and we let
  $\hat{\ell}^\tau_A$ be the $\theta$-structure associated with this
  framing; it restricts to the same $\theta$-structure on each of
  $\partial B_1(\pm 3e_1)$. The $\theta$-structure
  $\hat{\ell}^\tau_A\vert_{\partial B_1(3e_1)}$ on $S^{2n-1} = \partial
  W_{1,1}$ extends to a $\theta$-structure $\hat{\ell}^\tau_{D^{2n}}$ on
  $D^{2n}$ (as the framing extends), and by Lemma
  \ref{lem:ThetaStrOnW1} there exists a $\theta$-structure
  $\hat{\ell}_{W_1}$ on $W_1$ which is standard on $W_{1,1}$. By
  \cite[Proposition 5.8]{GR-W2} we can connect-sum $(D^{2n},
  \hat{\ell}^\tau_{D^{2n}})$ and $(W_{1},\hat{\ell}_{W_1})$ as $\theta$-manifolds
  to obtain a $\theta$-structure $\hat{\ell}_{W_{1,1}}$ on $W_{1,1}$ which
  is both standard and agrees with $\hat{\ell}^\tau_A\vert_{\partial
    B_1(3e_1)}$ on $S^{2n-1} = \partial W_{1,1}$. Taking
  $\hat{\ell}_{W_{1,1}}$ on \emph{both} of the glued in copies of $W_{1,1}$
  we obtain a $\theta$-structure $\hat{\ell}'$ on $\Sigma$ which is
  standard.

  By construction, $\sigma^*\hat{\ell}'$ agrees with $\hat{\ell}'$ on $\partial
  \Sigma$ and on both copies of $W_{1,1}$. Thus we must show that
  $\sigma^*\hat{\ell}^\tau_A \simeq \hat{\ell}^\tau_A$ relative to $\partial
  A$. This will be the case if $\sigma$ preserves the Euclidean
  framing of $A$ up to homotopy relative to $\partial A$: it does, and
  we now describe an explicit homotopy. It is convenient to first work
  with the diffeomorphism $\hat{\sigma}$ of $\bR^{2n}$ constructed in
  Lemma \ref{lem:SigmaDiffeo}. We can define an isotopy of
  diffeomorphisms $\hat{\sigma}_s$ for $s \in [0,1]$ by replacing
  $A(t)$ by $A(st)$ in the definition of $f$, $g$, $h$, and
  $\hat{\sigma}$.  Then $\hat{\sigma}_0 = \mathrm{Id}$,
  $\hat{\sigma}_1 = \hat{\sigma}$, and $\hat{\sigma}_s$ always
  restricts to a parallel translation on each $B_1(\pm 3e_1)$.  Then,
  for each $x \in \bR^{2n}$, we get a path $s \mapsto
  (D\hat{\sigma}_s)(x) \in \mathrm{GL}(\bR^{2n})$ from the identity
  matrix to $(D\hat{\sigma})(x)$.  This path of matrices gives a path
  from the Euclidean framing of $\bR^{2n}$ to the pullback of the
  Euclidean framing along $\hat{\sigma}$, and is constant outside
  $\Int(A)$.
	
The following argument was suggested to us by Michael Weiss.	By Corollary \ref{cor:ThetaStrOnWg} there is a homotopy $\hat{\ell} \simeq \hat{\ell}'$, not fixed on the boundary, between the given standard $\theta$-structure and the one we have constructed. Let us write $\gamma : [0,1] \to \Bun^\theta(\Sigma)$ for this path from $\hat{\ell}$ to $\hat{\ell}'$, and $\gamma' : [0,1] \to \Bun^\theta_\partial(\Sigma)$ for the path from $\hat{\ell}'$ to $\sigma^*\hat{\ell}'$ relative to the boundary constructed above. The concatenated path
$$\hat{\ell} \overset{\gamma}\lra \hat{\ell}' \overset{\gamma'}\lra \sigma^*\hat{\ell}' \overset{\sigma^* \gamma^{-1}}\lra \sigma^*\hat{\ell}$$
maps under the restriction map $\rho : \Bun^\theta(\Sigma) \to \Bun^\theta(\partial \Sigma)$ to a loop based at $\hat{\ell}\vert_{\partial \Sigma}$ which is nullhomotopic. As the restriction map $\rho$ is a Serre fibration we may lift this nullhomotopy, and hence obtain a path from $\hat{\ell}$ to $\sigma^*\hat{\ell}$ relative to the boundary, as required.
\end{proof}

It remains to describe a version of the construction of $\sigma$ satisfying Lemma~\ref{lem:SymmetricThetaStr} which is suitable for a general tangential structure, without requiring the assumption that it is spherical. The following should be considered as a combined analogue of Lemmas \ref{lem:SigmaDiffeo} and \ref{lem:SymmetricThetaStr} in this case. Part (\ref{it:Braid:4}) of the lemma is only required for dealing with abelian local coefficient systems later on.

\begin{lemma}\label{lem:BraidMove}
Let $\hat{\ell}$ be a standard $\theta$-structure on $W_{3,1}$ and $i_1, i_2, i_3 : W_{1,1} \hookrightarrow W_{3,1}$ be disjoint orientation preserving embeddings into the interior of $W_{3,1}$ on which $\hat{\ell}$ is standard. Then there is a diffeomorphism $\rho$ of $W_{3,1}$ which restricts to the identity on the boundary, and which satisfies
\begin{enumerate}[(i)]

\item\label{it:Braid:2} $\rho \circ i_2 = i_1$,

\item\label{it:Braid:3} $\rho^*\hat{\ell} \simeq \hat{\ell}$
  relative to $\partial W_{3,1}$,
\item\label{it:Braid:4} the diffeomorphism $\rho$ and the homotopy
  $\gamma: \rho^*\hat{\ell} \simeq \hat{\ell}$ can be chosen so that
  the corresponding loop in
  $\mathcal{M}^\theta(W_{3,1};\hat{\ell}\vert_{\partial W_{3,1}})$ becomes
  nullhomologous in
  $\mathcal{M}^\theta(W_{5,1};\hat{\ell}\vert_{\partial W_{5,1}})$, where the
  $\theta$-structure $\hat{\ell}$ is extended to $W_{5,1}$ by forming
  the boundary connected sum with $W_{2,1}$ with a standard
  $\theta$-structure.
\end{enumerate}

\end{lemma}
\begin{proof}
By thickening up an embedded path between $i_1(W_{1,1})$ and $i_2(W_{1,1})$, we obtain an embedding $i_{12} : \Sigma \hookrightarrow W_{3,1}$ of the model $\Sigma$ of $W_{2,1}$ constructed in the proof of Lemma \ref{lem:SigmaDiffeo}. Doing the same with $i_2(W_{1,1})$ and $i_3(W_{1,1})$ gives an embedding $i_{23} : \Sigma \hookrightarrow W_{3,1}$, and if we choose the thickened paths disjointly we may suppose that the intersection of the images of $i_{12}$ and $i_{23}$ gives a regular neighbourhood of $i_2(W_{1,1})$.

The diffeomorphism $\sigma$ constructed in Lemma \ref{lem:SigmaDiffeo} can thus be extended to a diffeomorphism $\sigma_{12}$ of $W_{3,1}$ using $i_{12}$, and to another diffeomorphism $\sigma_{23}$ of $W_{3,1}$ using $i_{23}$. We then define $\rho = \sigma_{12}^{-1} \circ\sigma_{23}^{-1} \circ\sigma_{12} \circ \sigma_{23}$, and observe that (\ref{it:Braid:2}) is satisfied.

For the remaining properties, just as in the previous lemma we first prove them for a particular standard $\theta$-structure $\hat{\ell}'$. 
Let us construct a framing $\xi$ on $W_{3,1}$ such that $\rho^*\xi \simeq \xi$ relative to $\partial W_{3,1}$, and such that $\xi$ pulls back under each of $\overline{e}_1, \overline{f}_1, \ldots, \overline{e}_3, \overline{f}_3 : S^n \times D^n \to W_{g,1}$ to a framing homotopic to $\xi_{S^n \times D^n}$. Taking $\hat{\ell}'$ to be the $\theta$-structure associated to such a framing, it will be standard and property (\ref{it:Braid:3}) for $\hat{\ell}'$ will then be evident; this property for $\hat{\ell}$ then follows as in the end of the proof of Lemma \ref{lem:SymmetricThetaStr}.

Recall that we have chosen a framing $\xi_{W_{1,1}}$ of $W_{1,1}$. We shall take the framing $\xi$ on $W_{3,1}$ to be given by $\xi_{W_{1,1}}$ on each $i_j(W_{1,1})$, and extend by a choice of framing $\xi_0$, which we shall be slightly more specific about shortly, on the set
$$A= W_{3,1} \setminus \Int\left(\bigcup_{j=1}^3 i_j(W_{1,1})\right) \approx D^{2n} \setminus \Int\left(\coprod_{j=1}^3 D^{2n}\right).$$

The diffeomorphisms $\sigma_{12}$, $\sigma_{23}$, and $\rho$ of $W_{3,1}$ induce diffeomorphisms $\hat{\sigma}_{12}$, $\hat{\sigma}_{23}$, and $\hat{\rho}$ of $A$ which permute the boundaries $i_j(\partial W_{1,1})$ and fix the boundary $\partial W_{3,1}$. If $\xi$ is a framing of $A$ agreeing with $\xi_0$ on $\partial A$, we shall write $[\xi]$ for its class modulo homotopy of framings fixed on $\partial A$.  We shall study the effect of the diffeomorphisms $\sigma_{12}$, $\sigma_{23}$ and $\rho$ and their compositions on the set of such $[\xi]$.
The framings $\hat{\sigma}_{23}^*[\xi_0]$ and $[\xi_0]$ differ by the homotopy class of a map
$$\delta_{23} : (A, \partial A) \lra (SO(2n), *),$$
and we similarly define $\delta_{12}$. We claim that $\delta_{23}$ and $\delta_{12}$ are homotopic to maps with support inside a ball $D^{2n} \subset A$. As $A$ may be obtained from $\partial A$ by attaching three embedded and normally framed 1-cells and a $2n$-cell, this is equivalent to saying that $\delta_{23}$ and $\delta_{12}$ are nullhomotopic when restricted to each of the three 1-cells. If we take the 1-cells to be attached to the same point of $\partial W_{3,1}$ and equivalent points of the $i_j(\partial W_{1,1})$ as shown in Figure \ref{fig:2}, and choose $\xi_0$ so that its restrictions to each of these 1-cells are homotopic relative to the end points of the 1-cells (when the restrictions of $TA$ to the 1-cells are identified using the normal framings), then it is clear that the framings $\hat{\sigma}_{23}^*\xi_0$ and $\xi_0$ are homotopic on the 1-cells, because up to isotopy the diffeomorphism $\hat\sigma_{23}$ simply permutes these 1-cells and their normal framings (this is not true for $n=1$ where they are ``braided'', but we have assumed $2n \geq 6$).  This shows that $\delta_{23}$ is homotopic to a map supported inside a ball, and the same argument applies for $\delta_{12}$.

\begin{figure}[h]
  \begin{center}
    \includegraphics[bb=0 0 88 88]{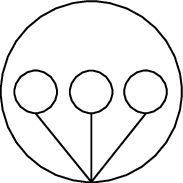}
  \end{center}
  \caption{The manifold $A$, obtained from $\partial A$ by attaching three 1-cells and a $2n$-cell.}
  \label{fig:2}
\end{figure}

In particular the support of $\delta_{12}$ can be made disjoint from the support of $\hat{\sigma}_{23}$, from which it follows that the homotopy classes of framings $\hat{\sigma}_{23}^*\hat{\sigma}_{12}^*[\xi_0]$ and $[\xi_0]$ differ by the product $[\delta_{23}] \cdot [\delta_{12}]$ of the commuting elements
$$[\delta_{23}], [\delta_{12}] \in [(A, \partial A), (SO(2n),*)].$$
Continuing in this way, we find that $\hat{\rho}^*[\xi_0]$ and $[\xi_0]$ differ by $[\delta_{23}] \cdot [\delta_{12}] \cdot [\delta_{23}]^{-1} \cdot [\delta_{12}]^{-1} = 0 \in [(A, \partial A), (SO(2n),*)]$. Therefore there is a homotopy of framings $\hat{\gamma} : \hat{\rho}^*\xi_0 \simeq \xi_0$ relative to $\partial A$, and gluing in three copies of $(W_{1,1}, \xi_{W_{1,1}})$ gives a homotopy $\gamma: \rho^*\xi \simeq \xi$ relative to $\partial W_{3,1}$, as required.

To establish property (\ref{it:Braid:4}), let $A'= D^{2n} \setminus \Int\left(\coprod_{j=1}^5 D^{2n}\right)$ be given a cell structure relative to $\partial A'$ analogous to that of Figure~\ref{fig:2}, with five normally framed 1-cells and a $2n$-cell. Let it be given a framing $\xi'_0$ which agrees with $\xi_{W_{1,1}}\vert_{\partial W_{1,1}}$ on each of the boundaries of the small discs, and such that the framings induced on each of the five 1-cells are homotopic as before. For each triple $a,b,c \in \{1,2,3,4,5\}$ we obtain an embedding $i_{abc}: A \to A'$ sending the boundaries labeled $1,2,3$ to those labeled $a,b,c$ respectively, and sending the three 1-cells of $A$ to the three corresponding 1-cells. Under this embedding, the diffeomorphism $\rho$ of $A$ constructed above extends to a diffeomorphism $\hat{\rho}_{abc}$ of $A'$, and there is a corresponding homotopy $\hat{\gamma}_{abc} : \hat{\rho}_{abc}^*\xi'_0 \simeq \xi'_0$.

In the symmetric group $\Sigma_5$ we have
$$(123) = (253)^{-1}(412)^{-1}(253)(412)$$
and hence the diffeomorphism of $A'$ defined as
$$\hat \psi = (\hat{\rho}_{412}^{-1}\hat{\rho}_{253}^{-1}\hat{\rho}_{412}\hat{\rho}_{253})\hat{\rho}_{123}$$
restricts to the identity on $\partial A'$.  The corresponding paths of framings $\hat{\gamma}_{abc}$ may be glued to induce a path of framings $\hat \gamma: \hat \psi^* \xi_0' \simeq \xi_0'$, which together defines a loop $(\hat \psi,\hat \gamma)$ the space $\mathcal{M}^{\mathrm{fr}}(A';\xi'_0\vert_{\partial A'})$, where $\mathrm{fr}: EO(2n) \to BO(2n)$ corresponds to the structure of a framing.  The loop is based at a point determined by the framing $\xi'_0$ on $A'$.
Now the five 1-cells of $A'$ must be fixed up to isotopy by $\hat{\psi}$, under our assumption $2n \geq 6$, so we may isotope $\hat{\psi}$ (and homotope $\hat{\gamma}$) so that they are fixed. Furthermore, the path $\hat{\gamma}$ gives a trivialisation of the differential $D\hat{\psi}$ along each 1-cell, so we may isotope $\hat{\psi}$ so that it is the identity on the normal bundle of each of the 1-cells, and homotope $\hat{\gamma}$ so that it is constant over each 1-cell. We have therefore changed the loop $(\hat{\psi}, \hat{\gamma})$ by a homotopy so that it gives a loop in $\mathcal{M}^{\mathrm{fr}}(A';\xi'_0\vert_{\partial A'})$ which is supported in a disc, which we may take to be inside $i_{123}(A)$ but disjoint from its 1-cells. Gluing five copies of $(W_{1,1}, \xi_{W_{1,1}})$ in to form $(W_{5,1}, \xi')$, we obtain loops $x_{abc} = (\rho_{abc},\gamma_{abc})$ in $\mathcal{M}^{\mathrm{fr}}(W_{5,1};\xi'\vert_{\partial W_{5,1}})$ based at $\xi'$, and hence elements $x_{abc} = [(\rho_{abc},\gamma_{abc})] \in \pi_1(\mathcal{M}^{\mathrm{fr}}(W_{5,1};\xi'\vert_{\partial W_{5,1}}), \xi')$.  We find that the commutator $[x_{253},x_{412}]$ satisfies
$$[(\rho_{123}, \gamma_{123})] = [x_{253},x_{412}] \cdot [(\psi, \gamma)] \in \pi_1(\mathcal{M}^{\mathrm{fr}}(W_{5,1};\xi'\vert_{\partial W_{5,1}}), \xi').$$
Hence we may re-choose $(\rho_{123},\gamma_{123})$ homotopic to the concatenation $(\rho_{123}, \gamma_{123})\cdot (\psi, \gamma)^{-1}$ representing an element of $\pi_1(\mathcal{M}^{\mathrm{fr}}(W_{3,1};\xi\vert_{\partial W_{3,1}}), \xi)$ which becomes a commutator in $\pi_1(\mathcal{M}^{\mathrm{fr}}(W_{5,1};\xi'\vert_{\partial W_{5,1}}),\xi')$ and hence nullhomologous. As usual, property (\ref{it:Braid:4}) for the standard $\theta$-structure $\hat{\ell}'$ coming from taking the $\theta$-structure induced by the framing follows.

To deduce property (\ref{it:Braid:4}) for $\hat{\ell}$, note that a choice of homotopy $\hat{\ell} \simeq \hat{\ell}'$ restricted to the boundary gives a $\theta$-structure on the cylinder $S^{2n-1} \times [0,1]$ restricting to $\hat{\ell}\vert_{\partial W_{3,1}}$ and $\hat{\ell}'\vert_{\partial W_{3,1}}$. Gluing on this cylinder to $W_{3,1}$, and identifying the resulting manifold with $W_{3,1}$ again, gives a homotopy equivalence
$$\mathcal{M}^\theta(W_{3,1};\hat{\ell}\vert_{\partial W_{3,1}}) \lra \mathcal{M}^\theta(W_{3,1};\hat{\ell}'\vert_{\partial W_{3,1}}),$$
and by the same method a compatible homotopy equivalence for $W_{5,1}$, from which property (\ref{it:Braid:4}) for $\hat{\ell}$ follows from that of $\hat{\ell}'$ proved above.
\end{proof}

Finally, we give a result which is not logically necessary for the results of this paper, but which may clarify the notion of standardness of a $\theta$-structure on $W_{g,1}$.

\begin{proposition}\label{prop:reframing}
Let $2n \geq 6$, and $\hat{\ell}$ be a $\theta$-structure on $W_{g,1}$ such that the underlying map $\ell : W_{g,1} \to B$ is nullhomotopic.
\begin{enumerate}[(i)]
\item For any $n$ there is an embedding $\phi : W_{g-1,1} \hookrightarrow W_{g,1}$ such that $\phi^*\hat{\ell}$ is standard,

\item If $n \neq 3,7$ then there is an embedding $\phi : W_{g,1} \hookrightarrow W_{g,1}$ such that $\phi^*\hat{\ell}$ is standard.
\end{enumerate}
\end{proposition}
\begin{proof}
As the underlying maps $\ell$ and $\ell^\tau_{W_{g,1}}$ are homotopic, there is a path of bundle maps from $\hat{\ell}$ to an $\hat{\ell}'$ whose underlying map is equal to $\ell^\tau_{W_{g,1}}$. There is therefore a bundle isomorphism $\rho : TW_{g,1} \to TW_{g,1}$, with underlying map the identity, such that $\hat{\ell}'= \hat{\ell}^\tau_{W_{g,1}} \circ \rho$. Thus $\hat{\ell}'$ is also a $\theta$-structure associated to a framing of $W_{g,1}$, namely the framing $\xi' = \xi_{W_{g,1}} \circ \rho$. To establish the lemma it is therefore enough to show that there is a framed embedding $\phi : (W_{h,1}, \xi_{W_{h,1}}) \hookrightarrow (W_{g,1}, \xi')$, with $h$ being either $g-1$ or $g$ depending on the case.

In order to do so it is convenient to work with a variation of the quadratic module described in Definition \ref{defn:ImmSphereQuadMod}. Namely, the framing $\xi'$ is a section $\xi' : W_{g,1} \to \Fr(W_{g,1})$ of the frame bundle, and hence gives an injective homomorphism $\pi_n(W_{g,1},w) \to \pi_n(\Fr(W_{g,1}),\xi'(w)) = I_n^{\mathrm{fr}}(W_{g,1})$, where $w \in W_{g,1}$ is any chosen basepoint.  Since $W_{g,1}$ is simply-connected, any two choices of $w$ will give canonically isomorphic groups and henceforth we shall omit $w$ and $\xi'(w)$ from the notation.  In this proof (only) we shall write $\mathsf{I}_n^{\mathrm{fr}}(W_{g,1}, \xi')$ for the quadratic module given by $\pi_n(W_{g,1})$ with the bilinear form $\lambda$ and quadratic function $\mu$ induced from $\mathsf{I}_n^{\mathrm{fr}}(W_{g,1})$. By Smale--Hirsch theory, $\pi_n(W_{g,1})$ is interpreted as the set of regular homotopy classes of \emph{compatibly framed} immersions of $(S^n \times D^n, \xi_{S^n \times D^n})$ into $(W_{g,1}, \xi')$.

Now, if $\sH^{\oplus h} \to \mathsf{I}_n^{\mathrm{fr}}(W_{g,1}, \xi')$ is a morphism of quadratic modules, then just as in the proof of Lemma \ref{lemthm:conn-K-delta} we may use Smale--Hirsch theory and the Whitney trick to find embeddings $e_i, f_i : S^n \times D^n \hookrightarrow W_{g,1}$, for $i=1,2,\ldots,h$, which pull back $\xi'$ to $\xi_{S^n \times D^n}$, and so that the cores of the $e_i$ and $f_i$ intersect as they do in $W_{h,1}$. By plumbing these together we obtain an embedding $\phi: W_{h,1} \hookrightarrow W_{g,1}$ such that $\phi^*\xi' \simeq \xi_{W_{h,1}}$.

It therefore remains to show that the quadratic module $\mathsf{I}_n^{\mathrm{fr}}(W_{g,1}, \xi')$ admits a morphism from $\sH^{\oplus g-1}$, and from $\sH^{\oplus g}$ if $n \neq 3,7$. This will use the classification of quadratic modules with form parameter either $(1, \{0\})$ or $(-1, 2\bZ)$. First note that the bilinear form $(\pi_n(W_{g,1}), \lambda)$ is $(-1)^n$-symmetric, and non-degenerate by the Hurewicz theorem and Poincar{\'e} duality. 

If $n$ is even, the quadratic function $\mu$ satisfies $2\mu(x) =\lambda(x,x)$ so is determined by $\lambda$, and it is enough to work with $\lambda$. The symmetric form $\lambda$ is the intersection form of $W_{g,1}$, so is isomorphic to $\sH^{\oplus g}$, which establishes the lemma in this case.

If $n$ is odd then the skew-symmetric form $(\pi_n(W_{g,1}), \lambda)$ must be isomorphic to $\sH^{\oplus g}$ as a bilinear form, but $\mu$ is a potentially non-standard quadratic structure. However, such quadratic refinements of a skew-symmetric form are classified by their Arf invariant, which may be supported inside a single copy of $\sH$, leaving a morphism of quadratic forms $\sH^{\oplus g-1} \to \mathsf{I}_n^{\mathrm{fr}}(W_{g,1}, \xi')$. Finally, we claim that if $n \neq 3,7$ then the Arf invariant of $\mathsf{I}_n^{\mathrm{fr}}(W_{g,1}, \xi')$ is zero, so that it is indeed $\sH^{\oplus g}$. Let us write $\mu^{\xi'} = \mu$ for the quadratic function constructed using the framing $\xi'$, and $\mu^{\xi_{W_{g,1}}}$ for the quadratic function constructed using the standard framing. The Arf invariant of $\mu^{\xi_{W_{g,1}}}$ is certainly zero, and $\xi'$ is obtained from $\xi_{W_{g,1}}$ by reframing using $\rho$, which is the same as a map $\rho : W_{g,1} \to O(2n)$. Jones--Rees have studied the effect of reframing on the Arf invariant, and in our setting their result \cite[p.\ 144]{JonesRees} says that $\mathrm{Arf}(\mu^{\xi'}) + \mathrm{Arf}(\mu^{\xi_{W_{g,1}}}) = \mu^{\xi'}(y) \in \bZ/2$ where $y \in H^n(W_{g,1};\bZ/2)$ is the class
$$W_{g,1} \overset{\rho}\lra O(2n) \overset{\text{stab.}}\lra O \overset{-1}\lra O \overset{\Omega v_{n+1}}\lra K(\bZ/2,n),$$
and $v_{n+1} \in H^{n+1}(BO;\bZ/2)$ is the $(n+1)$st Wu class. Finally, $y=0$ for $n \neq 3,7$, as the Hurewicz map $\pi_{n+1}(BO) \to H_{n+1}(BO;\bZ/2)$ is zero for $n+1\neq1,2,4,8$ by \cite[Theorem 1]{AtiyahHirzebruch}.
\end{proof}

\begin{remark}
If $(M, \hat{\ell}_M)$ is a $\theta$-manifold containing a submanifold
diffeomorphic to $W_{g,1}$ on which the underlying map $\ell_M : M \to B$ is nullhomotopic, then it follows from Proposition \ref{prop:reframing} that $g^\theta(M, \hat{\ell}_M) \geq g-1$, and that $g^\theta(M, \hat{\ell}_M) \geq g$ if $n \neq 3,7$. In particular, if a noncompact manifold $(M, \hat{\ell}_M)$ contains infinitely many copies of $W_{1,1}$ on which $\ell_M$ is nullhomotopic, then it also contains infinitely many copies of $W_{1,1}$ with standard $\theta$-structure.

This allows us to correct a minor omission in the proof of Proposition
7.8 in \cite{GR-W2}. Indeed, assumption (iii) of \cite[Proposition
7.8]{GR-W2} now implies that $K\vert_{[i,\infty)}$ contains
infinitely many
disjoint copies of $W_{1,1}$ with standard $\theta$-structure, and the
copies of $W_{1,1}$ used in the proof of Proposition 7.8 should be
taken to have standard $\theta$-structure.
\end{remark}

\subsection{A semisimplicial resolution}

We now return to the proof of Theorem~\ref{thm:main-theta}.  We shall first define an analogue of the resolution $X_\bullet \to \Mst$ from Section~\ref{sec:resolution}.

Recall that as part of the data in Theorem~\ref{thm:main-theta} we
have specified a $\theta$-structure $\hat{\ell}_S$ on the cobordism $S \approx
([-1,0] \times P) \# W_1$ which becomes standard when pulled back along
$\phi' : W_{1,1} \hookrightarrow S$ and induces the $\theta$-structure $\hat{\ell}_P$ on $\{0\} \times P$ and $\hat{\ell}_P'$ on $\{-1\} \times P$. The coordinate patch $\partial c: \R^{2n-1} \to P$ determines a coordinate patch $ c = \mathrm{Id}\times \partial c: (-\delta,0] \times \R^{2n-1} \hookrightarrow S$ near $\{0\} \times P \subset S$.  We shall make the following assumptions on $\partial c$ and
$\hat{\ell}_S$:
\begin{itemize}
\item The embeddings $c: (-\delta,0] \times
  \R^{2n-1} \hookrightarrow S$ and $\phi': W_{1,1} \hookrightarrow S$
  have image in the same path component.
\item The $\theta$-structure on $(-\delta,0] \times \R^{2n-1}$
  obtained by pulling back $\hat{\ell}_S$ along $c$ is \emph{equal} to the structure
  determined (according to Definition~\ref{defn:7.2}) by the canonical
  framing of $(-\delta,0] \times \R^{2n-1} \subset \R^{2n}$.
\end{itemize}
After possibly changing $\hat{\ell}_S$ (and therefore $\hat{\ell}_P$) by a small
homotopy, it is always possible to choose $\partial c: \R^{2n-1} \to P$ with
these two properties.

For a compact manifold $X$ equipped with a chart $d : (-\epsilon, 0] \times \bR^{2n-1} \to X$ for some $\epsilon > 0$ such that $d^{-1}(\partial X) = \{0\} \times \bR^{2n-1}$, in Definition~\ref{defn:K-p} we have defined a semisimplicial space $K_\bullet(X) = K_\bullet(X, d)$. We now want to define an analogue of this semi-simplicial space for $\theta$-manifolds.


There is a $(0, e_0) \in K_0(S) = K_0(S,c)$, where the embedding $e_0 : H \hookrightarrow S$ given by choosing an embedded path from $W_{1,1} \subset S$ to $(0, c(0)) \in (-\delta,0] \times c(\bR^{2n-1})$ and thickening it up. We may do this so that there is an $\epsilon>0$ such that on $(-\epsilon,0] \times D^{2n-1} \subset \Hmfld$ the embedding $e_0$ is given by $e_0(s,p) = c(s,p)$. We can extend this to a 1-parameter family of embeddings $e_t : \Hmfld \to S$ for $t \in \bR$ such that on $(-\epsilon,0] \times D^{2n-1} \subset \Hmfld$ it is given by $e_t(s,p)=c(s,p+te_1)$, and hence define a 1-parameter family $\hat{\ell}_H(t) = e_t^*\hat{\ell}_S$ of $\theta$-structures on $\Hmfld$.

Recall that we write $\mathcal{M}^{\mathrm{st},\theta}(\hat{\ell}_P)$ for the space from Definition~\ref{Defn:7.1}.  This space depends on the manifold $W$, but since $W$ shall stay fixed for a while we omit this dependence from the notation.

\begin{definition}\label{defn:K-p-theta}
  Let $(M, \hat{\ell}_M) \in \mathcal{M}^{\mathrm{st},\theta}(\hat{\ell}_P)$ and  $\epsilon > 0$ be such that
  $(-\epsilon, 0] \times P \subset M$.  Then $\partial c: \R^{2n-1} \to P$
  induces a chart $(-\epsilon,0] \times \R^{2n-1} \to M$, and so we have $K_\bullet(M)$.
\begin{enumerate}[(i)]
\item Let $K_0(M, \hat{\ell}_M)$ be the space of tuples $(t, \phi, \nu)$, where $(t, \phi) \in K_0(M)$ and $\nu$ is a path in $\Bun^\theta(\Hmfld)$ from $\phi^*\hat{\ell}_M$ to $\hat{\ell}_H(t)$ which is constant over $\{0\} \times D^{2n-1} \subset \Hmfld$.

\item Let $K_p(M, \hat{\ell}_M) \subset (K_0(M, \hat{\ell}_M))^{p+1}$ consist of those tuples which give an element of $K_p(M)$ after forgetting the paths $\nu$.

\item Topologise $K_p(M, \hat{\ell}_M)$ as a subspace of $K_p(M) \times (\Bun^\theta(\Hmfld)^I)^{p+1}$, and write $K_p^\delta(M, \hat{\ell}_M)$ for the same set considered as a discrete space. The collection $K_\bullet(M, \hat{\ell}_M)$ forms a semisimplicial space.

\item Let $\overline{K}_\bullet(M, \hat{\ell}_M) \subset {K}_\bullet(M, \hat{\ell}_M)$ be the sub-semisimplicial space consisting of those tuples $((t_0, \phi_0, \nu_0), \ldots, (t_p, \phi_p, \nu_p))$ where the $\phi_i$ are disjoint.

\item Let $K^\delta(M, \hat{\ell}_M)$ be the simplicial complex with vertices $K_0^\delta(M, \hat{\ell}_M)$, and where the set $\{(t_0, \phi_0, \nu_0), \ldots, (t_p, \phi_p, \nu_p)\}$ is a $p$-simplex if $\{(t_0, \phi_0), \ldots, (t_p, \phi_p)\}$ is a $p$-simplex of $K^\delta(M)$.
\end{enumerate}
\end{definition}

The main result concerning this modified semisimplicial space is the following analogue of Corollary \ref{cor:KOverlineConn}. Note that $\bar{g}(M, \hat{\ell}_M) \leq \bar{g}(M)$, so the connectivity range in the following proposition is potentially smaller than that of Corollary \ref{cor:KOverlineConn}.

\begin{proposition}\label{prop:ThetaCxConnected}
  If $2n \geq 6$ and $M$ is simply-connected then
  $\vert\overline{K}_\bullet(M, \hat{\ell}_M)\vert$ is $\lfloor
  \tfrac{\bar{g}(M, \hat{\ell}_M)-4}{2}\rfloor$-connected.
\end{proposition}
\begin{proof}
  We will explain the analogue of Lemma \ref{lemthm:conn-K-delta},
  that $\vert K^\delta(M, \hat{\ell}_M)\vert$ is $\lfloor \tfrac{\bar{g}(M,
    \hat{\ell}_M)-4}{2}\rfloor$-connected: passing from this to the claim in
  the proposition is exactly as in Theorem \ref{thm:high-conn} and
  Corollary \ref{cor:KOverlineConn}.

Recall
from Definition \ref{defn:ImmSphereQuadMod} that for a manifold $M$, with framed basepoint $b_M$, $I_n\hatfr(M) = I_n\hatfr(M,b_M)$ denotes the group of regular homotopy classes of framed immersions $i : S^n \times D^n \looparrowright M$ equipped with a path in $\Fr(M)$ from $Di(b_{S^n \times D^n})$ to $b_M$. In this proof we shall write $I_n\hatfr(M, \hat{\ell}_M) \subset I_n\hatfr(M)$ be
  the subgroup of those regular homotopy classes of immersions $i : S^n \times D^n \looparrowright M$ (together with
  paths in $\mathrm{Fr}(W)$) such that the $\theta$-structure
  $\hat{\ell}_M \circ Di$ on $S^n \times D^n$ is standard.  As
  we explained in Definition~\ref{defn:ImmSphereQuadMod},
  Smale--Hirsch theory identifies $I_n\hatfr(M)$ with the
  group $\pi_n(\mathrm{Fr}(M))$; under this identification
  $I_n\hatfr(M,\hat{\ell}_M)$ corresponds to the kernel of the
  homomorphism $\pi_n(\mathrm{Fr}(M)) \to
  \pi_n(\mathrm{Fr}(\theta^*\gamma_{2n}))$ induced by $\hat{\ell}_M$. The bilinear form $\lambda$ and quadratic
  function $\mu$ restrict to this subgroup giving a quadratic module $\mathsf{I}_n\hatfr(M,\hat{\ell}_M) = (I_n\hatfr(M,\hat{\ell}_M), \lambda, \mu)$, and as in the proof of
  Lemma \ref{lemthm:conn-K-delta} we have a map of simplicial
  complexes
  \begin{equation*}
    K^\delta(M, \hat{\ell}_M) \lra K^a(\mathsf{I}_n\hatfr(M, \hat{\ell}_M)).
  \end{equation*}

  An embedding $j: W_{g,1} \hookrightarrow M$ gives rise to elements
  $e_1, f_1, \ldots, e_g, f_g \in I_n\hatfr(M)$ determining
  $g$ hyperbolic summands. If in addition $j^* \hat{\ell}_M$ is standard,
  these elements lie in the subgroup $I_n\hatfr(M, \hat{\ell}_M)$.
  Hence we have $\bar{g}(\mathsf{I}_n\hatfr(M, \hat{\ell}_M))
  \geq \bar{g}(M, \hat{\ell}_M)$, and so $\lCM(K^a(\mathsf{I}_n\hatfr(M,
  \hat{\ell}_M))) \geq \lfloor \tfrac{\bar{g}(M,
    \hat{\ell}_M)-1}{2}\rfloor$, by Theorem \ref{thm:Charney}.

  We now proceed precisely as in the proof of Lemma
  \ref{lemthm:conn-K-delta}. The procedure explained there gives a
  lift $\hat{f}:I^{k+1} \to |K^\delta(M)|$ of $I^{k+1} \overset{f}\to
  \vert K^a(\mathsf{I}_n\hatfr(M, \hat{\ell}_M)) \vert \to
  \vert K^a(\mathsf{I}_n\hatfr(M))\vert$.  To upgrade
  this to a lift  $I^{k+1} \to \vert K^\delta(M ,\hat{\ell}_M) \vert$ we use
  Lemma~\ref{lem:StrOnH} to choose for each internal vertex $v$ of
  $I^{k+1}$ with $\hat{f}(v) = (t, \phi)$ a path $\nu: I \to
  \Bun^\theta(\Hmfld)$ making $(t,\phi,\nu)$ a vertex of
  $K^\delta(M,\hat{\ell}_M)$.  Whether or not an unordered $(p+1)$-tuple of
  elements $(t_i,\phi_i,\nu_i) \in K_0^\delta(M,\hat{\ell}_M)$
  forms an element of $K_p^\delta(M,\hat{\ell}_M)$ does not depend on the
  $\nu_i$, so we have produced a lift $I^{k+1} \to
  |K^\delta(M,\hat{\ell}_M)|$.
\end{proof}

\begin{remark}\label{remark:estimating-genus}
The Smale--Hirsch lifting argument in the proof above shows that $g(M,\hat\ell_M) = g(\mathsf{I}_n\hatfr(M,\hat\ell_M))$.  When the structure map $\ell_M: M \to B$ is $n$-connected (see \cite[Section 9]{HomStabII} for a discussion of how the general situation may be reduced to this case by a Moore--Postnikov argument), the genus of $(M,\hat\ell_M)$ may be effectively estimated in the following way.  If we write $\pi_{n+1}(\mathrm{Fr}(\theta^* \gamma_{2n}),\mathrm{Fr}(M))$ for the relative homotopy groups of the mapping cylinder of $\mathrm{Fr}(M) \to \mathrm{Fr}(\theta^* \gamma_{2n})$ relative to $\mathrm{Fr}(M)$, then $I_n\hatfr(M, \hat\ell_M)$ is the image of $\pi_{n+1}(\mathrm{Fr}(\theta^* \gamma_{2n}),\mathrm{Fr}(M)) \to \pi_{n}(\mathrm{Fr}(M))$ and the group $\pi_{n+1}(\mathrm{Fr}(\theta^* \gamma_{2n}),\mathrm{Fr}(M))$ becomes a quadratic module by composing with the map to $I_n\hatfr(M, \hat\ell_M)$.  The Hurewicz theorem gives an isomorphism $\pi_{n+1}(\mathrm{Fr}(\theta^* \gamma_{2n}),\mathrm{Fr}(M)) \cong H_{n+1}(B,M)$ and hence there is an exact sequence
  \begin{equation*}
    \pi_{n+1}(\mathrm{Fr}(\theta^* \gamma_{2n}),\mathrm{Fr}(M)) \lra H_n(M) \xrightarrow{(\ell_M)_*} H_n(B) \lra 0,
  \end{equation*}
  where the first map preserves intersection pairing.  If the bilinear form $(H_n(M;\bZ),\lambda)$ contains $g = g(H_n(M),\lambda)$ orthogonal hyperbolic forms, and the abelian group $H_n(B;\bZ)$ is generated by $e$ elements, then the kernel of $(\ell_M)_*$ contains at least $g-e$ orthogonal hyperbolic forms.  Therefore $\pi_{n+1}(\mathrm{Fr}(\theta^* \gamma_{2n}),\mathrm{Fr}(M))$ and hence $\mathsf{I}_n\hatfr(M,\hat\ell_M)$ contains at least $g-e$ orthogonal hyperbolic forms when disregarding the quadratic form $\mu$.  It follows that the quadratic module contains at least $g-e-1$ orthogonal hyperbolic forms.  We obtain the estimate
  \begin{equation*}
    g(M,\hat\ell_M) \geq g(H_n(M;\bZ),\lambda) - e - 1,
  \end{equation*}
  whose right hand side can be expressed as a constant depending only
  on $n$ and $B$ plus a term depending on $M$ only through characteristic
  numbers (Euler characteristic and when $n$ is even also signature).
  If $n$ is even then $\lambda$ determines $\mu$ and the estimate can
  be improved to $g(H_n(M;\bZ),\lambda) - e$.  If $n=3$ or $n=7$ then
  the kernel of
  $I_n\hatfr(M,\hat\ell_M) \to
  H_n(M;\bZ)$
  contains an element with $\mu=1$ whence we also obtain
  $g(M,\hat\ell_M) \geq g(H_n(M;\bZ),\lambda) - e$.
\end{remark}

Using the semisimplicial space $\overline{K}_\bullet(M, \hat{\ell}_M)$, we may define an augmented
semisimplicial space $X^\theta_\bullet \to \mathcal{M}^{\mathrm{st},\theta}(\hat{\ell}_P)$ analogous to
$X_\bullet \to \Mst$ from Section~\ref{sec:resolution}. We let $X_p^\theta$ be the set of
tuples $(M, \hat{\ell}_M ; x)$ where $(M, \hat{\ell}_M) \in \mathcal{M}^{\mathrm{st},\theta}(\hat{\ell}_P)$ and $x \in
\overline{K}_p(M,\hat{\ell}_M)$. We topologise $X_p^\theta$ as a quotient
space of a subspace of
$$\coprod_{[T]} \mathcal{E}(T) \times \Bun^\theta_\partial(T ; \hat{\ell}_P) \times K_p(T) \times (\Bun^\theta(\Hmfld)^I)^{p+1},$$
where the union is taken over the set of compact manifolds with
boundary $P$ stably diffeomorphic to $W$, one in each diffeomorphism
class.  Again $\mathcal{E}(T)$ is the embedding
space model~\eqref{eq:6} for $E\Diff_\partial(T)$.  This forms an augmented
semisimplicial space
$\epsilon^\theta : X_\bullet^\theta \to \mathcal{M}^{\mathrm{st},\theta}(\hat{\ell}_P)$, as
before.

We now wish to construct augmented semisimplicial spaces analogous to the  $Y_\bullet(p) \to Y_{-1}(p)$ from Definition~\ref{defn:Yspace}. First, choose a $\theta$-structure $\hat{\ell}_{S_p}^\mathrm{std}$ on the manifold $S_p \subset [-(p+1),0] \times \bR^\infty$ which agrees with $\hat{\ell}_P$ on $\{0\} \times P$, and which on the $(p+1)$ copies of $W_{1,1}$ is standard. Let us write $\hat{\ell}_P^{(p)}$ for the induced $\theta$-structure on $\{-(p+1)\} \times P$. Choosing these inductively, we may suppose that
\begin{enumerate}[(i)]
\item $\hat{\ell}_{S_0}^\mathrm{std} = \hat{\ell}_S$, so $\hat{\ell}_P' = \hat{\ell}_P^{(0)}$, and

\item  $(S_p, \hat{\ell}_{S_p}^\mathrm{std}) \subset (S_{p+1}, \hat{\ell}_{S_{p+1}}^\mathrm{std})$, so their difference is a cobordism from $(P, \hat{\ell}_P^{(p+1)})$ to $(P, \hat{\ell}_P^{(p)})$.
\end{enumerate}

We then construct $Y^\theta_\bullet(p) \to Y_{-1}^\theta(p)$
completely analogously to $Y_\bullet(p) \to Y_{-1}(p)$ from
Definition~\ref{defn:Yspace}, but where all manifolds $N$ are equipped
with $\theta$-structures $\hat{\ell}_N$ which make them cobordisms from $(P, \hat{\ell}_P^{(p)})$ to $(P, \hat{\ell}_P)$, and the data $(t_i, \phi_i)$ are equipped with paths of
bundle maps from $\phi_i^* \hat{\ell}_N$ to $\hat{\ell}_H(t_i)$. 

We shall establish the analogues of the properties given in Proposition
\ref{prop:MainProperties} for $\epsilon^\theta: X^\theta_\bullet \to
\mathcal{M}^{\mathrm{st},\theta}(\hat{\ell}_P)$, by proving analogues of Lemmas \ref{lem:proof:1}, \ref{lem:proof:StabMap}, \ref{lem:proof:2}, and \ref{lem:proof:3b}.

\begin{lemma}\label{lem:7A}
If $W$ is simply-connected, then the map $\vert X_\bullet^\theta \vert \to \mathcal{M}^{\mathrm{st},\theta}(\hat{\ell}_P)$, considered as a map of graded spaces, is
  $\lfloor\frac{g-2}{2}\rfloor$-connected.
\end{lemma}

\begin{proof}
It is probably no longer the case that the maps $X_p^\theta \to \mathcal{M}^{\mathrm{st},\theta}(\hat{\ell}_P)$ are locally trivial. However, the geometric fibre of $\vert X_\bullet^\theta \vert \to \mathcal{M}^{\mathrm{st},\theta}(\hat{\ell}_P)$ over $(M, \hat{\ell}_M)$ is $\vert\overline{K}_\bullet(M, \hat{\ell}_M)\vert$, which is $\lfloor
  \tfrac{\bar{g}(M, \hat{\ell}_M)-4}{2}\rfloor$-connected by Proposition \ref{prop:ThetaCxConnected}, so it will be enough to show that $\vert X_\bullet^\theta \vert \to \mathcal{M}^{\mathrm{st},\theta}(\hat{\ell}_P)$ is a quasifibration.

For an augmented semisimplicial space $X_\bullet \to X_{-1}$ and a point $x \in X_{-1}$, the map
$$\vert \mathrm{hofib}_x(X_\bullet \to X_{-1}) \vert \lra \mathrm{hofib}_x(\vert X_\bullet \vert \to X_{-1})$$
is a weak homotopy equivalence. (Supposing that all constructions are formed in the category of $k$-spaces, and that $X_{-1}$ is weak Hausdorff, this follows from the fact that the functor $p^* : \mathrm{Top}/X_{-1} \to \mathrm{Top}/{P_xX_{-1}}$ induced by the path fibration $p : P_x X_{-1} \to X_{-1}$ admits a right adjoint, cf.\ \cite[Proposition 2.1.3]{MaySigurdsson}.) If each $X_p \to X_{-1}$ is a quasifibration, it therefore follows that $\vert X_\bullet \vert \to X_{-1}$ is too.

We must therefore show that each $X_p^\theta \to \mathcal{M}^{\mathrm{st},\theta}(\hat{\ell}_P)$ is a quasifibration. It is enough to show this when restricted to individual path components of the base, so consider the space $Z$ defined by the pullback square
\begin{equation*}
\xymatrix{
Z \ar[r] \ar[d]& X_p^\theta \ar[d]\\
\Emb_\partial(M, (-\infty,0] \times \bR^\infty) \times \Bun^\theta_\partial(TM;\hat{\ell}_P)  \ar[r] & \mathcal{M}^{\mathrm{st},\theta}(\hat{\ell}_P).
}
\end{equation*}
The bottom map is a principal $\Diff_\partial(M)$-bundle onto the path components which it hits. In particular it is a fibration so this square is also a homotopy pullback, and hence the map between the vertical homotopy fibres is a weak homotopy equivalence. Therefore the right-hand vertical map is a quasifibration if and only if the left-hand vertical map is. 

We may identify $Z$ as the space of tuples $(e,\hat\ell_M,x)$, where  $e : M \hookrightarrow (-\infty,0] \times \bR^\infty$ is an embedding; $\hat{\ell}_M$ is a $\theta$-structure on $M$ extending $\hat{\ell}_P$; and $x \in \overline{K}_p(M, \hat{\ell}_M)$. From this point of view, we see that the left-hand vertical map is the pullback of the map
$$\pi : Z' \lra \Bun^\theta_\partial(TM;\hat{\ell}_P)$$
where $Z'$ is the space of pairs $(\hat \ell_M,x)$, where $\hat{\ell}_M$ is a $\theta$-structure on $M$ extending $\hat{\ell}_P$ and $x \in \overline{K}_p(M, \hat{\ell}_M)$. It is therefore enough to show that $\pi$ is a Serre fibration.

But this is clear, as $\pi$ has canonical path-lifting. Concretely, to lift a path $\hat{\ell}_s$ of $\theta$-structures on $M$ starting at $(\hat{\ell}_0;(t_0, \phi_0, \nu_0), \ldots, (t_p, \phi_p, \nu_p)) \in Z'$ we let $V_i(s)$ denote the concatenation
$$c_i^* \hat{\ell}_s \overset{c_i^*\hat{\ell}_{[0,s]}}\leadsto c_i^* \hat{\ell}_0 \overset{\nu_i}\leadsto \hat{\ell}_H(t_i)$$
rescaled so that the first portion takes time $\tfrac{s}{2}$ and the last portion takes time $1-\tfrac{s}{2}$. Then $s \mapsto (\hat{\ell}_s; t_0, c_0, V_0(s), \ldots, t_p, c_p, V_p(s))$ lifts $\hat{\ell}_s$, and is continuous in all the data.
\end{proof}

The map $h$ from Section \ref{sec:Resolutions} now takes the form
$$h : \mathcal{M}^{\mathrm{st},\theta}(\hat{\ell}_P^{(p)}) \times Y_{q}^\theta(p) \lra X_q^\theta,$$
but is given by the analogous formula, accounting for $\theta$-structures and paths of $\theta$-structures.

\begin{lemma}\label{lem:7B}
  For any $y \in Y_p^\theta(p)$, the map $h(-,y): \mathcal{M}^{\mathrm{st},\theta}(\hat{\ell}_P^{(p)}) \to X^\theta_p$ is a weak homotopy  equivalence.
\end{lemma}
\begin{proof}
  Let $y = (N, \hat{\ell}_N ; (\phi_0, \nu_0), \ldots, (\phi_p, \nu_p))$. Let $E^\theta$ denote the space consisting of tuples $(M, \hat{\ell}_M ;
  e, \nu)$ where $(M, \hat{\ell}_M) \in \mathcal{M}^\theta(\hat{\ell}_P)$, $e : N \hookrightarrow M$ is
  an embedding, and $\nu$ is a path from $e^*\hat{\ell}_M$ to $\hat{\ell}_N$
  through bundle maps which are fixed over $\{0\} \times P \subset
  N$. We topologise $E^\theta$ as a subspace of $\mathcal{M}^{\mathrm{st},\theta}(\hat{\ell}_P) \times \Emb(N,
  (-\infty,0] \times \bR^\infty) \times \Bun^\theta(N)^I$. The proof is then concluded in the same way
  as the proof of Lemma~\ref{lem:proof:StabMap}, using $E^\theta$
  instead of $E$.
\end{proof}

\begin{lemma}\label{lem:7C}
Let $\phi_0, \ldots, \phi_p : \Hmfld \to S$ be embeddings and $\nu_0, \ldots, \nu_p \in \Bun^\theta(\Hmfld)^I$ be paths satisfying that the tuples
$$x_p = (S_p, \hat{\ell}_{S_p}; (\phi_0,\nu_0), \dots, (\phi_p,\nu_p)) \quad x_{p-1} = (S_{p-1}, \hat{\ell}_{S_{p-1}};(\phi_0, \nu_0), \dots,
  (\phi_{p-1},\nu_{p-1}))$$
	define elements $x_p \in
  Y^\theta_p(p)$ and $x_{p-1} \in
  Y^\theta_{p-1}(p-1)$. Then the diagram
    \begin{equation*}
      \begin{aligned}
      \xymatrix{
        {\mathcal{M}^{\mathrm{st},\theta}(\hat{\ell}_P^{(p)})} \ar[d]_{s} \ar[rr]^-{h(-, x_p)}
        & &
        X_p^\theta \ar[d]^{d_p}
        \\
        {\mathcal{M}^{\mathrm{st},\theta}(\hat{\ell}_P^{(p-1)})}\ar[rr]^-{h(-, x_{p-1})}
        & &
        X_{p-1}^\theta
      }
      \end{aligned}
    \end{equation*}
		commutes. Embeddings $\phi_i$ and paths $\nu_i$ with this property exist.
\end{lemma}
\begin{proof}
	Commutativity of the diagram in the lemma
  is proved in the same way as Lemma~\ref{lem:proof:2}.  For existence, we construct the embeddings $\phi_i$ in exactly the
  same way as Lemma~\ref{lem:proof:2} and then appeal to Lemma~\ref{lem:StrOnH} for the existence
  of the $\nu_i$.
\end{proof}

The step corresponding to Lemma~\ref{lem:proof:3b} now has two versions, depending on whether $\theta$ is spherical or not, given in Lemmas \ref{lem:A:proof:3b} and \ref{lem:NonSph}. It is the one step in the argument where the presence of $\theta$-structures adds more than bookkeeping.

\begin{lemma}\label{lem:A:proof:3b}
Suppose that $\theta$ is spherical.  For each $i \in \{0, \dots, p-1\}$ there exists an element $y \in Y_p^\theta(p)$ of the form $y = (S_p,\hat{\ell}_{S_p}; (\phi_0,\nu_0), \dots,
  (\phi_p,\nu_p))$, such that the two elements
  \begin{align*}
    d_i(y) &= (S_p,\hat{\ell}_{S_p}; (\phi_0,\nu_0), \ldots,
    \widehat{(\phi_i,\nu_i)}, \ldots, (\phi_p,\nu_p))\\
    d_{i+1}(y) &= (S_p,\hat{\ell}_{S_p}; (\phi_0,\nu_0), \ldots,
    \widehat{(\phi_{i+1},\nu_{i+1})}, \ldots, (\phi_p,\nu_p))
  \end{align*}
  are in the same path component of $Y_{p-1}^\theta(p)$.
\end{lemma}
\begin{proof}
  For any choice of $y$, the construction in the proof of
  Lemma~\ref{lem:proof:3b} gives a path starting at $d_i(y)$ and ending
  at
  \begin{equation}\label{eq:10}
    (S_p, (\sigma_i^{-1})^*\hat{\ell}_{S_p};(\sigma_i \circ \phi_0,\nu_0), \dots,
    \widehat{(\sigma_i \circ \phi_{i}, \nu_{i})}, \dots,
    (\sigma_i \circ \phi_p,\nu_p)).
  \end{equation}
  
	Recall that in the proof of Lemma \ref{lem:proof:3b} we constructed a certain submanifold $B_i \subset S_p \cap ((-i-2, -i) \times \bR^\infty)$ which is diffeomorphic to $W_{2,1}$, and which contains the images of the embeddings $\phi'_i, \phi'_{i+1} : W_{1,1} \to S_p$ and is disjoint from the images of the remaining $\phi'_i$. We also constructed a diffeomorphism $\sigma_i$ of $S_p$, supported in the interior of $B_i$, which interchanges $\phi'_i$ and $\phi'_{i+1}$.
	
  By Lemma \ref{lem:SymmetricThetaStr}, and our assumption that
  $\theta$ is spherical, there is a homotopy  $\sigma_i^*\hat{\ell} \simeq \hat{\ell}$ relative to $\partial B_i$. Gluing ($(\sigma_i^{-1})^*$ applied to) this into $S_p$ gives a path of $\theta$-structures $k_t$, constant over the complement of
  $B_i$, starting at $\hat{\ell}_{S_p}$ and ending at
  $(\sigma_i^{-1})^*\hat{\ell}_{S_p}$. 
This gives rise to a path
  from~\eqref{eq:10} to
  \begin{equation}\label{eq:11}
    (S_p, \hat{\ell}_{S_p}; (\sigma_i \circ \phi_0,\nu_0'), \dots,
    \widehat{(\sigma_i \circ \phi_{i}, \nu_{i}')}, \dots,
    (\sigma_i \circ \phi_p,\nu_p')),
  \end{equation}
  where $\nu_j'$ is the concatenation $(t \mapsto (\sigma_i \circ \phi_{j})^* k_t) \cdot \nu_{j}$.

  We have constructed for any $y$ a loop in $Y^\theta_{-1}(p)$,
  covered by a path in $Y^\theta_{p-1}(p)$, starting at $d_i(y)$ and
  ending at~\eqref{eq:11}.  If we choose the $\phi_j$ as in the proof
  of Lemma~\ref{lem:proof:3b}, we will have $(\sigma_i \circ
  \phi_j,\nu_j') = (\phi_j,\nu_j)$ for $j \not \in\{i,i+1\}$, since
  the supports of $\sigma_i$ and $k_t$ are disjoint from the image of
  these $\phi_j$.  The isotopy $g_t$ from the proof of
  Lemma~\ref{lem:proof:3b} gives a path from~\eqref{eq:11} to
  \begin{equation*}
    (S_p, \hat{\ell}_{S_p}; (\phi_0,\nu_0), \dots,
    (\phi_{i-1},\nu_{i-1}), (\phi_i, \nu_i''),
    (\phi_{i+2}, \nu_{i+2}), \dots, (\phi_p,\nu_p)),
  \end{equation*}
  where $\nu_i'' = (t \mapsto g_t^* \hat{\ell}_{S_p}) \cdot \nu'_{i+1}$.
  If we choose $\nu_j$ arbitrarily for $j \neq i$, and define $\nu_i =
  \nu_i''$, then we have constructed a path from $d_i(y)$ to $d_{i+1}(y)$.
\end{proof}

The following plays the role of Lemma \ref{lem:A:proof:3b} in the case of tangential structures which are not necessarily spherical. The additional property will be used for dealing with abelian coefficient systems.

\begin{lemma}\label{lem:NonSph}
For each $i \in \{0, \dots, p-1\}$ there exists an element $z \in Y^\theta_{p+1}(p+1)$ of the form $z = (S_{p+1},\hat{\ell}_{S_{p+1}}; (\phi_0,\nu_0), \dots,
  (\phi_{p+1},\nu_{p+1}))$, such that the two elements
  \begin{align*}
    d_id_{p+1}(z) &= (S_{p+1},\hat{\ell}_{S_{p+1}}; (\phi_0,\nu_0) \ldots,
    \widehat{(\phi_i,\nu_i)} \ldots, (\phi_{p},\nu_{p}))\\
    d_{i+1}d_{p+1}(z) &= (S_{p+1},\hat{\ell}_{S_{p+1}}; (\phi_0,\nu_0) \ldots,
    \widehat{(\phi_{i+1},\nu_{i+1})} \ldots, (\phi_{p},\nu_{p}))
  \end{align*}  
  are connected by a path $\gamma: I \to Y_{p-1}^\theta(p+1)$, with
  the additional property that the loop $\epsilon \circ \gamma:
  I \to Y^\theta_{p-1}(p+1) \to Y^\theta_{-1}(p+1)$ is such that under each map
  $$\{M, \hat{\ell}_M\} \times Y^\theta_{-1}(p+1) \lra (\mathcal{M}^{\mathrm{st},\theta}(\hat{\ell}_P^{(p+1)}))_{g-p-1} \times Y^\theta_{-1}(p+1) \overset{h}\lra (\mathcal{M}^{\mathrm{st},\theta}(\hat{\ell}_P))_{g}$$
  it is nullhomologous as long as $g \geq 5$.
\end{lemma}
\begin{proof}
Recall from the proof of Lemma \ref{lem:proof:2} that the canonical embedding $\phi' : W_{1,1} \hookrightarrow S$ induces embeddings $\phi_j' : W_{1,1} \hookrightarrow S_{p+1}$ for $j=0, 1, \ldots, p+1$. In particular we have the three disjoint embeddings $\phi'_i, \phi'_{i+1}, \phi'_{p+1} : W_{1,1} \hookrightarrow S_{p+1}$, whose images can be thickened and joined by thickened arcs to obtain a submanifold $C_i \subset S_{p+1}$ diffeomorphic to $W_{3,1}$ and disjoint from the remaining $\phi'_j(W_{1,1})$. Applying Lemma \ref{lem:BraidMove} to the standard $\theta$-structure $\hat{\ell}_{S_p}\vert_{C_i}$ and the embeddings $\phi'_i, \phi'_{i+1}, \phi'_{p+1} : W_{1,1} \hookrightarrow C_i \cong W_{3,1}$ we obtain a diffeomorphism $\rho_i : S_{p+1} \to S_{p+1}$ supported in $C_i$, satisfying $\rho_i \circ \phi'_i = \phi'_{p+1}$, $\rho_i \circ \phi'_{i+1} = \phi'_i$, and $\rho_i \circ \phi'_{p+1} = \phi'_{i+1}$.

For any choice of $z \in Y^\theta_{p+1}(p+1)$, the construction described in the proof of Lemma \ref{lem:proof:3b} applied to the diffeomorphism $\rho_i$ gives a path starting at $d_i d_{p+1}(z)$ and ending at
\begin{equation}\label{eq:NonSph:1}
(S_{p+1},(\rho_i^{-1})^*\hat{\ell}_{S_{p+1}};(\rho_i \circ \phi_0,\nu_0), \dots,
    \widehat{(\rho_i \circ \phi_{i}, \nu_{i})}, \dots,
    (\rho_i \circ \phi_p,\nu_p)).
\end{equation}
By Lemma \ref{lem:BraidMove} (\ref{it:Braid:3}), there is a homotopy 
$\rho_i^*\hat{\ell}_{S_{p+1}}\vert_{C_i} \simeq \hat{\ell}_{S_{p+1}}\vert_{C_i}$ relative to $\partial C_i$. Gluing in ($(\rho_i^{-1})^*$ applied to) this homotopy shows that 
	there is a
  path of $\theta$-structures $k_t$, constant over the complement of
  $C_i$, starting at $\hat{\ell}_{S_{p+1}}$ and ending at
  $(\rho_i^{-1})^*\hat{\ell}_{S_{p+1}}$.  This gives rise to a path
  from \eqref{eq:NonSph:1} to
  \begin{equation}\label{eq:NonSph:2}
(S_{p+1},\hat{\ell}_{S_{p+1}};(\rho_i \circ \phi_0,\nu_0'), \dots,
    \widehat{(\rho_i \circ \phi_{i}, \nu_{i}')}, \dots,
    (\rho_i \circ \phi_p,\nu_p'))
  \end{equation}
  where $\nu_j' = ((\rho_i \circ \phi_{j})^* k_t) \cdot \nu_{j}$. We have therefore constructed a path in $Y^\theta_{p-1}(p+1)$ from $d_id_{p+1}(z)$ to \eqref{eq:NonSph:2}, covering a loop in $Y^\theta_{-1}(p+1)$.
  
We now make a particular choice of $z$.  Choose the $\phi_j$ for $j \nin \{i, i+1, p+1\}$ starting with the $\phi'_j : W_{1,1} \to S_{p+1}$ and extending by a thickening of a path to the point $c(0;3j,0,\ldots)$ in the coordinate patch $c$, ensuring that these paths are disjoint from each other and from $C_i$. For $j \not \in\{i,i+1, p+1\}$, since the supports of $\rho_i$ and $k_t$ are disjoint from the image of $\phi_j$, we have $\rho_i \circ \phi_j = \phi_j$, and $\nu_j'$ is the concatenation of $\nu_j$ with a constant path, so is homotopic to $\nu_j$. We then extend $\phi'_{i+1}$ to $\phi_{i+1}$ and $\phi'_{p+1}$ to $\phi_{p+1}$, disjointly from each other and the previously chosen $\phi_j$. Finally, we choose an extension $\phi_i$ of $\phi'_i$ so that it is isotopic to $\rho_i \circ \phi_{i+1}$ by an isotopy $g_t$ disjoint from the remaining $\phi_j$. This gives a path from \eqref{eq:NonSph:2} to 
  \begin{equation*}
    (S_{p+1}, \hat{\ell}_{S_{p+1}}; (\phi_0,\nu_0), \dots,
    (\phi_{i-1},\nu_{i-1}), (\phi_i, \nu_i''),
    (\phi_{i+2}, \nu_{i+2}), \dots, (\phi_p,\nu_p)),
  \end{equation*}
  where $\nu_i'' = (g_t^* \hat{\ell}_{S_p}) \cdot \nu'_{i+1}$. If we choose $\nu_j$ arbitrarily for $j \neq i$, and set $\nu_i = (g_t^* \hat{\ell}_{S_{p+1}}) \cdot \nu_{i+1}'$, we have constructed a path in $Y^\theta_{p-1}(p+1)$ from $d_i d_{p+1}(z)$ to $d_{i+1}d_{p+1}(z)$.


The loop $\epsilon \circ \gamma:  I \to Y^\theta_{-1}(p+1)$ we have constructed is supported inside $W_{3,1} \cong C_i \subset S_{p+1}$ and here is given by the diffeomorphism $\rho_i : C_i \to C_i$ and the homotopy $\rho_i^*\hat{\ell}_{S_{p+1}}\vert_{C_i} \simeq \hat{\ell}_{S_{p+1}}\vert_{C_i}$ provided by Lemma \ref{lem:BraidMove} (\ref{it:Braid:3}). Lemma \ref{lem:BraidMove} (\ref{it:Braid:4}) shows that this data may be chosen so that this loop becomes nullhomologous when included into any $\theta$-manifold having two additional copies of $W_{1,1}$ with standard $\theta$-structure. This establishes the additional property of this loop.
\end{proof}

\begin{proof}[Proof of Theorem~\ref{thm:main-theta}]
We will follow a similar strategy to the proof of Theorem \ref{thm:Main:sec6}, but using the augmented semisimplicial space $\epsilon^\theta : X^\theta_\bullet \to
  \mathcal{M}^{\mathrm{st},\theta}(\hat{\ell}_P) = X^\theta_{-1}$.

We first treat case (\ref{it:main-theta:1}), where $\theta$ is spherical.  
For an element $y \in Y^\theta_p(p)$ provided by Lemma~\ref{lem:A:proof:3b}, the
  argument of Lemma~\ref{lem:proof:3a} shows that the two compositions
  \begin{equation}\label{eq:SphericalSetup}
    \xymatrix{
      {\mathcal{M}^{\mathrm{st},\theta}(\hat{\ell}_P^{(p)})} \ar[r]^-{h(-, y)}_-\simeq& X^\theta_p 
      \ar@/^/[r]^{d_i} \ar@/_/[r]_{d_{i+1}} 
      & X^\theta_{p-1}
    }
  \end{equation}
  are homotopic.  
	We may then conclude that the graded augmented semisimplicial  space
  $X^\theta_\bullet \to \mathcal{M}^{\mathrm{st},\theta}(\hat{\ell}_P)$ satisfies the same properties as
  established for $X_\bullet \to \Mst$ in
  Proposition~\ref{prop:MainProperties} and therefore induces a
  spectral sequence satisfying the same formal properties as the
  spectral sequence we used in the proof of
  Theorem~\ref{thm:Main:sec6}.  The proof given there applies here
  word for word.

Let us now treat case (\ref{it:main-theta:2}), so suppose that we are given an abelian local coefficient system $\mathcal{L}$ on $\mathcal{M}^{\mathrm{st},\theta}(\hat{\ell}_P)$. This may be pulled back via the augmentation to a
  coefficient system on each $X^\theta_p$, which by abuse of notation
  we also call $\mathcal{L}$. There is a trigraded augmented spectral sequence
  \begin{equation*}
    E^1_{p,q, g} = H_q(X_p^{\theta};\mathcal{L})_g \quad \Longrightarrow \quad
    H_{p+q+1}(X_{-1}^{\theta}, |X_\bullet^{\theta}| ; \mathcal{L})_g
  \end{equation*}
  with differential $d^1 = \sum_i (-1)^i(d_i)_*$, and
  $E^\infty_{p,q,g} = 0$ if $p + q \leq (g-4)/2$. By Lemma \ref{lem:7B}, 
	for $y = (S, \hat{\ell}_S ; (\phi_0, \nu_0))$ in
  $Y^\theta_0(0)$ the map
  $$h(-,y) : \mathcal{M}^{\mathrm{st},\theta}(\hat{\ell}_P^{(0)}) \overset{\simeq}\lra X_0^\theta$$
  is a weak homotopy equivalence, which identifies the differential $d^1 : E^1_{0,*,g+1} \to E^1_{-1,*,g+1}$ with the map
    \begin{equation}\label{eq:16}
  s_* :  H_k(\mathcal{M}^{\mathrm{st},\theta}(\hat{\ell}'_P); s^* \mathcal{L})_g \lra H_k(\mathcal{M}^{\mathrm{st},\theta}(\hat{\ell}_P); \mathcal{L})_{g+1}
  \end{equation}
  we are studying (as we arranged that $\hat{\ell}_P' = \hat{\ell}_P^{(0)}$).
  
  Attempting the same induction argument as before, we may again
  conclude from the spectral sequence that the map~\eqref{eq:16} is
  surjective for $k=0$ and $g \geq 1$, proving the stated theorem for
  $g \leq 3$. We therefore suppose that $g\geq 4$, and so $g+1 \geq 5$.  Proceeding again by induction, we assume that theorem
  is proved up to $g-1$.  As before, we wish to identify the
  differential
  \begin{equation}\label{eq:4}
    d^1 = \sum_{i=0}^{2j} (-1)^i (d_i)_*: E^1_{2j,q,g+1} \lra E^1_{2j-1,q,g+1}
  \end{equation}
  with a previously determined stabilisation map 
 \begin{equation}\label{eq:99}
 s_*:H_q(\mathcal{M}^{\mathrm{st},\theta}(\hat{\ell}_P^{(2j)});s^*\mathcal{L})_{g-2j} \lra H_q(\mathcal{M}^{\mathrm{st},\theta}(\hat{\ell}_P^{(2j-1)}); \mathcal{L})_{g-2j+1}
 \end{equation} 
for $j > 0$.  As before this
  will follow if we can show that all of the face maps $(d_i)_*:
  H_q(X_{2j}^\theta;\mathcal{L})_g \to H_q(X_{2j-1}^\theta;\mathcal{L})_g$ are equal, since
  all but one term then cancel, and the remaining term is identified
  with $s_*$ by Lemma \ref{lem:7C}.

Before continuing, let us discuss what would go wrong if we just repeat the earlier argument for the $d_i$ to induce equal maps, but carry the coefficient systems along.
A homotopy between the two compositions in \eqref{eq:SphericalSetup} gives a self-homotopy of the map
$$\mathcal{M}^{\mathrm{st},\theta}(\hat{\ell}_P^{(p)}) \overset{h(-,y)}\lra  X^\theta_p \overset{\epsilon}\lra X^\theta_{-1},$$
and so a map
$H:(I/\partial I) \times \mathcal{M}^{\mathrm{st},\theta}(\hat{\ell}_P^{(p)}) \to
X_{-1}^{\theta} = \mathcal{M}^{\mathrm{st},\theta}(\hat{\ell}_P)$.
For each $x = (W,\hat{\ell}_W) \in \mathcal{M}^{\mathrm{st},\theta}(\hat{\ell}_P^{(p)})$, the loop
$s \mapsto H(s,x) \in \mathcal{M}^{\mathrm{st},\theta}(\hat{\ell}_P)$ has a potentially
non-trivial monodromy in the coefficient system $\mathcal{L}$.  These
monodromies assemble to an automorphism of the coefficient system
$h(-,y)^* \mathcal{L}$ on $\mathcal{M}^{\mathrm{st},\theta}(\hat{\ell}_P^{(p)})$, and hence
(since $h(-,y)$ is a weak homotopy equivalence) an automorphism of the
coefficient system $\mathcal{L}$ on $X_p^\theta$.  If we denote the
induced automorphism of $H_*(X_p^{\theta};\mathcal{L})$ by $\eta_*$,
the correct consequence of the homotopy between $d_i$ and $d_{i+1}$ is the
equation
  \begin{equation*}
    (d_i)_* = (d_{i+1})_* \circ \eta_*: H_*(X_p^\theta;\mathcal{L}) \lra
    H_*(X_{p-1}^\theta;\mathcal{L}).
  \end{equation*}
  If $\eta_*$ acts non-trivially on $H_*(X_p^\theta;\mathcal{L})$, the terms in
  the sum~\eqref{eq:4} no longer cancel out in pairs.

To deal with this issue, and simultaneously with non-spherical tangential structures, we shall replace the diagram~\eqref{eq:SphericalSetup} by 
  \begin{equation*}
    \xymatrix{
      {\mathcal{M}^{\mathrm{st},\theta}(\hat{\ell}_P^{(p+1)})} \ar[r]^-{h(-, z)}_-\simeq&
      X_{p+1}^{\theta} \ar[r]^{d_{p+1}} &
      X_{p}^{\theta} \ar@/^/[r]^{d_i} \ar@/_/[r]_{d_{i+1}} 
      & X_{p-1}^{\theta}
      \ar[r]^-\epsilon & X_{-1}^{\theta},
    }
  \end{equation*}
  where $z \in Y_{p+1}^\theta(p+1)$ is an element provided by
  Lemma~\ref{lem:NonSph}.  The path $\gamma: I \to
  Y^\theta_{p-1}(p+1)$ provided by the lemma then gives a homotopy
  $H': I \times \mathcal{M}^{\mathrm{st},\theta}(\hat{\ell}_P^{(p+1)}) \to X^\theta_{p-1}$ and in turn $\epsilon
  \circ H': (I/\partial I) \times \mathcal{M}^{\mathrm{st},\theta}(\hat{\ell}_P^{(p+1)}) \to X^\theta_{-1} = \mathcal{M}^{\mathrm{st},\theta}(\hat{\ell}_P)$.  If
  $\eta'_*$ denotes the automorphism of $H_*(X^\theta_{p+1};\mathcal{L})$
  induced by monodromy along the loops $s \mapsto H'(s,x)$, we now
  have
  \begin{equation*}
    (d_id_{p+1})_* = (d_{i+1}d_{p+1})_* \circ \eta'_*: H_*(X_{p+1}^\theta;\mathcal{L}) \lra H_*(X_{p-1}^\theta;\mathcal{L}).
  \end{equation*}
  By the choice made in Lemma~\ref{lem:NonSph}, the loops $s \mapsto \epsilon \circ H'(s,x) = h(x,\epsilon \circ \gamma(s))$ are all nullhomologous, as long as $h(x,\epsilon \circ \gamma(s))$ lies in grading at least $5$. Thus when $\mathcal{L}$ has trivial
  monodromy along nullhomologous loops, the automorphism $\eta'_*$ is
  the identity.  It follows that the two maps $(d_i)_*, (d_{i+1})_*:
  H_*(X^\theta_{2j};\mathcal{L}) \to H_*(X^\theta_{2j-1};\mathcal{L})$ do agree when
  restricted to the image of
  \begin{equation}\label{eq:PreStab}
(d_{p+1})_* : H_q(X_{2j+1}^\theta ; \mathcal{L})_{g+1} \lra H_q(X_{2j}^\theta ; \mathcal{L})_{g+1},
\end{equation}
and $g+1 \geq 5$. We have already established the theorem for $g \leq 3$, so need not worry about this restriction. By Lemma \ref{lem:7C} 
the map \eqref{eq:PreStab} is identified with $s_* : H_q(\mathcal{M}^{\mathrm{st},\theta}(\hat{\ell}_P^{(2j+1)}) ; s^*\mathcal{L})_{g-2j-1} \to H_q(\mathcal{M}^{\mathrm{st},\theta}(\hat{\ell}_P^{(2j)}) ; \mathcal{L})_{g-2j}$ which by inductive hypothesis is an epimorphism for $q \leq \tfrac{g-2j-2}{3}$. Hence \eqref{eq:4} and \eqref{eq:99} agree in degrees $q \leq \tfrac{g-2j-2}{3}$. Using the inductive hypothesis again, it follows that \eqref{eq:4} is an epimorphism for $q \leq \tfrac{g-2j-2}{3}$ and an isomorphism for $q \leq \tfrac{g-2j-4}{3}$. This means in particular that $E^2_{p,q,g+1}=0$ for $p > 0$ and $q \leq \tfrac{g-2p-1}{3}$, from which it follows that $d^1 : E^1_{0,q,g+1} \to E^1_{-1,q,g+1}$ is an epimorphism for $q \leq \tfrac{g-1}{3}$ and an isomorphism for $q \leq \tfrac{g-4}{3}$. This provides the inductive step.
\end{proof}


\bibliographystyle{amsalpha}
\bibliography{biblio}

\def\cprime{$'$}
\providecommand{\bysame}{\leavevmode\hbox to3em{\hrulefill}\thinspace}
\providecommand{\MR}{\relax\ifhmode\unskip\space\fi MR }
\providecommand{\MRhref}[2]{%
  \href{http://www.ams.org/mathscinet-getitem?mr=#1}{#2}
}
\providecommand{\href}[2]{#2}
\begin{thebibliography}{GRW16b}

\bibitem[AH61]{AtiyahHirzebruch}
M.~F. Atiyah and F.~Hirzebruch, \emph{Bott periodicity and the
  parallelizability of the spheres}, Proc. Cambridge Philos. Soc. \textbf{57}
  (1961), 223--226.

\bibitem[Bak69]{Bak}
Anthony Bak, \emph{On modules with quadratic forms}, Algebraic {K}-{T}heory and
  its {G}eometric {A}pplications ({C}onf., {H}ull, 1969), Springer, Berlin,
  1969, pp.~55--66.

\bibitem[Bak81]{Bak2}
\bysame, \emph{{$K$}-theory of forms}, Annals of Mathematics Studies, vol.~98,
  Princeton University Press, Princeton, N.J., 1981.

\bibitem[BF81]{MR613004}
E.~Binz and H.~R. Fischer, \emph{The manifold of embeddings of a closed
  manifold}, Differential geometric methods in mathematical physics ({P}roc.
  {I}nternat. {C}onf., {T}ech. {U}niv. {C}lausthal, {C}lausthal-{Z}ellerfeld,
  1978), Lecture Notes in Phys., vol. 139, Springer, Berlin, 1981, With an
  appendix by P. Michor, pp.~310--329.

\bibitem[BM13]{BerglundMadsen}
Alexander Berglund and Ib~Madsen, \emph{Homological stability of diffeomorphism
  groups}, Pure and Appl. Math. Quart. \textbf{9} (2013), 1--48.

\bibitem[Bol12]{Boldsen}
S{\o}ren Boldsen, \emph{Improved homological stability for the mapping class
  group with integral or twisted coefficients}, Mathematische Zeitschrift
  \textbf{270} (2012), 297--329.

\bibitem[Cer61]{Cerf}
Jean Cerf, \emph{Topologie de certains espaces de plongements}, Bull. Soc.
  Math. France \textbf{89} (1961), 227--380.

\bibitem[Cha87]{Charney}
Ruth Charney, \emph{A generalization of a theorem of {V}ogtmann}, Proceedings
  of the {N}orthwestern conference on cohomology of groups ({E}vanston, {I}ll.,
  1985), vol.~44, 1987, pp.~107--125.

\bibitem[GRW12]{OldVersion}
S{\o}ren Galatius and Oscar Randal-Williams, \emph{Homological stability for
  moduli spaces of high dimensional manifolds}, Preprint, arXiv:1203.6830v2,
  2012.

\bibitem[GRW14a]{GR-WDetect}
\bysame, \emph{Detecting and realising characteristic classes of manifold
  bundles}, Algebraic topology: applications and new directions, Contemp.
  Math., vol. 620, Amer. Math. Soc., Providence, RI, 2014, pp.~99--110.

\bibitem[GRW14b]{GR-W2}
\bysame, \emph{Stable moduli spaces of high-dimensional manifolds}, Acta Math.
  \textbf{212} (2014), no.~2, 257--377.

\bibitem[GRW16a]{GR-WAb}
\bysame, \emph{Abelian quotients of mapping class groups of highly connected
  manifolds}, Math. Ann. \textbf{365} (2016), no.~1-2, 857--879.

\bibitem[GRW16b]{HomStabII}
\bysame, \emph{Homological stability for moduli spaces of high dimensional
  manifolds. {II}}, Annals of Mathematics, to appear. arXiv:1601.00232, 2016.

\bibitem[Har85]{H}
John~L. Harer, \emph{Stability of the homology of the mapping class groups of
  orientable surfaces}, Ann. of Math. (2) \textbf{121} (1985), no.~2, 215--249.

\bibitem[Hir59]{Hirsch}
Morris~W. Hirsch, \emph{Immersions of manifolds}, Trans. Amer. Math. Soc.
  \textbf{93} (1959), 242--276.

\bibitem[HV15]{HV}
Allen Hatcher and Karen Vogtmann, \emph{Tethers and homology stability for
  surfaces}, arXiv:1508.04334, 2015.

\bibitem[HW10]{HW}
Allen Hatcher and Nathalie Wahl, \emph{Stabilization for mapping class groups
  of 3-manifolds}, Duke Math. J. \textbf{155} (2010), no.~2, 205--269.

\bibitem[Iva93]{Ivanov}
Nikolai~V. Ivanov, \emph{On the homology stability for {T}eichm\"uller modular
  groups: closed surfaces and twisted coefficients}, Mapping class groups and
  moduli spaces of {R}iemann surfaces ({G}\"ottingen, 1991/{S}eattle, {WA},
  1991), Contemp. Math., vol. 150, Amer. Math. Soc., Providence, RI, 1993,
  pp.~149--194.

\bibitem[JR78]{JonesRees}
John Jones and Elmer Rees, \emph{Kervaire's invariant for framed manifolds},
  Algebraic and geometric topology ({P}roc. {S}ympos. {P}ure {M}ath.,
  {S}tanford {U}niv., 1976), {P}art 1, Proc. Sympos. Pure Math., XXXII, Amer.
  Math. Soc., Providence, R.I., 1978, pp.~141--147.

\bibitem[Kre99]{Kreck}
Matthias Kreck, \emph{Surgery and duality}, Ann. of Math. (2) \textbf{149}
  (1999), no.~3, 707--754.

\bibitem[Maa79]{Maazen}
Hendrik Maazen, \emph{Homology stability for the general linear group}, Utrecht
  thesis, 1979.

\bibitem[MS06]{MaySigurdsson}
J.~P. May and J.~Sigurdsson, \emph{Parametrized homotopy theory}, Mathematical
  Surveys and Monographs, vol. 132, American Mathematical Society, Providence,
  RI, 2006. \MR{2271789 (2007k:55012)}

\bibitem[Mum83]{Mumford}
David Mumford, \emph{Towards an enumerative geometry of the moduli space of
  curves}, Arithmetic and geometry, {V}ol. {II}, Progr. Math., vol.~36,
  Birkh\"auser Boston, Boston, MA, 1983, pp.~271--328.

\bibitem[MW07]{MW}
Ib~Madsen and Michael Weiss, \emph{The stable moduli space of {R}iemann
  surfaces: {M}umford's conjecture}, Ann. of Math. (2) \textbf{165} (2007),
  no.~3, 843--941.

\bibitem[RW13]{R-WPerfect}
Oscar Randal-Williams, \emph{`{G}roup-completion', local coefficient systems
  and perfection}, Q. J. Math. \textbf{64} (2013), no.~3, 795--803.

\bibitem[RW16]{R-WResolution}
\bysame, \emph{Resolutions of moduli spaces and homological stability}, J. Eur.
  Math. Soc. \textbf{18} (2016), no.~1, 1--81.

\bibitem[Spa66]{Spanier}
Edwin~H. Spanier, \emph{Algebraic topology}, McGraw-Hill Book Co., New
  York-Toronto, Ont.-London, 1966.

\bibitem[Wal62]{WallTrans}
C.~T.~C. Wall, \emph{On the orthogonal groups of unimodular quadratic forms},
  Math. Ann. \textbf{147} (1962), 328--338.

\bibitem[Wal70]{Wall}
\bysame, \emph{Surgery on compact manifolds}, Academic Press, London, 1970,
  London Mathematical Society Monographs, No. 1.

\bibitem[Wei05]{WeissCatClass}
Michael Weiss, \emph{What does the classifying space of a category classify?},
  Homology Homotopy Appl. \textbf{7} (2005), no.~1, 185--195.

\end{thebibliography}

\end{document}